%% file: main.tex
\documentclass[a4paper,oneside]{amsart}
\usepackage{geometry}

\usepackage{mathtext}
\usepackage[T1]{fontenc}
\usepackage{color}
\usepackage{amsmath}
\usepackage{amsfonts,amssymb,amsthm}
\usepackage[pdftex]{graphicx}
\usepackage{comment}
\usepackage{float}
\usepackage{todonotes}
\usepackage{microtype}
\usepackage{todonotes}

\usepackage[colorlinks=false, linktocpage=true]{hyperref}

\usepackage{enumerate}

\usepackage{cite}

\newcommand{\midwd}{\,\middle\vert\,}

\DeclareMathOperator{\Aut}{Aut}

\DeclareMathOperator{\diag}{diag}
\DeclareMathOperator{\Id}{Id}
\DeclareMathOperator{\Sym}{Sym}
\DeclareMathOperator{\Span}{Span}

\DeclareMathOperator{\Lag}{Lag}
\DeclareMathOperator{\Mat}{Mat}

\DeclareMathOperator{\Sp}{Sp}
\DeclareMathOperator{\GL}{GL}
\DeclareMathOperator{\SL}{SL}

\DeclareMathOperator{\OO}{O}
\DeclareMathOperator{\UU}{U}

\DeclareMathOperator{\SO}{SO}

\DeclareMathOperator{\Gr}{Gr}

\DeclareMathOperator{\Stab}{Stab}

\DeclareMathOperator{\oo}{\mathfrak{o}}

\DeclareMathOperator{\spp}{\mathfrak{sp}}
\DeclareMathOperator{\ksp}{\mathfrak{ksp}}
\DeclareMathOperator{\Cl}{Cl}

\DeclareMathOperator{\End}{End}
\DeclareMathOperator{\Lie}{Lie}
\DeclareMathOperator{\Spin}{Spin}

\DeclareMathOperator{\Imm}{Im}
\DeclareMathOperator{\Ree}{Re}
\DeclareMathOperator{\Fix}{Fix}
\DeclareMathOperator{\KSp}{KSp}
\DeclareMathOperator{\Herm}{Herm}
\DeclareMathOperator{\Skew}{Skew}

\DeclareMathOperator{\Is}{Is}
\DeclareMathOperator{\pol}{pol}
\DeclareMathOperator{\bpol}{\overline{pol}}
\DeclareMathOperator{\tr}{tr}
\DeclareMathOperator{\Isom}{Isom}
\DeclareMathOperator{\ch}{char}
\DeclareMathOperator{\Der}{Der}
\DeclareMathOperator{\ad}{ad}
\DeclareMathOperator{\Ad}{Ad}

\newcommand{\act}{\mathop{\triangleright}}
\newcommand{\ract}{\mathop{\triangleleft}}

\newcommand{\R}{\mathbb R}
\newcommand{\RR}{\mathbb R}
\newcommand{\CC}{\mathbb C}
\newcommand{\HH}{\mathbb H}
\newcommand{\PP}{\mathbb P}
\newcommand{\N}{\mathbb N}
\newcommand{\Z}{\mathbb Z}

\newcommand{\Q}{\mathbb Q}

\newcommand{\K}{\mathbb K}
\newcommand{\Oc}{\mathbb O}
\newcommand{\COme}[1]
{\begin{matrix}
0 & #1\\
-#1 & 0
\end{matrix}}
\newcommand{\MOme}[1]
{\begin{matrix}
0 & -#1\\
#1 & 0
\end{matrix}}
\newcommand{\Ome}[1]
{\begin{pmatrix}
0 & #1\\
-#1 & 0
\end{pmatrix}}
\newcommand{\mtrx}[4]
{\begin{pmatrix}
#1 & #2\\
#3 & #4
\end{pmatrix}}
\newcommand{\Om}{
\begin{pmatrix}
0 & 1\\
-1 & 0
\end{pmatrix}}

\newcommand{\Di}[1]
{\begin{pmatrix}
#1 & 0\\
0 & #1
\end{pmatrix}}
\newcommand{\Dia}[2]
{\begin{pmatrix}
#1 & 0\\
0 & #2
\end{pmatrix}}
\newcommand{\Clm}[2]
{\begin{pmatrix}
#1\\
#2
\end{pmatrix}}

\theoremstyle{plain}
\newtheorem{teo}{Theorem}[section]
\newtheorem{theorem}[teo]{Theorem}
\newtheorem{cor}[teo]{Corollary}
\newtheorem{corollary}[teo]{Corollary}
\newtheorem{lem}[teo]{Lemma}
\newtheorem{lemma}[teo]{Lemma}
\newtheorem{prop}[teo]{Proposition}
\newtheorem{proposition}[teo]{Proposition}
\newtheorem{fact}[teo]{Fact}

\newtheorem{question}[teo]{Question}

\theoremstyle{definition}
\newtheorem{df}[teo]{Definition}
\newtheorem{definition}[teo]{Definition}
\newtheorem{ex}[teo]{Example}
\newtheorem{example}[teo]{Example}

\theoremstyle{remark}
\newtheorem{rem}[teo]{Remark}
\newtheorem{remark}[teo]{Remark}

\newcommand{\bs}{\setminus}
\newcommand{\defin}{\emph}

\renewcommand{\emptyset}{\varnothing}

\title{Symplectic groups over noncommutative algebras}

\author[D. Alessandrini]{Daniele Alessandrini}
\address{Department of Mathematics, Columbia University, New York, USA}
\email{daniele.alessandrini@gmail.com}

\author{Arkady Berenstein}
\address{\noindent Department of Mathematics, University of Oregon,
Eugene, OR 97403, USA} \email{arkadiy@math.uoregon.edu}

\author{Vladimir Retakh}
\address{\noindent Department of Mathematics, Rutgers University,
Piscataway, NJ 08854, USA} \email{vretakh@math.rutgers.edu}

\author[E. Rogozinnikov]{Eugen Rogozinnikov} \address{Institut de Recherche Math\'ematique Avanc\'ee, Universit\'e de Strasbourg, Strasbourg, France}
\email{erogozinnikov@gmail.com}

\author[A. Wienhard]{Anna Wienhard}
\address{Mathematisches Institut, Ruprecht-Karls-Universität Heidelberg,
  Germany\hfill{}\ \linebreak
  HITS gGmbH, Heidelberg Institute for Theoretical Studies, Heidelberg,
  Germany} \email{wienhard@uni-heidelberg.de}

\begin{document}

\sloppy

\thanks{D.A, E.R and A.W. acknowledge support from U.S. National Science Foundation grants DMS 1107452, 1107263, 1107367 "RNMS: GEometric structures And Representation varieties" (the GEAR Network). E.R. and A.W. were supported by the National Science Foundation under Grant No. 1440140 and the Clay Foundation (A.W.), while they were in residence at the Mathematical Sciences Research Institute in Berkeley, California, during the Fall Semester 2019. A.W acknowledges funding by the Deutsche Forschungsgemeinschaft within the Priority Program SPP 2026
	“Geometry at Infinity”, by the European Research Council under
	ERC-Consolidator grant 614733, and by the Klaus-Tschira-Foundation. 
E.R. acknowledges funding by the Deutsche Forschungsgemeinschaft within the RTG 2229
``Asymptotic invariants and limits of groups and spaces'' and the Priority Program SPP 2026 ``Geometry at Infinity'', and thanks the Labex IRMIA of the Universit\'e de Strasbourg for support during the finishing of this project.
A.B. was partially supported by Simons Foundation Collaboration Grant  No. 636972. This work has been supported under Germany’s Excellence Strategy EXC-2181/1 - 390900948 (the Heidelberg STRUCTURES Cluster of Excellence).}

\begin{abstract}
We introduce the symplectic group $\Sp_2(A,\sigma)$ over a noncommutative algebra $A$ with an anti-involution $\sigma$.
We realize several classical Lie groups as $\Sp_2$ over various noncommutative algebras, which provides new insights into their structure theory.

We construct several geometric spaces, on which the groups $\Sp_2(A,\sigma)$ act. 
 We introduce the space of isotropic $A$-lines, which generalizes the projective line. We describe the action of $\Sp_2(A,\sigma)$ on isotropic $A$-lines, generalize the  Kashiwara-Maslov index of triples and the cross ratio of quadruples of isotropic $A$-lines as invariants of this action.
When the algebra $A$ is Hermitian or the complexification of a Hermitian algebra, we introduce the symmetric space $X_{\Sp_2(A,\sigma)}$, and construct different models of this space. 

Applying this to classical Hermitian Lie groups of tube type (realized as $\Sp_2(A,\sigma)$) and their complexifications, we obtain different models of the symmetric space as noncommutative generalizations of models of the hyperbolic plane and of the three-dimensional hyperbolic space. 

We also provide a partial classification of Hermitian algebras in Appendix A.
\end{abstract}

\maketitle

\tableofcontents

\section{Introduction}

The special linear groups $\SL_2(R)$, when $R$ is a commutative ring, are among the most important and best studied groups in mathematics. They arise in many different contexts: for example in number theory and arithmetic geometry and in the theory of finite simple groups. In the special case when $R = \R$ or $\CC$, they are ubiquitous in Lie theory and in representation theory. Via their geometric actions, they are fundamental objects in projective, conformal and hyperbolic geometry. 

The main aim of this paper is to generalize these groups to the case when $R$ is a noncommutative ring. Unfortunately, for a noncommutative ring, the definition of $\SL_2(R)$ is tricky, see \cite{BR_noncommutative1,BR_noncommutative2} and the comments below. 

We will thus slightly change our point of view, and notice that, for a commutative ring $R$, the special linear group $\SL_2(R)$ is isomorphic to the symplectic group $\Sp_2(R)$. In this article we develop the theory of symplectic groups $\Sp_2(A,\sigma)$ over possibly noncommutative algebras $A$ with an anti-involution $\sigma$.
%
%
We show that many aspects of the classical $\SL_2$ theory can be developed over such noncommutative algebras $(A,\sigma)$. By realizing several classical Lie groups as $\Sp_2$ over a noncommutative algebra, this also  provides new insights into the theory of some higher dimensional classical Lie groups. 
\begin{enumerate}
\item We define the group $\Sp_2(A,\sigma)$ and describe its action on a generalization of the projective line: the space of isotropic lines. We generalize the Kashiwara-Maslov index of triples and the  cross ratio of $4$-tuples of isotropic lines. 

\item We introduce Hermitian algebras $(A,\sigma)$ (see Definition~\ref{Herm_A}), a special class of $\R$-algebras exhibiting positivity properties. In this case, we construct different models of the symmetric space associated to $\Sp_2(A,\sigma)$ and its complexification $\Sp_2(A_ \mathbb{C},\sigma_\mathbb{C})$. We obtain generalizations of several models of the hyperbolic plane and the three-dimensional hyperbolic space, which are the symmetric spaces associated to $\SL(2,\R)$ and $\SL(2,\CC)$ respectively.   

\item We show that classical Hermitian Lie groups of tube type, such as the standard real symplectic group $\Sp(2n,\mathbb{R})$ can be naturally realized as $\Sp_2(A,\sigma)$. This explains many aspects of the structure theory of Hermitian Lie groups of tube type. 

\item As an application we describe new explicit models of the symmetric spaces associated to the complex Lie groups $\Sp(2n,\CC), \GL(2n,\CC)$, and  $\OO(4n,\CC)$.

\item A large part of the theory outlined in (2) is completely algebraic, and works even when $\R$ is replaced by any real closed field. This may be of interest for some applications. 
\end{enumerate}

In \cite{BR_noncommutative1,BR_noncommutative2} two of the authors started to develop the general theory of Lie groups and Lie algebras over noncommutative rings. This is a highly non-trivial theory. Some of the difficulties can be already be seen when trying to define the group $\SL_n$ over a noncommutative ring. It is immediate to define the group $\GL_n$ over a noncommutative ring $R$ as linear invertible maps from $R^n$ to $R^n$. But there is no appropriate definition for the group $\SL_n$ as a subgroup of $\GL_n$ because there is no canonical choice of a determinant, even though there is rich  theory of quasi-determinants \cite{GGRW}.

Our approach is motivated by the notion of $\Theta$-positivity, a generalization of Lusztig's total positivity in the context of real Lie groups which are not necessarily split, introduced in \cite{GW,GW_pos, WienhardICM}. Hermitian Lie groups of tube type admit a $\Theta$-positive structure.\footnote{In fact, for Hermitian Lie groups the $\Theta$-positive structure is related to the theory of Lie semigroups.}
The combinatorics of this $\Theta$-positive structure is governed by the Weyl group of a root system of type $A_1$, making Hermitian Lie groups of tube type ``look like'' $\SL_2$ theories over noncommutative algebras. Here we make this precise. 

Our description of the different  models of the symmetric spaces for most of the Hermitian groups of tube type shows that the geometry of these symmetric space is a noncommutative version of the classical planar hyperbolic geometry. This is not surprising: for example a well-known model of the symmetric space of $\Sp(2n,\R)$ is the Siegel upper half space, and the formulae describing this model look like the formulae of planar hyperbolic geometry, with matrices replacing numbers. The theory developed here puts this observation into a more general theoretical framework.  

Even more interesting is our description of different models of the  symmetric spaces of the complexifications of the groups discussed above, such as $\Sp(2n,\CC)$, $\GL(2n,\CC)$, $\OO(4n,\CC)$. Here we show that the geometry of their symmetric spaces is a noncommutative version of the classical $3$-dimensional hyperbolic geometry. This fact is more surprising.

We believe that this new point of view can be helpful and leads to new applications. An immediate application is that the construction of noncommutative coordinates for symplectic representations of fundamental groups of punctured surfaces developed in \cite{AGRW} and its relation to noncommutative cluster algebras introduced in \cite{BR_surf} generalizes to representations into classical Hermitian Lie groups of tube type. This will appear in a forthcoming work. 

We now describe the results of the paper in more detail.

\subsection{Symplectic groups over involutive algebras} 
Let $\K$ be a field and consider an associative, possibly noncommutative unital finite-dimensional $\K$-algebra $A$. Let $\sigma: A \to A $ be an anti-involution, i.e.\ a $\K$-linear map with $\sigma(ab) = \sigma(b)\sigma(a)$ for all $a,b \in A$ and $\sigma^2 = \Id$. We denote the set of fixed points of $\sigma$ by $A^\sigma$.
Then we introduce the non-degenerate sesquilinear form $\omega: A^2 \times A^2 \to A$, defined by 
\[\omega(x,y):= \sigma(x)^T \Omega y\,,\] 
for all $x,y \in A^2$ and $\Omega =\Ome{1}$. 
The symplectic group $\Sp_2(A, \sigma)$ is defined as 
\[\Sp_2(A, \sigma) = \Aut(\omega)= \{f\in\Aut(V)\mid \forall x,y\in V:\omega(f(x),f(y))=\omega(x,y)\}\,.\] 

Similar to the classical case, we look at natural spaces, on which the group $\Sp_2(A, \sigma)$ acts. The first such space is our noncommutative analog of the projective line. The group does not act transitively on the space of all lines in $A^2$, hence we restrict our attention to the lines that are isotropic for $\omega$:
\[\PP(\Is(\omega)) :=\{xA\mid \omega(x,x) = 0, \,  x\in A^2\text{ regular}\}\,,\]
where an element $x\in A^2$ is called regular if there exists $y \in A^2$ such that $x,y$ form a basis. Two isotropic lines $xA$ and $yA$ are said to be transverse if $x,y$ form a basis.

The action of $\Sp_2(A, \sigma)$ on $\PP(\Is(\omega))$ has some of the features we know from projective or M\"obius geometry, e.g. the action on quadruples of points preserve an invariant that generalizes the cross-ratio:   

\begin{theorem}\label{thm_intro:isotropic}
\begin{enumerate}
    \item The group $\Sp_2(A,\sigma)$ acts transitively on $\PP(\Is(\omega))$
    \item The group $\Sp_2(A,\sigma)$ acts transitively on the set of pairs of transverse isotropic lines. 
\item The $\Sp_2(A,\sigma)$-orbits in the space of pairwise transverse triples of isotropic lines are in one-to-one correspondence with the orbits of the action of the group of invertible elements $A^\times$ on the set of invertible fixed points $(A^\sigma)^\times$ by 
$$\begin{matrix}
A^\times\times (A^\sigma)^\times & \mapsto & (A^\sigma)^\times\\
 (a,b) & \mapsto & ab\sigma(a).
\end{matrix}$$ 
This gives rise to a generalization of the Kashiwara-Maslov index. 
\item The 
$\Sp_2(A,\sigma)$-orbits in the space of pairwise transverse 4-tuples  of isotropic lines are in one-to-one correspondence with the orbits of the action of the group of invertible elements $A^\times$ on the set $A_0$ of products of invertible fixed points $(A^\sigma)^\times$, $A_0 =\{ bb'\, \mid \, b, b' \in (A^\sigma)^\times\}$: 
$$\begin{matrix}
A^\times\times A_0 & \mapsto & A_0\\
(a,b) & \mapsto & aba^{-1}.
\end{matrix}$$ The conjugacy class in $A_0$ is thus a  noncommutative cross-ratio of a four-tuple of pairwise transverse isotropic lines. 
\end{enumerate}
\end{theorem}

Of particular interest to us is a special class of involutive algebras that we call Hermitian algebras. We describe the geometry of the symmetric spaces for $\Sp_2(A,\sigma)$, when $(A,\sigma)$ is Hermitian or the complexification of an Hermitian algebra. 

Hermitian algebras are  algebras $(A,\sigma)$ over a real closed field $\K$  where the following holds: if $a,b\in A^\sigma$, then $a^2 + b^2 = 0$ if and only if $a = b = 0$. Readers not familiar with real closed fields can simply think about the case where $\K=\R$, since this is the most interesting case for our applications.
Involutive algebras are closely related to Jordan algebras, and Hermitian algebras give rise to formally real Jordan algebras see Section~\ref{subsect_spectral}. A key feature of Hermitian algebras is the existence of a proper convex cone $A^\sigma_+\subset A^\sigma$.

In the main body of the paper, we focus on semisimple Hermitian algebras. So, for the rest of the introduction,  whenever we say Hermitian algebra, we mean semisimple Hermitian algebra. In Appendix~\ref{app:classification} we consider a more general class of rings with anti-involution, which we call pre-Hermitian, investigate the theory of general Hermitian algebras, and classify semisimple Hermitian algebras.

There are different ways in which we can construct new Hermitian algebras out of old ones: Matrix algebras over Hermitian algebras are Hermitian, and complexifications (and quaternionifications) of Hermitian algebras provide new Hermitian algebras. 
Given a Hermitian algebra $(A,\sigma)$ we consider the complexification $A_\CC$ of $A$ and extend the involution complex linearly to an involution $\sigma_\CC$ and complex anti-linearly to an involution $\bar\sigma_\CC$. Then $(A_\CC, \bar\sigma_\CC)$ is an Hermitian algebra, but $(A_\CC, \sigma_\CC)$ is not, see Section~\ref{A_case}. 

For Hermitian algebras and their complexifications we can refine Theorem~\ref{thm_intro:isotropic}. 
For this we make the following definition: we say that a triple of pairwise transverse isotropic lines is positive, if its $\Sp_2(A,\sigma)$-orbits corresponds to an element in $A^\sigma_+$ by the identification given in Theorem~\ref{thm_intro:isotropic}.
\begin{theorem}
If $(A,\sigma)$ is Hermitian, then $\Sp_2(A,\sigma)$ acts transitively on the space of positive triples of pairwise transverse isotropic lines. 
\end{theorem}

\begin{theorem}
Assume $\K=\R$ and let $(A,\sigma)$ be a Hermitian algebra or its complexification. Then the space of isotropic lines is compact. 
\end{theorem}


\subsection{Symmetric spaces associated to \texorpdfstring{$\Sp_2(A,\sigma)$}{Sp2(A,sigma)}}


Recall that the symmetric space associated to $\Sp_2(\R)$ is $X_{\Sp_2(\R)} =\Sp_2(\R)/ \UU_1(\CC)$, where $\UU_1(\CC)$ is a maximal compact subgroup. It admits many explicit models. It can be realized as the space of compatible complex structures on $\R^2$, this is the space 
$\mathfrak C:=\left\{J\text{ complex structure on $\R^2$}\mid \omega(J \cdot, \cdot) \text{ is an inner product}
\right\}.$ We call this the complex structure model.
To obtain the upper hemisphere model, we complexify $\R^2$ to $\CC^2$ and extend the symplectic form $\omega$ complex linearly to a symplectic from $\omega_\CC$. Similarly we extend a complex structure $J$ complex linearly to $\CC^2$. Over $\CC$ now the complex structure is a diagonalizable endomorphism and we can associate to every compatible complex structure its $+i$-eigenspace in $\CC^2$. This provides an embedding of $\mathfrak C$ into the space of isotropic lines $\PP(\Is(\omega_\CC)) = \PP(\CC^2)$, whose image we denote by $\mathfrak P$ and call the projective model. 

The Poincar\'e disk model $D = \{ z\in \CC\mid \overline{z}z <1\}$ and the upper half plane model $\mathfrak U = \{ z\in \CC \mid \Imm(z) >0\}$ arise from projective model $\mathfrak{P}$ naturally by picking specific affine charts in $\PP(\CC^2)$. We call the disk model the precompact model. 
Taking its closure in $\CC$, or the closure of $\mathfrak P$ in $\PP(\CC^2)$, we obtain a compactification of the symmetric space, in which the space of (isotropic) lines in $\R^2$ arises as the boundary. 

We show that these constructions can be appropriately generalized to symplectic groups $\Sp_2(A,\sigma)$ over real Hermitian algebras $(A,\sigma)$.
In order to introduce the maximal compact subgroup we observe that 
given $(A,\sigma)$ the algebra $\Mat_n(A)$ of $n\times n$-matrices over $A$ can be endowed with the anti-involution $\sigma^T$, which applies $\sigma$ to each entry and then the transpose. When $(A,\sigma)$ is a Hermitian algebra, then $(\Mat_n(A), \sigma^T)$ is Hermitian as well. In that case the subgroup $\UU_n(A,\sigma) = \{ M \in \Mat_n(A) \, \mid \, \sigma(M)^T M = \Id_n\}$ is compact.  
This allows us to define the maximal compact subgroups 
$\KSp_2(A,\sigma) = \Sp_2(A,\sigma)\cap \UU_2(A,\sigma)  \subset \Sp_2(A,\sigma)$, and $\KSp^c_2(A_\CC, \sigma_\CC)= \Sp_2(A_\CC,\sigma_\CC)\cap \UU_2(A_\CC,\bar\sigma_\CC)  \subset \Sp_2(A_\CC,\sigma_\CC)$. 
We can thus consider the symmetric space associated to $\Sp_2(A,\sigma)$, 
$X_{\Sp_2 (A,\sigma)}=  \Sp_2(A, \sigma)/ \KSp_2(A, \sigma)$, and the symmetric space 
$X_{\Sp_2 (A_\CC,\sigma_\CC)}=  \Sp_2(A_\CC, \sigma_\CC)/ \KSp^c_2(A_\CC, \sigma_\CC)$. 
We develop different explicit models for the symmetric spaces 
$X_{\Sp_2 (A,\sigma)}$ and $X_{\Sp_2 (A_\CC,\sigma_\CC)}$.



\begin{theorem}\label{thm_intro:symmetric}
Let $(A,\sigma)$ be a real Hermitian algebra and $(A_\CC, \sigma_\CC, \bar\sigma_\CC)$ its complexification. Then the symmetric space $X_{\Sp_2 (A,\sigma)}=  \Sp_2(A, \sigma)/ \KSp_2(A, \sigma)$ admits 
\begin{enumerate}
\item a complex structure model \\$\mathfrak C:=\left\{J\text{ complex structure on $A^2$}\mid \omega(J\cdot,\cdot)\text{ is a $\sigma$-inner product}
\right\}.$
\item a projective model $\mathfrak P:=\{vA_\CC\mid v\in\Is(\omega_\CC),\; i \omega_\CC(\overline{v},v)\in (A_\CC^{\bar\sigma})_+\},$ where $(A_\CC^{\bar\sigma})_+$ is the proper convex cone in $(A_\CC^{\bar\sigma})$. 
\item a precompact model $\mathring{D}(A^{\sigma_\CC}_\CC,\bar\sigma_\CC):=\{c\in A^{\sigma_\CC}_\CC\mid 1-\bar cc\in (A^{\bar\sigma_\CC}_\CC)_+\}$. 
\item an upper half-space model $\mathfrak U:=\{z\in A_\CC^{\sigma_\CC}\mid \Imm(z)\in A^\sigma_+\}$, where $A^\sigma_+$ is the proper convex cone in $A^\sigma$.
\end{enumerate}
Furthermore, there are natural maps between these different models.
The closure of the precompact model gives rise to a compactification of $X_{\Sp_2 (A,\sigma)}$ in which the space of isotropic lines $\PP(\Is(\omega))$ appears as the closed $\Sp_2(A,\sigma)$-orbit. 
\end{theorem}


For the group $\Sp_2(\CC)$ there is a similar construction of explicit models of the symmetric space, which is less well known. 
The symplectic group $\Sp_2(\CC)$ acts on $(\CC^2, \omega)$, which we view as the complexification of $\R^2$. 
A quaternionic structure on $\CC^2$ is an additive map $J: \CC^2 \to \CC^2$ such that $J^2=-\Id$ and such that  $J(xa)=J(x)\bar a$ for all $x\in 
\CC^2$, $a\in \CC$. 
 Then $\Sp_2(\CC)$ acts naturally on the space of compatible quaternionic structures on $\CC^2$, which is 
 $\mathfrak C:=\{J\text{ quaternionic structure on $\CC^2$}\mid \omega( J \cdot, \cdot) \text{ is a Hermitian inner product}\}.$ We call $\mathfrak C$ the quaternionic structure model of $X_{\Sp_2(\CC)}$. 
Analogously to the above construction, where we complexified $\R^2$ we can now quaternionify $\CC^2$ and extend the symplectic form $\omega$ as well as the quaternionic structure. One has to be a bit more cautious working with the quaternions $\HH$, but 
one gets a projective model $\mathfrak P \subset \PP(\HH^2)$, a precompact model $D = \{x+yi+zj \in \HH\mid 1-(x^2+y^2+z^2)>0\}  \subset \HH$ and a upper half space model $\mathfrak U =\{x+iy+kj\in \HH\mid x,y,z\in\R,\;z>0\}  \subset \HH$. 
These quaternionic models for the three-dimensional hyperbolic space are for example described in \cite{Abikoff, Quinn}. 

We prove that an analogous construction can be made for the complexified symplectic groups $\Sp_2(A_\CC, \sigma_\CC)$, using quaternionic extensions $A_\HH$ of the Hermitian algebra $A$ with two appropriate quaternionic extensions $\sigma_0$ and $\sigma_1$ of the anti-involution $\sigma$. In particular, we obtain 

\begin{theorem}\label{thm_intro:complexsymmetric}
Let $(A_\CC,\sigma_\CC)$ be the complexification of a real Hermitian algebra $(A, \sigma)$. Then the symmetric space $X_{\Sp_2 (A_\CC,\sigma_\CC)}=  \Sp_2(A_\CC, \sigma_\CC)/ \KSp^c_2(A_\CC, \sigma_\CC)$ admits 
\begin{enumerate}
\item a quaternionic structure model\\
$\mathfrak C:=\{J\text{ quaternionic structure on $A_\CC^2$}\mid \omega( J \cdot, \cdot) \text{ is a 
4
 $\bar\sigma$-inner product}\}$
\item a projective model $\mathfrak P\subset \Is(\omega_\HH)$, 
\item a precompact model $\mathring{D}(A_\HH^{\sigma_0},\sigma_1):=\{c\in A^{\sigma_0}_\HH\mid 1-\sigma_1(c)c\in (A^{\sigma_1}_\HH)_+\}$, where $(A^{\sigma_1}_\HH)_+$ is the proper convex cone in the quaternionification of $A$.
\item an upper half-space model $\mathfrak U:=\{z_0+z_1j\in A_\HH^{\sigma_0}\mid z_0\in A^{\sigma_\CC}_{\CC},\;z_1\in (A_\CC^{\bar\sigma})_+\}$. 
\end{enumerate}
Furthermore, there are natural maps between these different models.
The closure of the precompact model gives rise to a compactification of $X_{\Sp_2(A_\CC,\sigma_\CC)}$ in which the space of isotropic lines $\PP(\Is(\omega_\CC))$ appears as the closed $\Sp_2(A_\CC,\sigma_\CC)$-orbit. 
\end{theorem}

\subsection{Hermitian Lie groups of tube type} 

As we mentioned before, Hermitian Lie groups of tube type look a lot like groups of type $\Sp_2$ over real involutive noncommutative algebras. We illustrate it on the following example:

We consider the $\R$-algebra $A=\Mat(n,\R)$ of real $n\times n$-matrices. A natural anti-involution on $A$ is the transposition that we denote by $\sigma$. Then by definition, $M\in\Sp_2(A,\sigma)$ if and only if $\sigma(M)\Omega M=M^T\Omega M-\Omega$ where $\Omega=\Om$. Here we used the standard identification of $\Mat(2n,\R)$ and $\Mat_2(\Mat(n,\R))$. That means, $M$ is a symplectic $2n\times 2n$-matrix and the group $\Sp_2(A,\sigma)$ agrees with $\Sp(2n,\R)$.

With the theory of symplectic groups over noncommutative involutive algebras, we can make the correspondence between Hermitian Lie groups of tube type and symplectic groups over semisimple Hermitian algebras very precise, at least for classical Hermitian Lie group of tube type. 
\begin{theorem}\label{thm_intro:Hermitian}
The following classical Hermitian Lie groups of tube type can be realized as symplectic groups over Hermitian algebras: 
\begin{enumerate}
\item $\Sp(2n,\R) = \Sp_2(A,\sigma)$, where $A=\Mat(n,\R)$ is the algebra of $n\times n$ matrices over $\R$ with involution $\sigma: A \to A$ given by $\sigma(r)=r^T$. 
\item $\UU(n,n) = \Sp_2(A,\sigma)$, where $A=\Mat(n,\CC)$ is the algebra of $n\times n$ matrices over $\CC$ with involution $\sigma: A \to A$ given by $\sigma(r)=\bar r^T$.
\item $\SO^*(4n) =  \Sp_2(A,\sigma)$, where $A=\Mat(n,\HH)$  is the algebra of $n\times n$ matrices over $\HH$ with involution $\sigma: A \to A$ given by $\sigma: A \to A$ is $\sigma(r)=\bar r^T = \bar{r_1}^T - \bar{r_2}^T j$ for $r=r_1+r_2j$ and $r_1,r_2\in \Mat(n,\CC)$. 
\end{enumerate}
\end{theorem}

\begin{remark}
The other classical Hermitian Lie group of tube type, $\SO_0(2,n)$, cannot be realized in the same way as $\Sp_2(A,\sigma)$. However its double cover $\Spin_0(2,n)$) can be realized as $\Sp_2$ over a slightly more complicated object $B_n$, 
which is a Jordan subalgebra of an appropriate Clifford algebra $(A,\sigma)$. We will discuss this in a forthcoming article. 

The exceptional Hermitian Lie group of tube type cannot be realized as $\Sp_2(A,\sigma)$, see Remark~\ref{rem:exceptional}.

There are other Lie groups that can be realized as $\Sp_2(A,\sigma)$ over involutive algebras $(A,\sigma)$ that are not Hermitian, see Section~\ref{other_examples}.
\end{remark}

For these Hermitian Lie groups of tube type, the above models for the symmetric space $X_{\Sp_2(A, \sigma)}$ give the well known explicit models of Hermitian symmetric spaces. The precompact model is the bounded symmetric domain model, and the space of isotropic lines identifies with the Shilov boundary of the bounded symmetric domain model. 

In view of Theorem~\ref{thm_intro:Hermitian} and Theorem~\ref{thm_intro:symmetric} classical Hermitian symmetric spaces of tube type can be thought of as hyperbolic planes over noncommutative involutive algebra $(A,\sigma)$. 

As a Corollary of Theorem~\ref{thm_intro:Hermitian} we can realize the complexifications of Hermitian Lie groups of tube type $\Sp_2(A,\sigma)$  as symplectic groups over complexifications of Hermitian algebras. 

\begin{theorem}
The following complex Lie groups can be realized as symplectic groups over involutive algebras: 
\begin{enumerate}
\item $\Sp(2n,\CC) = \Sp_2(A_\CC,\sigma_\CC)$, where $A=\Mat(n,\R)$ is the algebra of $n\times n$ matrices over $\R$ with involution $\sigma: A \to A$ given by $\sigma(r)=r^T$. 
\item $\GL(2n,\CC)= \Sp_2(A_\CC,\sigma_\CC)$, where $A=\Mat(n,\CC)$ is the algebra of $n\times n$ matrices over $\CC$ with involution $\sigma: A \to A$ given by $\sigma(r)=\bar r^T$.
\item $\OO(4n,\CC) =  \Sp_2(A_\CC,\sigma_\CC)$, where $A=\Mat(n,\HH)$  is the algebra of $n\times n$ matrices over $\HH$ with involution $\sigma: A \to A$ given by $\sigma: A \to A$ is $\sigma(r)=\bar r^T = \bar r_1^T - r_2^T j$ for $r=r_1+r_2j$ and $r_1,r_2\in \Mat(n,\CC)$. 
\end{enumerate}
\end{theorem}

In particular, Theorem~\ref{thm_intro:complexsymmetric} applies and we obtain explicit new realizations of models for the symmetric spaces associated to $\Sp(2n,\CC)$, $\GL(2n,\CC)$, and $\OO(4n,\CC)$. 
We illustrate here the upper half-space model for $\Sp(2n,\CC)$ and refer the reader to Section~\ref{A-class_ex} for an explicit description of all the models for $\Sp(2n,\CC)$, $\GL(2n,\CC)$, and $\OO(4n,\CC)$.

For $\Sp(2n,\CC)$, we consider again $A=\Mat(n,\R)$ with the anti-involution $\sigma$ given by transposition. Then $A_\CC=\Mat(n,\CC)$, $\sigma_\CC$ is the transposition and $\Sp_2(A_\CC,\sigma_\CC)=\Sp(2n,\CC)$. The upper half-space model of the symmetric space of $\Sp_2(A_\CC,\sigma_\CC)$ is then (see  Theorem~\ref{thm_intro:complexsymmetric}(4)) 
$$\mathfrak{U}=\{z_1+z_2 j\mid z_1\in A_\CC^{\sigma_\CC},\; z_2\in (A_\CC^{\bar\sigma_\CC})_+\}=\{z_1+z_2 j\mid z_1\in \Sym(n,\CC),\; z_2\in \Herm^+(n,\CC)\}.$$

\noindent
{\bf Structure of the paper:} In Section~\ref{A_case} we discuss algebras $A$ with anti-involution, and introduce the notion of Hermitian algebra. A more general notion of pre-Hermitian algebras and a classification of Hermitian algebras is given in Appendix~\ref{app:classification}. 
In Section~\ref{sec:symplectic} we introduce the symplectic groups $\Sp_2(A,\sigma)$ over noncommutative rings and give examples of classical Lie groups that are realized as $\Sp_2(A,\sigma)$. In Section~\ref{A-islines} we investigate the action of $\Sp_2(A,\sigma)$ on the space of isotropic lines. 
 We construct the various models of the symmetric space $X_{\Sp_2(A,\sigma)}$ in Section~\ref{AR-models}, and of the symmetric space $X_{\Sp_2(A_\CC,\sigma_\CC)}$ in Section~\ref{AC-models}. 
In Section~\ref{A-class_ex} we spell this construction out for the complexifications of the Hermitian Lie groups of tube type, giving explicit models for the symmetric spaces of $\Sp(2n,\CC)$, $\GL(2n,\CC)$, and $\OO(4n,\CC)$.

\section{Algebras with anti-involution}\label{A_case}
In this section we consider algebras with anti-involutions and introduce basic notions which play an important role throughout the paper. 

\subsection{Main definitions}
Let $\K$ be a field and $A$ a unital associative possibly noncommutative finite-dimensional $\K$-algebra. If $\K$ is a topological field, $A$ has a well defined topology. 



\begin{df}\label{def:antiinv}
An \defin{anti-involution} on $A$ is a $\K$-linear map $\sigma\colon A\to A$ such that
\begin{itemize}
\item $\sigma(ab)=\sigma(b)\sigma(a)$;
\item $\sigma^2=\Id$.
\end{itemize}
An \defin{involutive $\K$-algebra} is a pair $(A,\sigma)$, where $A$ is a $\K$-algebra and $\sigma$ is an anti-involution on $A$.
\end{df}

\begin{rem}
Sometimes in the literature the maps that satisfy  Definition~\ref{def:antiinv} are called just involutions. We add the prefix ``anti'' in order to emphasise that they exchange the factors.
\end{rem}

\begin{rem}
Notice, since the algebra $A$ is unital, we always have the canonical copy of $\K$ in $A$, namely $\K\cdot 1$ where $1$ is the unit of $A$. We will always identify $\K\cdot 1$ with $\K$. Moreover, since $\sigma$ is linear, for all $k\in \K$, $\sigma(k\cdot 1)=k\sigma(1)=k\cdot 1$, i.e. $\sigma$ preserves $\K\cdot 1$. 
\end{rem}

\begin{df} An element $a\in A$ is called \defin{$\sigma$-normal} if $\sigma(a)a=a\sigma(a)$. An element $a\in A$ is called \defin{$\sigma$-symmetric} if $\sigma(a)=a$. An element $a\in A$ is called \defin{$\sigma$-anti-symmetric} if $\sigma(a)=-a$. We denote
$$A^{\sigma}:=\Fix_A(\sigma)=\{a\in A\mid \sigma(a)=a\},$$
$$A^{-\sigma}:=\Fix_A(\sigma)=\{a\in A\mid \sigma(a)=-a\}.$$
\end{df}

\begin{ex}
Typical examples of involutive algebras are given by matrix algebras: If $\K$ is a field, then the space of $\K$-valued $n\times n$-matrices $\Mat(n,\K)$ with  anti-involution  given by the transposition is an involutive algebra.

If additionally $\K$ admits an involution $\delta\colon \K\to \K$, then $\Mat(n,\K)$ with the anti-involution $\delta\circ\sigma$ is again an example of an involutive $\K$-algebra.
\end{ex}

Semisimple involutive finite-dimensional algebras over perfect fields with the additional assumption $A^\sigma=\K$ can be classified. To state the classification, we need the following well-known definitions:

\begin{df}
A field $\K$ is called \defin{perfect} if every irreducible polynomial over $\K$ has distinct roots in its splitting field.
\end{df}

\begin{rem}
Notice that every field of characteristic zero is perfect.
\end{rem}

\begin{df}
A non-zero algebra is called \defin{simple} if it has no two-sided ideal besides the zero ideal and itself. A finite-dimensional algebra is called \defin{semisimple} if it is isomorphic to a product of simple algebras.
\end{df}

\begin{teo}\label{th:thin_main}
Let $\K$ be a perfect field with $\ch(\K)\neq 2$ and $(A,\sigma)$ be a semisimple finite-dimensional involutive algebra over $\K$ such that $A^\sigma=\K$. Then either $A$ is a division algebra over $\K$ of dimension $1, 2$ or $4$, or $A = \K \oplus \K$ and the anti-involution exchanges the two summands.
\end{teo}

A proof of this statement will be given in Proposition~\ref{pr:thin} and Theorem~\ref{th:thin}.

We denote by $A^\times$ the group of all invertible elements of $A$. If $V \subset A$ is a vector subspace, we denote
$$V^\times = A^\times \cap V\,,$$
the set of invertible elements in $V$. We consider the following map $\theta$, that will play the role of a norm on $A$ 
$$\begin{matrix}
\theta\colon & A & \to & A^\sigma\\
& a & \mapsto & \sigma(a)a
\end{matrix}$$

\begin{df}
The closed subgroup $$U_{(A,\sigma)}=\{a\in A^\times\mid \theta(a)=1\}$$ of $A^\times$ is called the \defin{unitary group} of $A$. The Lie algebra of $U_{(A,\sigma)}$ agrees with $A^{-\sigma}$. 
\end{df}

\begin{df}
\label{df:cone}
Let $(A,\sigma)$ be an algebra with an anti-involution. We define the set of $\sigma$-positive elements by 
$$A^\sigma_+:=\left\{\ \sum_{i=1}^k a_i^2\ \midwd\ a_i\in (A^\sigma)^\times,\; k\in\N\ \right\},$$
and the set of $\sigma$-non-negative elements by
$$A^\sigma_{\geq 0}:=\left\{\ \sum_{i=1}^k a_i^2\ \midwd\ a_i\in A^\sigma,\; k\in\N\ \right\}.$$
\end{df}

\begin{rem}
If $(A,\sigma)$ is an algebra over an ordered field $\K$, then $A^\sigma_{\geq 0}$ is the topological closure of $A^\sigma_+$. This follows from Corollary~\ref{open_dense_cone}.
\end{rem}

We will now give the definition of a Hermitian algebra that will be a key notion in this paper. Before this, we remind the well-known definition of a real closed field. A \defin{real closed} field is an ordered field $\K$ which is of index two in its algebraic closures. Equivalently, it is an  ordered field $\K$ in which every positive element of $\K$ is a square and every odd degree polynomial has at least one zero.


\begin{df}\label{Herm_A}
Let $\K$ be a real closed field. A unital associative $\K$-algebra with an anti-involution $(A,\sigma)$ is called \defin{Hermitian} if for all $x,y\in A^\sigma$, $x^2+y^2=0$ implies $x=y=0$.
\end{df}

Let $\K$ be a real closed field. Slightly abusing our notation, we denote the algebraic closure of $\K$ by $\K_\CC$. Since $\K$ is real closed, $\K_\CC=\K[i]$ where $i$ is a square root of $-1$ called also the \defin{imaginary unit} of $\K_\CC$. The field $\K_\CC$ is called the \defin{complexification of $\K$}. The involution $\bar\cdot\colon \K_\CC\to\K_\CC$ defined as $\bar 1=1$, $\bar i=-i$ is called the \defin{complex conjugation}. With this (anti-)involution, $\K_\CC$ becomes a Hermitian algebra over $\K$ of dimension $2$. If $\K=\R$ then as usual we denote $\CC=\R_\CC$.

For a real closed field $\K$ there exist a generalized quaternion skew-field $\K_\HH$ defined by the following presentation:
$$\K_\HH=\{x_0+x_1i+x_2j+x_3k\mid i^2=j^2=-1,~ij=-ji=k\} .$$
If $\K=\R$, then $\R_\HH$ is the classical quaternion skew-field that we will denote as usual by $\HH$.
We call $\K_\HH$ the \defin{quaternionic extension} of $\K$. With the anti-involution $\bar\cdot$ defined as $\bar i=-i$, $\bar j=-j$ and called the \defin{quaternionic conjugation}, $\K_\HH$ becomes a Hermitian algebra over $\K$ of dimension $4$. The elements $i,~j,~k$ are called imaginary units of $\K_\HH$.

Also the generalized octonionic algebra $\K_\Oc$ can be defined over any real closed field $\K$. This is the 8-dimendional non-associative noncommutative algebra that is generated (as a $\K$-vector space) by the unit $1\in\K$ and seven imaginary units $i,~j,~k,~E,~I,~J,~K$. The multiplication rule of the imaginary units is the same as in the classical octonionic $\Oc$ algebra over $\R$ 

\begin{rem}
Although in the notation of complexification and quaternionification we use symbols $\CC$, $\HH$ and $\Oc$, we do not assume that $\CC$ is a subfield of $\K_\CC$, $\HH$ is a subalgebra of $\K_\HH$ and $\Oc$ is a subalgebra of $\K_\Oc$. In the Appendix~\ref{app:classification}, we consider other division algebras over fields (not necessarily over real closed fields). In the notation of Appendix~\ref{app:classification}, $\K_\HH=\HH_{-1,-1}$ and $\K_\Oc=\Oc_{-1,-1}$.
\end{rem}

\begin{ex}
The main examples of real closed fields are $\R$ and the subfield of algebraic reals. An example of real closed field not contained into $\R$ is given by the field of Puiseux series with real coefficients.  Another example is the field of hyperreals. The field $\Q$ is not real closed, since for example $2$ is not a square. 
\end{ex}

\begin{rem}
Every real closed field has characteristic zero and contains the field of algebraic reals as a subfield. 
\end{rem}

For real closed fields the following theorem holds:
\begin{teo}[Generalized Frobenius Theorem]\label{Frob}
Let $\K$ be a real closed field. Every associative division algebra over $\K$ is isomorphic to $\K$, $\K_\CC$ or $\K_\HH$. The only non-associative noncommutative division algebra over $\K$ is $\K_\Oc$. 
\end{teo}

In case $\K=\R$ this theorem is the classical Frobenius theorem. The general version of this theorem follows from the classical one using the Tarski–Seidenberg transfer principle (see~\cite[Proposition~5.2.3]{real_algebraic}).

\begin{proposition}\label{sub-Herm_A}
A subalgebra of a Hermitian algebra that is closed under the anti-involution is Hermitian.
\end{proposition}

\begin{ex}
The following matrix algebras provide examples of Hermitian $\R$-algebras: $\Mat(n,\R)$ with the transposition and $\Mat(n,\CC)$ with the transposition composed with the complex conjugation. 
\end{ex}

Semisimple finite-dimensional algebras over real closed fields with the additional property $A^\sigma=\K$ are all Hermitian and can be classified as follows:

\begin{cor}\label{antisym_inner_product}
Let $(A,\sigma)$ be a semisimple finite-dimensional algebra over $\K$ with $A^\sigma=\K$. Then one of the following cases holds:
\begin{enumerate}
    \item $A=\K$, the anti-involution $\sigma$ acts trivially;
    \item $A=\K\oplus\K$, the anti-involution $\sigma$ permutes copies of $\K$, i.e. $\sigma(x,y)=(y,x)$;
    \item $A=\K_\CC$, the anti-involution $\sigma$ acts by complex conjugation;
    \item $A=\K_\HH$, the anti-involution $\sigma$ acts by quaternionic conjugation.
\end{enumerate}
In every case $(A,\sigma)$ is Hermitian. The bilinear form 
$$\begin{matrix}
\beta\colon & A^{-\sigma}\times A^{-\sigma} &\to&\K\\
& (a,a') & \mapsto & -aa'-a'a
\end{matrix}$$ is an inner product on $A^{-\sigma}$.
\end{cor}

This follows from the Theorems~\ref{th:thin_main} and~\ref{Frob}. 

\begin{rem}
The property to be Hermitian can also be defined in the same way for algebras with an anti-involution over any field. In this paper, we are discussing only Hermitian algebras over real closed fields. In the Appendix~\ref{app:classification}, a larger class of Hermitian and pre-Hermitian rings is considered and classified.
\end{rem}

\begin{rem}
Notice, if $\K$ is a real closed field and $(A,\sigma)$ is an involutive algebra over $\K$, then the addition and the multiplication in $A$ as well as the anti-involution $\sigma$ are semi-algebraic maps and the spaces $A$, $A^\sigma$, $A^{-\sigma}$, $A^\sigma_+$, $A^\sigma_{\geq 0}$, $U_{(A,\sigma)}$ are semi-algebraic sets.
\end{rem}

\begin{rem}\label{rk:pc_cone}
If $\K$ is a real closed field and $(A,\sigma)$ is a Hermitian algebra over $\K$, then $A^\sigma_+=\{a^2\mid a\in (A^\sigma)^\times\}$ and $A^\sigma_{\geq 0}=\{a^2\mid a\in A^\sigma\}$ are proper convex cones in $A^\sigma$. For a proof in case $\K=\R$ see~{\cite[Theorem~III.2.1]{Faraut}}. In the general case, the proof is identical.
\end{rem}

\subsection{Some properties}

\begin{prop}\label{A_subalgebra_Mat}
Let $\K$ be a field and $A$ be a unital associative $\K$-algebra of finite dimension $n$ over $\K$. Then $A$ is isomorphic to a subalgebra of $\Mat(n,\K)$.
\end{prop}

\begin{proof}
For every $x \in A$, consider the linear map $L_x:A \rightarrow A$ defined by
$$L_x(y) = x y \,. $$
Consider the map
$$A \ni x \rightarrow L_x \in \Mat(n,\K) \,.$$
This is an injective $\K$-algebra homomorphism (there is no kernel because $A$ is unital).
\end{proof}

\begin{df}\label{df:ev}
Let $\K$ be a field, $A$ be a unital associative finite dimensional $\K$-algebra and $a\in A$. An element $\lambda\in\K$ is called \defin{eigenvalue} of $a$ if $a-\lambda\cdot 1$ is not invertible. 
\end{df}

\begin{rem}\label{rem:ev_nonas}
Notice that this definition of an eigenvalue works for non-associative algebras as well.
\end{rem}

\begin{prop}\label{zero-divisor}
Let $\K$ be a field and $A$ be a unital associative finite dimensional $\K$-algebra. Then $a\in A$ is not invertible if and only if $a$ is a zero divisor in $\K[a]$.
\end{prop}

\begin{proof}
Since $A$ is of finite dimension over $\K$, for every $a\in A$ there exists $k\geq 0$ such that set $\{1,a,a^2,\dots,a^{k-1}\}$ is linearly independent over $\K$ and $\{1,a,a^2,\dots,a^k\}$ is linearly dependent over $\K$. This means that there exists a non-trivial polynomial $p\in\K[X]$ such that $\deg(p)=k$ and $p(a)=0$. Without lost of generality, assume $p=\sum_{i=0}^{k} c_iX^i$, $c_i\in \K$ and $c_k=1$.

First we note that if $a$ is non-invertible, then $p(0)=c_0=0$. Indeed, assume $p(0)=c_0\in\K^\times$. Then
$$p(a)=\sum_{i=0}^{k} c_ia^i=(\sum_{i=0}^{k-1} c_{i}a^{i-1})a+c_0=0.$$
This means that $a$ is invertible and $a^{-1}=-c_0^{-1}(\sum_{i=0}^{k-1} c_{i}a^{i-1})$. 

Now, since $c_0=0$, we write $p(X)=q(X)X$ where $q\in\K[X]$ is a non-trivial polynomial of degree $k-1$. Hence, $q(a)\neq 0$ and we obtain $q(a)a=0$. That means, $a$ is a zero divisor in $\K[a]$
\end{proof}

\begin{prop}
Let $\K$ be a field and $A$ be a unital associative finite dimensional $\K$-algebra. Let $a\in A$ and $p\in\K[X]$ such that $p(a)=0$, then $p(\lambda)=0$ for all eigenvalues $\lambda$ of $a$. In particular, there are only finitely many eigenvalues of $a$.
\end{prop}

\begin{proof}
Let $\lambda$ be an eigenvalue of $a$. Therefore, by previous Proposition, there exists $0\neq b\in\K[a]$ with $(a-\lambda\cdot 1)b=0$. In particular, $ab=\lambda b$ and by induction $a^ib=a^{i-1}(ab)=a^{i-1}(\lambda b)=\lambda^i b$ for all $i\in\N$.

If $p\in\K[X]$ be such that $p(a)=0$, then $p(a)b=0$. Since $a^ib=\lambda^i b$ for all $i\in\N$, we obtain $0=p(a)b=p(\lambda)b$. But $p(\lambda)\in \K$ and $b\neq0$, therefore, $p(\lambda)=0$.
\end{proof}

\begin{df}
We will say that a  topological field has a \defin{non-trivial topology} if its topology is not discrete nor indiscrete. 
\end{df}

The condition that the topology on a field is non-trivial is remarkably strong. Such topologies are automatically $T_0$, because the intersection of all the neighborhoods of $0$ must be an ideal, hence it is $\{0\}$. This implies that the topology is Hausdorff, from the properties of topological groups. Fields with non-trivial topologies must be infinite: in fact, for finite $\K$, every one-point set is closed means that every one-point set is open as complement of finite-point (closed) set, i.e. the topology is discrete. An important example is an ordered field endowed with the order topology.    

\begin{prop}\label{open_dense_cone}
Let $\K$ be a topological field with a non-trivial topology. Let $A$ be a unital associative finite-dimensional $\K$-algebra, and let $V$ be a vector subspace of $A$ that contains at least one invertible element. Then $V^\times$ is an open dense subset of $V$.
\end{prop}

\begin{proof} 
Notice that $x\in A$ is invertible in $A$ if and only if the linear map $L_x$ is surjective, which holds if and only if $L_x$ is invertible in $\Mat(n,\K)$. Since $0\in\K$ is closed, the set of all invertible elements $\GL(n,\K)=\Mat(n,\K)\bs \det^{-1}(0)\subset \Mat(n,\K)$ is open. Hence, the intersection of $\GL(n,\K)$ with $V$ is open in $V$ as well. 

To see the density, consider an invertible element $u \in V$. The set $u^{-1} \cdot V$ is again a vector subspace of $A$ and contains the unit $1$. It suffices to show that density holds for $u^{-1} \cdot V$. If $x \in u^{-1} \cdot V$ is not invertible, consider $y_\epsilon = x + \epsilon \cdot 1$. By the previous proposition, there are only finitely many values of $\epsilon$ such that $y_\epsilon$ is not invertible, namely the opposites of eigenvalues of $x$.  

Since $\K$ is infinite and non-discrete with closed points, there exists a (maybe punctured) neighborhood $U$ of $0\in\K$ that does not contain eigenvalues of $x$. Now every open neighborhood of $x$ intersects non-trivially the set $\{x + \epsilon \cdot 1\mid \epsilon\in U\}$. To see this, we choose a basis $(e_1,\dots,e_n)$ of $u^{-1}\cdot V$ containing $1=e_1$. Then by definition of the product topology for every neighborhood $W$ of $x$ there exist a neighborhood $W'$ of $0\in\K$ such that $W$ contains the following open neighborhood $\{x+\lambda_1e_1+\dots+\lambda_ne_n\mid \lambda_i\in W'\}$. Since $U\cap W'\neq \emptyset$, $W\cap \{x + \epsilon \cdot 1\mid \epsilon\in U\}\neq \emptyset$. This shows that $V^\times$ is dense in $V$.
\end{proof}

\begin{cor} For algebras $(A,\sigma)$ satisfying Corollary~\ref{open_dense_cone},
$A^\times$ is open and dense in $A$, $(A^\sigma)^\times$ is open and dense in $A^\sigma$.
\end{cor}

When $\K$ is an ordered field, we will assume that it is endowed with the order topology, which is always non-trivial. Recall that a subset $C \subset V$ of a $\K$-vector space is a \defin{cone} if it is stable under multiplication by a strictly positive scalar. A cone is \defin{convex} if it is stable by sums of its elements.
If $C$ is a convex cone, its closure $\overline{C}$ and its interior $\mathring{C}$ are still convex cones. The set of the opposites of the elements of $C$, denoted by $-C$, is still a convex cone.
A convex cone $C$ is \defin{proper} if
$$\overline{C} \cap - \overline{C} = \{0\}\,.$$
If $C$ is a cone in some algebra over $\K$, then we denote by $C^\times$ the subset of all invertible elements of $C$.

Similarly to Proposition~\ref{open_dense_cone} can be proven:
\begin{prop}\label{open_dense_cone2}
Let $\K$ be an ordered field. Let $A$ be a unital associative finite-dimensional $\K$-algebra, and let $C$ be a convex cone of $A$ that contains at least one invertible element. Then $C^\times$ is an open dense subcone of $C$.
\end{prop}

\begin{rem}
As we have seen in Remark~\ref{rk:pc_cone}, for Hermitian algebras, $A^\sigma_+$ and $A^\sigma_{\geq 0}$ are proper convex cones in $A^\sigma$. Proposition~\ref{open_dense_cone2} implies that $A^\sigma_+$ is open and dense in $A^\sigma_{\geq 0}$. 
\end{rem}

\subsection{Jordan algebras and spectral theorem}\label{subsect_spectral}
In this section we recall the definitions of a Jordan algebra and formally real Jordan algebra and prove some properties of formally real Jordan algebra that we will need later. We also show how Jordan algebras arise from associative algebras.

\begin{df}
Let $(V,\circ)$ be an possibly non-associative unital algebra over some field $\K$. $(V,\circ)$ said to be a \defin{Jordan algebra} if for all $x,y\in V$
\begin{enumerate}
\item $x\circ y=y\circ x$;
\item $(x\circ y)\circ(x\circ x)=x\circ(y\circ(x\circ x))$\;\;(Jordan identity).
\end{enumerate}

A Jordan algebra $(V,\circ)$ over a real closed field is called \defin{formally real} if for all $x,y\in V$, $x^2+y^2=0$ implies $x,y=0$.
\end{df}

The following is immediate:

\begin{prop}\label{ass_to_Jordan}
Let $(A,\sigma)$ be a finite dimensional algebra with anti-involution, then the following hold:
\begin{itemize}
\item The algebra $(A^\sigma,\circ)$ is a Jordan algebra where
$$x\circ y= \frac{xy+yx}{2}.$$
\item If $(A,\sigma)$ is Hermitian, the Jordan algebra $(A^\sigma,\circ)$ is formally real.
\end{itemize}
\end{prop}

Jordan algebras that arise form associative algebras as in~\ref{ass_to_Jordan} are called \defin{special Jordan algebras}. All other Jordan algebras are called \defin{exceptional Jordan algebras}.

In this section we establish some properties of formally real Jordan algebras. We do not assume that they are necessarily special.

Let $(V,\circ)$ be a formally real Jordan algebra over a real closed field $\K$. We use the following notation:
$$V_+=\left\{\sum_{i=1}^na_i^2\midwd a_i\in V^\times,\;n\in\N\right\},$$
$$V_{\geq 0}=\left\{\sum_{i=1}^na_i^2\midwd a_i\in V,\;n\in\N\right\}.$$


\begin{df}
An element $c\in V$ is called an \defin{idempotent} if $c^2=c$.
\begin{itemize}
\item Two idempotents $c,c'\in V$ are called \defin{orthogonal} if $c\circ c'=0$.
\item A tuple $(c_1,\dots,c_k)$ of pairwise orthogonal idempotents is called a \defin{complete orthogonal system of idempotents} if $c_1+\dots+c_k=1$.
\end{itemize}
\end{df}

\begin{rem}
Note that every idempotent $c$ satisfies $c\in V_{\geq 0}$. Moreover, if $c\neq 1$, then $c\in V_{\geq 0}\bs V_+$
\end{rem}

\begin{teo}[Spectral theorem, first version]\label{Spec_teo_B1}
Let $V$ be a formally real Jordan algebra over a real closed field $\K$. For every $b\in V$, there exist a unique $k\in\N$, unique elements $\lambda_1,\dots,\lambda_k\in\K$, all distinct, and a unique complete system of orthogonal idempotents $c_1,\dots,c_k\in \K[b]\subseteq V$ such that
$$b=\sum_{i=1}^k\lambda_ic_i.$$
We call this the \defin{spectral decomposition} of $b$.
\end{teo}

For a proof of this theorem for real Jordan algebras see~{\cite[Theorem~III.1.1]{Faraut}}. A proof for general real closed fields is identical. 

We collect some direct consequences of Theorem~\ref{Spec_teo_B1}

\begin{cor}\label{orth_idemp}
Let $(V,\circ)$ be a formally real Jordan algebra and $c_1,c_2\in V$ be two orthogonal idempotents. There exists an $a\in V$ such that $c_1,c_2\in \K[a]$. If $V=A^\sigma$ for an associative Hermitian algebra $(A,\sigma)$, then $c_1c_2=0$.
\end{cor}

\begin{proof}
We consider $a:=c_1-c_2$. This is the spectral decomposition of $a$. By the spectral theorem, $c_1,c_2\in \K[a]$. Let $V=A^\sigma$ for an Hermitian algebra. Since $c_1,c_2\in\K[a]$, they commute with respect to the product of $A$. Therefore, $0=c_1\circ c_2=c_1c_2$.
\end{proof}

\begin{cor}\label{theta_Bsym+}
For $b\in V_{\geq 0}$, we have $\lambda_1,\dots,\lambda_k\geq 0$. If $b\in V_+$, we have moreover, $\lambda_1,\dots,\lambda_k > 0$. 
\end{cor}

\begin{cor}
The set of all invertible elements $V^\times$ of $V$ consists of elements such that in the spectral decomposition we have $\lambda_i\neq 0$ for all $i$. If all $\lambda_i\neq 0$, then
$$\left(\sum_{i=1}^k\lambda_ic_i\right)^{-1}=\sum_{i=1}^k\lambda_i^{-1}c_i.$$
\end{cor}

\begin{cor}
For every (continuous/smooth) function $f\colon \K \to \K$, the (continuous/smooth) map
$$\hat f\colon V \to V$$
can be defined: if
$$b=\sum_{i=1}^k\lambda_ic_i,$$
then
$$\hat f(b):=\sum_{i=1}^k f(\lambda_i)c_i.$$
This map is well defined because the spectral decomposition is unique. Analogously, for any function $f\colon \K_{\geq 0} \to \K$ or $f\colon \K_+ \to \K$, $\hat f\colon V_{\geq 0} \to V$ resp. $\hat f\colon V_+ \to V$ can be defined.

In particular, for every $b\in V_{\geq 0}$, the element $b^t\in V_{\geq 0}$ for $t\in\Q_+$ is well-defined (If $\K=\R$ it is even well-defined for all $t\in\R_+$). This definition is compatible with integer powers of elements.
\end{cor}

\begin{prop}\label{contract}\begin{itemize}
\item The space $V_+$ is
semi-algebraically contractible. The set $\{1\}$ is a semi-algebraic deformation retract of $V_+$. In particular, $V_+$ is semi-algebraically connected.

\item The space $V_{\geq 0}$ is
semi-algebraically contractible. The set $\{1\}$ is a semi-algebraic deformation retract of $V_{\geq 0}$. In particular, $V_{\geq 0}$ is semi-algebraically connected.
\end{itemize}
\end{prop}

\begin{proof}
We consider the following continuous semi-algebraic map: $H(t,a)=t\cdot a+(1-t)\cdot 1$ where $t\in [0,1]=\{s\in\K\mid 0\leq s\leq 1\}\subset\K$, $a\in  V_+$. Since, $H(t,a)=\sum_{i=1}^k (\lambda_i t+(1-t))c_i$ where $(c_i)$ is a complete system of orthogonal idempotents in the spectral decomposition of $a$. Since all $\lambda_i$ are positive, the convex combination $\lambda_i t+(1-t)$ is positive for all $t\in [0,1]$. Therefore, $H(t,a)\in V_+$ for all $t\in [0,1]$, $a\in V_+$. Moreover, $H(1,a)=a$ and $H(0,a)=1$, i.e. $H$ is a contraction and $\{1\}$ is a semi-algebraic deformation retract of $V_+$.

The same contraction works also for $V_{\geq 0}$.
\end{proof}

\begin{cor}\label{Bsym+_contr_open} Let $\K=\R$.
The space $V_+$ is (non-semi-algebraically) homeomorphic to $V$. In particular, $V_+$ is open in $V$ and contractible, and  $\{1\}\subset V_+$ is a deformation retract of $V_+$.
\end{cor}

\begin{proof}
The map $f(t)=\log(t)$ gives a homeomorphism between $V_+$ and $V$. The rest follows from~\ref{contract}.
\end{proof}

We now state the second version of the spectral theorem. For this we need to give some additional definitions:

\begin{df}
An idempotent $0\neq c\in V$ is called \defin{primitive} if it cannot be written as a sum of two orthogonal non-zero idempotents.
\begin{itemize}
\item A complete orthogonal system of primitive idempotents $(c_1,\dots,c_k)$ is called a \defin{Jordan frame}.
\item The maximal number of elements in a Jordan frame is called the \defin{rank} of a Jordan algebra.
\end{itemize}
\end{df}

\begin{teo}[Spectral theorem, second version]\label{Spec_teo_B2}
Let $V$ be a formally real Jordan algebra over a real closed field $\K$. Suppose that $V$ has rank $n$. For every $b\in V$ there exist a Jordan frame $(e_1,\dots,e_n)$ and a unique $n$-tuple of elements  $(\lambda_1(b),\dots,\lambda_n(b))\in\K^n$ such that $\lambda_1(b)\geq\dots\geq\lambda_n(b)$ and
$$b=\sum_{i=1}^n\lambda_i(b)e_i.$$
The elements $\lambda_1(b),\dots,\lambda_n(b)\in\K$ (with their multiplicities) are called the  \defin{eigenvalues} of $b$ are uniquely determined by $b$. In particular, they do not depend (up to permutations) on the Jordan frame $e_1,\dots,e_n\in V$.
\end{teo}

For a proof of this theorem for real Jordan algebras see~{\cite[Theorem~III.1.2]{Faraut}}. The proof for a general real closed field is identical. 

\begin{rem}
A Jordan frame $e_1,\dots,e_n\in V$ associated to the element $b\in V$ by Theorem~\ref{Spec_teo_B2} is in general not unique and elements $e_1,\dots,e_n\notin \K[b]$, in contrast to the complete system of orthogonal idempotents from  Theorem~\ref{Spec_teo_B1}.
\end{rem}

\begin{rem}
In Definition~\ref{df:ev} we already defined  eigenvalues for elements of an associative algebras. This definition works also for non-associative algebras (see Remark~\ref{rem:ev_nonas}). If $V$ is a formally real Jordan algebra, then the eigenvalues defined by the spectral theorem agree with the eigenvalues defined in Definition~\ref{df:ev}.
\end{rem}

\begin{df}
The \defin{signature} of an element $b\in V$ is a pair $(p,q)$ with $p,q\in\N\cup\{0\}$ such that $b$ has $p$ positive eigenvalues and $q$ negative eigenvalues in the spectral decomposition defined in~\ref{Spec_teo_B2}. 
\end{df}

\begin{rem}
Notice that for an element $b$ of signature $(p,q)$, $p+q\leq n$, moreover, $p+q=n$ if and only if $b$ is invertible.
\end{rem}

\begin{df}\label{tr_det}
Let $b\in V$ and $\lambda_1(b),\dots,\lambda_n(b)$ are all its eigenvalues with multiplicities. We define 
the \defin{eigenvalue map}:
$$\begin{matrix}
\lambda\colon & V & \mapsto & \K^n\\
& b & \to & (\lambda_1(b),\dots,\lambda_n(b)),
\end{matrix}$$
and the \defin{trace} and the \defin{determinant} maps:
$$\tr(b):=\sum_{i=1}^n \lambda_i(b),\; \det(b):=\prod_{i=1}^n \lambda_i(b).$$
\end{df}

\begin{rem}
Since the eigenvalues defined by the spectral theorem and eigenvalues defined in~\ref{df:ev} agree, the eigenvalue map is semi-algebraic. The eigenvalue map is continuous. For a proof when $\K = \R$ see~\cite[Corollary 24]{Baes}. For a general real closed field the proof is identical.
\end{rem}

\begin{cor}\label{inner_prod_B}
The function
$$\begin{matrix}
\beta\colon & V\times V & \to & \K\\
& (b_1,b_2) & \mapsto & \tr\left(b_1\circ b_2\right)
\end{matrix}$$
is an inner product on $V$, where $V$ is intended as a $\K$-vector space.
\end{cor}


\begin{cor}\label{pc_cone}\begin{itemize}
\item The set $V_+$ is an open and closed proper convex cone in $V^\times$. In particular, $V_+$ is a semi-algebraic connected component of $V^\times$.

\item The set $V_{\geq 0}$ is a closed proper convex cone in $V$.
\end{itemize}
\end{cor}

\begin{proof}
From~\ref{rk:pc_cone} we already know that $V_+$ and $V_{\geq 0}$ are proper convex cones in $V^\times$ and in $V$ respectively. 

The set $V_+=\lambda^{-1}(\K_+^n)$ is open and closed in $V^\times=\lambda^{-1}((\K\bs\{0\})^n)$. In particular, $V_+$ is a semi-algebraic connected component of $V$ because by~\ref{contract} $V_+$ is semi-algebraically connected.

The set $V_+=\lambda^{-1}(\K^n_{\geq 0})$ is closed in $V$.
\end{proof}

\subsection{Classification of simple formally real Jordan algebras}

In this section, we recall the well-known classification of simple formally real Jordan algebras over $\R$ and generalize it to all real closed fields.

\begin{teo}\label{Jclassif}
Every simple formally real Jordan algebra over a real closed field $\K$ is isomorphic to one of the following Jordan algebras:
\begin{enumerate}
\item Symmetric real matrices $(\Sym(n,\K),\circ)$ where $a\circ b=\frac{ab+ba}{2}$ for $a,b\in \Sym(n,\K)$;
\item Hermitian complex matrices $(\Herm(n,\K_\CC),\circ)$ where $a\circ b=\frac{ab+ba}{2}$ for $a,b\in \Herm(n,\K_\CC)$;
\item Hermitian quaternionic matrices $(\Herm(n,\K_\HH),\circ)$ where $a\circ b=\frac{ab+ba}{2}$ for $a,b\in \Herm(n,\K_\HH)$;
\item $(B_n,\circ)$ where $B_n=\Span_\K(1,x_1,\dots, x_n)$ with $1\circ x_i=x_i$, $x_i\circ x_j=0$ for $i\neq j$, $x_i\circ x_i=1$ for all $i,j\in\{1,\dots,n\}$;
\item Hermitian octonionic $3\times 3$ matrices $(\Herm(3,\K_\Oc),\circ)$ where $a\circ b=\frac{ab+ba}{2}$ for $a,b\in \Herm(3,\K_\Oc)$.
\end{enumerate}
\end{teo}

In case $\K=\R$ this theorem agrees with the classification of formally real Jordan algebras over $\R$ that is proven in~\cite{Hanche84, Faraut}. The general version of this classification follows from the classification in case $\K=\R$ using the Tarski–Seidenberg transfer principle (see~\cite[Proposition~5.2.3]{real_algebraic}).

The Jordan algebras (1)--(3) from~\ref{Jclassif} can be seen as $A^\sigma$ for certain simple algebras $(A,\sigma)$. Let $\K$ be a real closed field.

\begin{enumerate}
\item Symmetric real matrices: $A=\Mat(n,\K)$, $\sigma(r):=r^T$ is an algebra with an anti-involution. Then $A^\sigma=\Sym(n,\K)$ space of all symmetric matrices. The algebra $(A,\sigma)$ is Hermitian with $A^{\sigma}_+=\Sym^+(n,\K)$ real symmetric positive definite matrices.

\item Hermitian complex matrices: $A=\Mat(n,\K_\CC)$, $\bar\sigma(r):=\bar r^T$ is a Hermitian algebra with $A^{\bar\sigma}=\Herm(n,\K_\CC)$ complex Hermitian matrices and $A^{\bar\sigma}_+=\Herm^+(n,\K_\CC)$ complex Hermitian positive definite matrices.

\item Hermitian quaternionic matrices: $A=\Mat(n,\K_\HH)$, $\bar\sigma(r):=\bar r^T$, is a Hermitian algebra  with $A^{\bar\sigma}=\Herm(n,\K_\HH)$ quaternionic Hermitian matrices and $A^{\bar\sigma}_+=\Herm^+(n,\K_\HH)$ quaternionic Hermitian positive definite matrices.\\
There is another anti-involution on $A=\Mat(n,\K_\HH)$, namely $\sigma_1(r):=\sigma(r_1)+\bar\sigma(r_2)j$ where $r_1,r_2\in\Mat(n,\K_\CC)$. This algebra $(A,\sigma_1)$ is not Hermitian. 
\end{enumerate}

\begin{fact}[{\cite[Corollary~2.8.5]{Hanche84}}]
The Jordan algebra $(\Herm(3,\K_\Oc),\circ)$ is exceptional. This means that there is no associative real algebra $A$ that contains $\Herm(3,\K_\Oc)$ as a Jordan subalgebra.
\end{fact}

\begin{rem}
The Jordan algebra $(B_n,\circ)$ can be embedded as a Jordan subalgebra into the even Clifford algebra $\Cl_{even}(1,n)$ for some appropriate anti-involution $\sigma$, but $\Cl(1,n)^\sigma$ is strictly bigger than $B_n$ for $n>2$. In the case $n=2$, $\Cl(1,2)$ is isomorphic to $\Mat(2,\K)$ as an associative algebra, and $B_2$ is isomorphic to $\Sym(2,n)$ as a Jordan algebra.
\end{rem}


\subsection{Positivity of the norm}

Let $\K$ be a real closed field and $(A,\sigma)$ be a Hermitian semisimple algebra over $\K$. The goal of this subsection is to show that $\theta(a)=\sigma(a)a\in A^\sigma_{\geq 0}$ for all $a\in A$. To do it, first, we prove some technical propositions.

\begin{prop}\label{minus_antisym}
Let $(A,\sigma)$ be a Hermitian semisimple algebra, then $-a^2\in A^\sigma_{\geq 0}$ for every $a\in A^{-\sigma}$.
\end{prop}

\begin{proof}
We prove it by induction on the rank of $A$. If its rank is one, then the statement follows from the Corollary~\ref{antisym_inner_product}.

Let now rank of $A$ be equal $n>1$ and we assume that for every Hermitian semisimple algebra $A'$ of rank smaller then $n$, the Proposition holds. Take $a\in A^{-\sigma}$, then $a^2\in A^\sigma$. We consider a spectral decomposition:
$$a^2=\sum_{i=1}^n \lambda_i e_i$$
where $(e_i)_{i=1}^n$ is a Jordan frame. We denote
$$E:=\sum_{i=1}^{n-1}e_i,\;\;e:=e_{n}.$$
Then $E+e=1$, $eE=Ee=0$, and algebras $(EAE,\sigma)$ and $(eAe,\sigma)$ are Hermitian semisimple subalgebras of $A$ of rank strictly smaller then rank of $A$. Therefore,
\begin{align*}
a^2 = & (EaE+Eae+eaE+eae)^2=EaEaE+eaeae+(eaEae+EaeaE)\\
& +(Eaeae+EaEae)+(eaeaE+eaEaE).
\end{align*}
Since $EaEaE=(EaE)^2$ and $EaE\in (EAE)^{-\sigma}$, we have $-EaEaE\in (EAE)^\sigma_{\geq 0}\subseteq A^\sigma_{\geq 0}$. The same holds for $-eaeae$. Further,
$$-(eaEae+EaeaE)=(exE-Exe)^2\in A^\sigma_{\geq 0}$$
because $exE-Exe\in A^\sigma$.
Finally,
$$Eaeae+EaEae=Ea(E+e)ae=Ea^2e=0$$
beacsue of the spectral decomposition of $a^2$ and $eE=Ee=0$. The same holds for $eaeaE+eaEaE$.
\end{proof}

\begin{prop}\label{prop:Jacobson}
Let $(A,\sigma)$ be a Hermitian semisimple algebra and $x\in A$ satisfy the property $\sigma(x)x=0$, then $x=0$.
\end{prop}

\begin{proof} We denote: $A_0:=\{x\in A\mid \sigma(x)x=0\}$.
First, note that $\sigma(x)x=0$ if and only if $x\sigma(x)=0$. Indeed, let $\sigma(x)x=0$, then $x\sigma(x)\in A^\sigma$ and $(x\sigma(x))^2=x\sigma(x)x\sigma(x)=0$. Therefore, since $A$ is Hermitian, $x\sigma(x)=0$.

Further, write $x=x^s+x^a$ where $x^s\in A^\sigma$, $x^a\in A^{-\sigma}$. Then
$$\sigma(x)x=(x^s)^2-(x^a)^2+x^sx^a-x^ax^s=0=x\sigma(x)=(x^s)^2-(x^a)^2-x^sx^a+x^ax^s.$$
Therefore, $x^sx^a-x^ax^s=0$ and $(x^s)^2-(x^a)^2=0$. Since $(x^s)^2,-(x^a)^2\in A^{\sigma}_{\geq 0}$, $x^s=0$ and $(x^a)^2=0$. Therefore $1+x^a=1+x$ is invertible and $(1+x)^{-1}=1-x$.

Let $r\in A$, consider $y=xr$. Then $\sigma(y)y=\sigma(r)\sigma(x)xr=0$, i.e. $y\in A_0$. Further, take $r'\in A$ and consider $z=r'y$, then $z\sigma(z)=0$. Therefore, as we have seen, $\sigma(z)z=0$ and $z\in A_0$. That means, for every $r,r'\in A$, $1+r'xr\in A^\times$, i.e. $x\in J(A)$ where $J(A)$ is the Jacobson radical of $A$ (see Definition~\ref{Jacobson}). Because $A$ is semisimple, $J(A)=\{0\}$.
\end{proof}

Let $\K$ be a real closed field. As before we denote by $\K_\CC$ the algebraic closure of $\K$. 

\begin{prop}\label{Herm_complexification}
Let $(A,\sigma)$ be a Hermitian semisimple algebra, then the complexified algebra $(A_\CC,\bar\sigma_\CC)$, where $A_\CC=A\otimes_\K \K_\CC$ and $\bar\sigma_\CC$ the complex anti-linear extension of $\sigma$ (i.e. $\bar\sigma(x+iy)=\sigma(x)-i\sigma(y)$ for $x,y\in A$), is Hermitian as an involutive algebra over $\K$ ans semisimple.
\end{prop}

\begin{proof}
The algebra $(A_\CC,\bar\sigma_\CC)$ is clearly unital and associative. We check now that, if $x^2+y^2=0$, then $x=y=0\in A^{\bar\sigma_\CC}_\CC$.

For $x\in A^{\bar\sigma_\CC}_\CC$ we write $x=x_1+ix_2$, $x_1,x_2\in A$. Since 
$\bar\sigma_\CC(x)=\sigma(x_1)-i\sigma(x_2)$,
$x_1\in A^\sigma$, $x_2\in A^{-\sigma}$. Further,
$$x^2+y^2=x_1^2+y_1^2-x_2^2-y_2^2+i(x_1y_2+y_1x_2+x_2y_1+y_2x_1)=0.$$
Since $x_1^2,y_1^2,-x_2^2,-y_2^2\in A^\sigma_{\geq 0}$, we have $x_1^2=y_1^2=-x_2^2=-y_2^2=0$. Therefore, by Proposition~\ref{prop:Jacobson}, $x_1=x_2=y_1=y_2=0$, i.e. $x=y=0$.

Assume now that $A_\CC$ is not semisimple, i.e. the Jacobson radical $J(A_\CC)$ of $A_\CC$ is non-trivial. Let $0\neq x\in J(A_\CC)$, then by~\ref{preHerm_prop}, $\bar\sigma_\CC(x)=-x$ and $x^2=0$. Therefore, $\bar\sigma_\CC(ix)=ix$, i.e. $ix\in A^{\bar\sigma_\CC}_\CC$ and $(ix)^2=0$. This contradicts to the property to be Hermitian for $(A_\CC,\bar\sigma_\CC)$.
\end{proof}

\begin{prop}
Let $Y$ be a finite dimensional $\K_\CC$-algebra, $V\subseteq Y$ be a $\K_\CC$-vector subspace. Then $V^\times$ is semi-algebraically connected.
\end{prop}

\begin{proof} If $V^\times=\emptyset$ then $V^\times$ is semi-algebraically connected.

Assume now that $1\in V^\times\neq\emptyset$. As we have seen in the Proposition~\ref{A_subalgebra_Mat}, $Y$ can be embedded as a subalgebra into $\Mat(r,\K_\CC)$ for some $r\in \N$. We identify $Y$ as a subalgebra of $\Mat(r,\K_\CC)$.

Let $a\in V^\times\subseteq\GL(n,\K_\CC)$. We consider 
$$S^1_{\K_\CC}:=\{x+iy\in\K_\CC\mid x,y\in\K, x^2+y^2=1\}.$$
Notice, that $S^1_{\K_\CC}$ is a semi-algebraically connected set. In particular, it has infinitely many points.
Since $a$ has only finitely many eigenvalues, and $0$ is not one of them, there is a point $z\in S^1_{\K_\CC}\subset\K_\CC$ such that the $\K$-line $\{tz\mid t\in\K\}$ in $\K_\CC$ through the origin containing $z$ does not intersect any of the eigenvalues of $a$. Now, consider the $\K$-path $f(t)=at+z(1-t)\Id$, $t\in[0,1]=\{s\in\K\mid 0\leq s\leq 1\}\subset \K$. It lies completely in $V$ because it is a $\CC$-vector space. This has determinant $0$ if and only if $z(t-1)$ is an eigenvalue of $at$, which happens if and only if $z(1-t)/t$ is an eigenvalue of $a$ (this does not work when $t=0$, but then it is clear that the determinant is non-zero). By construction, it is not the case for any $t\in[0,1]\subset\K$, so this defines a path form $a$ to $z\Id$.
Now, there is a path in $\K_\CC$ not passing through $0$ from $z$ to $1$, and, since $\{z\Id\mid z\in S^1_{\K_\CC}\}\subset V^\times$, this gives rise to a semi-algebraic path in $A_\CC^\times$ from $z\Id$ to $\Id$, and so concatenating these two paths, we get a path from $a$ to $\Id$ that lies in $V^\times$, showing that $V^\times$ is semi-algebraically connected.

Finally, if $1\notin V^\times$ but there exists $v\in V^\times$, then $V^\times$ is semi-algebraically connected if and only if $(v^{-1}V)^\times=v^{-1}(V^\times)$ is semi-algebraically connected. Moreover, $v^{-1}V$ is also a $\K_\CC$-vector space and $1\in (v^{-1}V)^\times$. So $(v^{-1}V)^\times$ is semi-algebraically connected and, therefore, $V^\times$ is semi-algebraically connected as well.
\end{proof}

Now we are ready to prove the following

\begin{teo}\label{Positive_sigma(a)a}
Let $(A,\sigma)$ be a Hermitian semisimple algebra, then $\sigma(a)a\in A^\sigma_{\geq 0}$ for all $a\in A$. In particular, $A^\sigma_+=\theta(A^\times)$ and $A^\sigma_{\geq 0}=\theta(A)$.
\end{teo}

\begin{proof}
First, we embed $(A,\sigma)$ into its complexification $(A_\CC,
\bar\sigma_\CC)$ that is Hermitian as well. Let $a\in A^\times\subset A_\CC^\times$. Since $A_\CC^\times$ is semi-algebraically connected, we take a semi-algebraic path $\gamma\colon [0,1]\to A^\times_\CC$ such that $\gamma(0)=a$, $\gamma(1)=1\in A$. The eigenvalue map:
$\lambda(\sigma(\gamma(t))\gamma(t))$ takes values in the semi-algebraic connected component of $(1,\dots,1)\in(\K\bs\{0\})^n$, i.e. in $\K_+^n$. In particular, $\lambda(\sigma(a)a)=\lambda(\sigma(\gamma(0))\gamma(0))$ is positive meaning that $a\in A^\sigma_+\subset (A^{\bar\sigma_\CC}_\CC)_+$.

Since $A$ is the closure of $A^\times$, the eigenvalue map takes values in the closure of the connected component of $(1,\dots,1)\in(\R\bs\{0\})^n$, i.e. all eigenvalues of $\sigma(a)a$ are non-negative for $a\in A$.
\end{proof}

\begin{rem}
We conjecture that the Theorem~\ref{Positive_sigma(a)a} holds also for non-semisimple Hermitian algebras.
\end{rem}

\begin{df}\label{df:str_gr}
Let $V$ be a formally real Jordan algebra over a real closed field $\K$ and $V_+$ be the cone of positive elements of $V$. The group of all linear transformations of $V$ (as $\K$-vector space) which preserve $V_+$ is called \defin{structure group of $V$} and denoted by $G(V)$.
\end{df}

\begin{cor}\label{action_on_Asigma}
The group $A^\times$ acts on $A^\sigma$ preserving $A^\sigma_+$ in the following way:
$$\begin{matrix}
\psi\colon & A^\times\times A^\sigma & \mapsto & A^\sigma\\
& (a,b) & \to & \sigma(a)ba.
\end{matrix}$$
In other words, $\psi(A^\times)$ is a subgroup of $G(A^\sigma)$. Moreover, the restriction of this action to $A^\sigma_+$ is transitive.
\end{cor}

\subsection{Polar decomposition}
Let $(A,\sigma)$ to be a Hermitian algebra. From now on, we always assume the algebra $A$ to be semisimple.

\begin{teo}[Polar decomposition, first version]\label{pol_decomp1}
The map
$$\begin{matrix}
\pol\colon & U_{(A,\sigma)}\times A^\sigma_+ & \to & A^\times \\
& (u,b) & \mapsto & ub
\end{matrix}$$
is a semi-algebraic homeomorphism. In particular, for every $g\in A^\times$ there exist unique $b\in A^\sigma_+$ and $u\in U_{(A,\sigma)}$ such that $g=ub$.
\end{teo}

\begin{proof}
The map $\pol$ is well-defined because $A^\sigma_+\subseteq A^\times$. First, we prove the surjectivity. Take $g\in A^\times$, then $\sigma(g)g\in A^\sigma_+$ by Proposition~\ref{Positive_sigma(a)a}. Take $b:=(\sigma(g)g)^{\frac{1}{2}}$, then $u:=g(\sigma(g)g)^{-\frac{1}{2}}\in U_{(A,\sigma)}$. Indeed,
$$\sigma(u)u=(\sigma(g)g)^{-\frac{1}{2}}\sigma(g)g(\sigma(g)g)^{-\frac{1}{2}}=1.$$

Now, we prove the injectivity. Let $g=ub=u'b'$ where $u,u'\in U_{(A,\sigma)}$, $b,b'\in A^\sigma_+$. Then $\sigma(g)g=(b')^2=b^2\in A^\sigma_+$. We take the spectral decompositions of $b$ and $b'$:
$$b=\sum_{i=1}^k\lambda_i c_i,\;b'=\sum_{i=1}^{k'}\lambda_i' c_i'$$
where all $k,k'\in \N$, $\lambda_i,\lambda_i'>0$ and $\{c_i\}$, $\{c'_i\}$ are complete orthogonal systems of idempotents of $A^\sigma$. Then
$$b^2=\sum_{i=1}^k\lambda_i^2 c_i=\sum_{i=1}^{k'}(\lambda_i')^2 c_i'=(b')^2.$$
Because of the uniqueness of the spectral decomposition, $k=k'$ and, up to reordering, all $\lambda_i^2=(\lambda_i')^2$, $c_i=c_i'$. But all $\lambda_i>0$, therefore, $\lambda_i=\lambda_i'$, i.e. $b=b'$ and $u=gb^{-1}=g(b')^{-1}=u'$.

Finally, by definition, $\pol$ is continuous. Moreover, $$\pol^{-1}(g)=(g(\sigma(g)g))^{-\frac{1}{2}},(\sigma(g)g))^\frac{1}{2})$$
is continuous as well and contains only semi-algebraic operations. Therefore, $\pol$ is a semi-algebraic homeomorphism.
\end{proof}

\begin{cor}
The map
$$\begin{matrix}
A^\sigma_+ \times U_{(A,\sigma)} & \to & A^\times \\
(b,u) & \mapsto & bu
\end{matrix}$$
is a semi-algebraic homeomorphism. In particular, for every $g\in A^\times$ there exist unique $b\in A^\sigma_+$ and $u\in U_{(A,\sigma)}$ such that $g=bu$.
\end{cor}

\begin{cor}\label{U_def_retr}
The group $U_{(A,\sigma)}<A^\times$ is a semi-algebraic deformation retract of $A^\times$.
\end{cor}

\begin{proof}
This follows from the polar decomposition~\ref{pol_decomp1} and from~\ref{contract}.
\end{proof}

\begin{prop}\label{pr:U_orth}
The following $\K$-bilinear form $\beta\colon A\times A\to\K$:
$$\beta(a_1,a_2):=\tr\left(\frac{\sigma(a_1)a_2+\sigma(a_2)a_1}{2}\right)$$
on $A$ is positive definite. The group $U_{(A,\sigma)}$ acts on $A$ by the left and right multiplication preserving $\beta$.
\end{prop}

\begin{proof}
The map $\beta$ is $\K$-bilinear on $A$. Moreover, $\beta(a,a)=\tr(\sigma(a)a)\geq 0$ for all $a\in A$ because $\sigma(a)a\in A^\sigma_{\geq 0}$.

If $\beta(a,a)=0$, then all eigenvalues of $\sigma(a)a$ are zero, i.e. $\sigma(a)a=0$ and, therefore, because of semisimplicity of $A$ we have $a=0$ (see Proposition~\ref{prop:Jacobson}).

The form $\beta$ is invariant under the left and right action by multiplication of $U_{(A,\sigma)}$ because the trace is invariant under conjugation.
\end{proof}

\begin{cor}
The adjoint action of $U_{(A,\sigma)}$ on $A$ preserves $\beta$. In particular, the Lie algebra $A^{-\sigma}$ of $U_{(A,\sigma)}$ is compact.
\end{cor}

\begin{cor}\label{cb_sets}
With respect to the norm induced by the inner product $\beta$ on $A$, the following subsets of $A$ are closed, bounded and semi-algebraic:
\begin{itemize}
    \item the group $U_{(A,\sigma)}$;
    \item $D(A,\sigma):=\{a\in A\mid 1-\sigma(a)a\in A^\sigma_{\geq 0}\}$.
\end{itemize}
In particular, for every vector subspace $V$ of $A$, the set
$$D(V,\sigma):=\{a\in V \mid 1-\sigma(a)a\in A^\sigma_{\geq 0}\}=D(A,\sigma)\cap V$$
closed, bounded and semi-algebraic.
\end{cor}

We remind the definition of a derivation of a Jordan algebra:

\begin{df}
Let $V$ be a Jordan algebra over some field $\K$. A $\K$-linear map $D\colon V\to V$ is called a \defin{derivation} if $D(xy)=D(x)y+xD(y)$ for all $x,y\in V$. The space of all derivations of $V$ is denoted by $\Der(V)$. 
\end{df}

\begin{rem}
The space $\Der(V)$ is a subspace of the space of all $\K$-vector space endomorphisms $\End(V)$ of $V$. Moreover, $\Der(V)$ is a Lie algebra with respect to the commutator and it agrees with the Lie algebra of the group of all Jordan algebra automorphisms $\Aut(V)$ of $V$.
\end{rem}

\begin{prop}
Let $(A,\sigma)$ be a semisimple associative algebra over some field $\K$. The map $\ad\colon A^{-\sigma}\to \Der(A^\sigma)$, $\ad(u)a:=ua-au$ for $u\in A^{-\sigma}$, $a\in A^\sigma$ is a surjective homomorphism of Lie algebras.
\end{prop}

\begin{proof}
As we have seen before, for all $u\in A^{-\sigma}$ and for all $a\in A^\sigma$, $\ad(u)a=ua-au\in A^\sigma$. Moreover, it is easy to see that $\ad(u)$ is a derivation for all $u\in A^{-\sigma}$ and that $\ad$ is a Lie algebra homomorphism. Therefore, $\ad(A^{-\sigma})\subseteq \Der(A^\sigma)$.

Since $A$ is semisimple, $A^\sigma$ is a semisimple Jordan algebra.  Every derivation of a semisimple Jordan algebra is inner (see~\cite[Theorem~2]{Jacobson}), i.e. for every derivation $D\in\Der(A^\sigma)$ there exist $a_1,a_2\in A^\sigma$ such that for all $a\in A^\sigma$, $D(a)=\ad([a_1,a_2])a=[a_1,a_2]a-a[a_1,a_2]$, where $[\cdot,\cdot]$ is the commutator on $A$. Finally, notice that $[a_1,a_2]\in A^{-\sigma}$ for $a_1,a_2\in A^\sigma$. That means that $\ad(A^{-\sigma})=\Der(A,\sigma)$.
\end{proof}

\begin{cor} Let $(A,\sigma)$ be a semisimple associative algebra over a real closed field $\K$.
The semi-algebraic connected component of the identity $\Aut_0(A^\sigma)$ of the automorphism group $\Aut(A^\sigma)$ of the Jordan algebra $A^\sigma$ agrees with the semi-algebraic  connected component of the identity of the group $\Ad(U_{(A,\sigma)})$ where $\Ad$ is the adjoint Lie group action on its Lie algebra.
\end{cor}

\begin{proof}
This statement follows from the fact that $\Lie(U_{(A,\sigma)})=A^{-\sigma}$, $\Lie(\Aut(A^\sigma))=\Der(A^\sigma)$ and the derivative of the adjoint action $\Ad$ is $\ad$.
\end{proof}

\begin{cor}\label{transit_JF}
Let $(A,\sigma)$ be a semisimple Hermitian algebra over a real closed field $\K$ such that the Jordan algebra $A^\sigma$ is simple. Then $U_{(A,\sigma)}$ acts transitively on the set of all Jordan frames of $A^\sigma$. 
\end{cor}

\begin{proof}
This statement follows from the fact that $\Aut_0(A^\sigma)$ acts transitively on the set of all Jordan frames of $A^\sigma$ (see~\cite[Corollary~IV.2.6]{Faraut} for a proof over $\R$, general case follows identically).
\end{proof}

\begin{cor}[Sylvester's law of inertia]\label{Sylvester}
Let $(A,\sigma)$ be a semisimple Hermitian algebra over a real closed field $\K$ such that the Jordan algebra $A^\sigma$ is simple. The action
$$\begin{matrix}
\psi\colon & A^\times\times A^\sigma & \mapsto & A^\sigma\\
& (a,b) & \to & \sigma(a)ba.
\end{matrix}$$
of the group $A^\times$ on $A^\sigma$ preserves the signature of elements. The orbits of this action are precisely the sets of all elements of fixed signature in $A^\sigma$.
\end{cor}

\begin{proof} Let $k$ be the rang of $A^\sigma$. We fix a Jordan frame $(e_i)_{i=1}^k$ and denote: $o_{p,q}:=\sum_{i=1}^p e_i-\sum_{i=1}^qe_{k-i+1}$ for $p+q\leq k$.
First notice that for every element $b\in A^\sigma$ of signature $(p,q)$ there is an element $g\in A^\times$ such that $o_{p,q}=\sigma(g)b_1g$. To see this, first, we take a spectral decomposition $b=\sum_{i=1}^k \lambda_ic_i$ for a Jordan frame $(c_i)_{i=1}^k$ and $\lambda_1\geq\dots\geq\lambda_k$. Let $u\in U_{(A,\sigma)}$ such that $u$ maps the Jordan frame $(c_i)_{i=1}^k$ to $(e_i)_{i=1}^k$. Therefore: $\sigma(u)bu=\sum_{i=1}^k \lambda_ie_i$. Finally, we take $a:=\sum_{i=1}^k \mu_ie_i$ where $\mu_i=\lambda_i^{-1}$ if $\lambda_i\neq 0$ and $\mu_i=1$ otherwise. Then $\sigma(ua)b(ua)=o_{p,q}$.

The structure group $G(A^\sigma)$ acts on $A^\sigma$ preserving the signature (see~\cite[Theorem~1]{Kaneyuki}). Therefore, $A^\times$ does it as well.
\end{proof}

\subsection{Matrix algebra over a Hermitian algebra}
One way to construct new semisimple Hermitian algebras is to consider a matrix algebra over a semisimple Hermitian algebra. For this, we assume $(A,\sigma)$ to be a semisimple Hermitian algebra. We consider the following anti-involution on the algebra $\Mat_n(A)$ of $n\times n$-matrices over $A$:
$$\begin{matrix}
\sigma^T\colon & \Mat_n(A) & \to & \Mat_n(A)\\
& M & \mapsto & \sigma(M)^T
\end{matrix}$$
where $\sigma(M)$ means applying $\sigma$ componentwise to elements of the matrix $M\in\Mat_n(A)$. We denote
$$\Sym_n(A,\sigma):=\Fix_{\Mat_n(A)}(\sigma^T);$$
$$\Sym_n^{\geq 0}(A,\sigma):=\{M\in\Sym_n(A,\sigma)\mid \sigma(x)^TMx\in A^\sigma_{\geq 0}\text{ for all }x\in A^n\};$$
$$\Sym_n^+(A,\sigma):=\{M\in\Sym_n(A,\sigma)\mid \sigma(x)^TMx\in A^\sigma_+\text{ for all regular }x\in A^n\};$$
$$\UU_n(A,\sigma):=U_{(\Mat_n^\times(A),\sigma^T)}=\{M\in \Mat_n(A)\mid \sigma(M)^TM=\Id_n\}.$$

\begin{prop}\label{Herm_2Matrix}
For a Hermitian semisimple algebra $(A,\sigma)$, the algebra $(\Mat_n(A),\sigma^T)$ is Hermitian and semisimple.
\end{prop}

\begin{proof} First, we have to check that for every two $M,N\in\Sym_n(A,\sigma)$, if $M^2+N^2=0$, then $M=N=0$.

Note that for every $x\in A^n$, if $\sigma(x)^Tx=0$, then $x=0$. Indeed, take $x=(x_1,\dots,x_n)^T$ such that $\sigma(x)^Tx=\sigma(x_1)x_1+\dots+\sigma(x_n)x_n=0$. Since all $\sigma(x_i)x_i\in A^\sigma_{\geq 0}$, all $\sigma(x_i)x_i=0$. Because of semisimplicity of $A$, all $x_i=0$.

Now, take $x\in A^n$ and consider $Mx,Nx\in A^n$. Assume $M^2+N^2=0$, then
$$0=\sigma(x)^T(M^2+N^2)x=\sigma(x)^TM^2x+\sigma(x)^TN^2x=\sigma(Mx)^T Mx+\sigma(Nx)^T Nx.$$
Since for every $y\in A^2$, $\sigma(y)^Ty\in A^{\sigma}_{\geq 0}$, $\sigma(Mx)^T Mx,\sigma(Nx)^T Nx=0$, By semisimplicity of $A$, $Mx=Nx=0$ for all $x\in A^n$. Therefore, $M=N=0$.

Finally, we have to check that $\Mat_n(A)$ is semisimple. We do it by induction. Assume $\Mat_k(A)$ is semisimple for all $k<n$.
Take $X\in J:=J(\Mat_n(A))$ in the Jacobson radical of $\Mat_n(A)$ (see Definition~\ref{Jacobson}), we write $X$ as a block matrix:
$$X=\begin{pmatrix}
x & a \\
b & y
\end{pmatrix}$$
where $x\in\Mat_{n-1}(A)$, $a,b^T\in A^{n-1}$, $y\in A$. Since $X\in J$ by Proposition~\ref{preHerm_prop}, $\sigma^T(X)=-X$ and $X^2=0$, i.e. $\sigma(x)^T=x$, $b=-\sigma(a)^T$, $y\in A^{-\sigma}$,
$$X^2=\begin{pmatrix}
x^2-a\sigma(a)^T & xa+ay \\
-\sigma(a)^Tx-y\sigma(a)^T & -\sigma(a)^Ta+y^2
\end{pmatrix}=0.$$
Since $0=\sigma(a)^Ta-y^2=\sigma(a)^Ta+\sigma(y)y$ and $\sigma(a)^Ta,\sigma(y)y\in A^\sigma_{\geq 0}$, $\sigma(a)^Ta=\sigma(y)y=0$. Because of semisimplicity of $A$, $a=0$, $y=0$. Moreover, since $x^2-a\sigma(a)^T=x^2=0$ and $\Mat_{n-1}(A)$ is semisimple, $x=0$, i.e. $X=0$. Therefore, $J=\{0\}$ and $\Mat_n(A)$ is semisimple.
\end{proof}

\subsection{Complex extensions of algebras}\label{CH-ext}

Let $A$ be an algebra over a real closed field $\K$. We denote $A_\CC:=A\otimes_\K\K_\CC$ and call $A_\CC$ the complexification of $A$. We extend $\sigma$ ``complex anti-linearly'' to an  anti-involution $\bar\sigma_\CC$ on $A_\CC$:
$$\bar\sigma_\CC(x+iy):=\sigma(x)-\sigma(y)i.$$ 

We embed $A_\CC$ into $\Mat_2(A)$ in the following way:
\begin{equation}\label{Upsilon}
\begin{matrix}
\Upsilon \colon& A_\CC & \to & \Mat_2(A)\\
& x+yi & \mapsto & \begin{pmatrix}
x & y \\
-y & x
\end{pmatrix}.
\end{matrix}
\end{equation}
This map is an injective homomorphism of $\K$-algebras. Moreover, the anti-involution $\bar\sigma$ corresponds to $\sigma^T$ under this embedding. From Corollary~\ref{sub-Herm_A} we obtain: 

\begin{cor}
Let $(A,\sigma)$ be a Hermitian algebra. The algebra $(A_\CC,\bar\sigma_\CC)$ is Hermitian if and only if $A$ is semisimple. If $A$ is semisimple then $A_\CC$ is semisimple.
\end{cor}

\begin{proof}
If $(A,\sigma)$ is Hermitian and semisimple, then $(A_\CC,\sigma)$ is Hermitian and semisimple by~\ref{Herm_complexification}.

If $A$ is not semisimple, i.e. the Jacobson radical $J(A)$ of $A$ is not trivial, then by~\ref{preHerm_prop}, for every $0\neq x\in J(A)$, $x^2=0$ and $\sigma(x)=-x$. Therefore, $\bar\sigma_\CC(ix)=ix\neq 0$, i.e. $ix\in A^{\bar\sigma_\CC}_\CC$ and $(ix)^2=0$. This contradicts to the property to be Hermitian for $(A_\CC,\bar\sigma_\CC)$.
\end{proof}

\begin{cor}\label{AC-Herm_A} Let $(A,\sigma)$ be Hermitian and semisimple.
\begin{itemize}
\item The group $U_{(A_\CC,\bar\sigma_\CC)}=\{z\in A_\CC\mid\bar\sigma_\CC(z)z=1\}$ is semi-algebraically connected (as a semi-algebraic deformation retract of a semi-algebraically connected space $A^\times_\CC$).
\item If $\K=\R$, then the group $U_{(A_\CC,\bar\sigma_\CC)}$ is a maximal compact subgroup of $A_\CC^\times$.
\end{itemize}
\end{cor}

\begin{rem} There is another anti-involution on $A_\CC$, namely the complex linear extension $\sigma_\CC$ of $\sigma$. Together with this anti-involution $A_\CC$ is never Hermitian because $$0=1-1=\sigma_\CC(1)\cdot 1+\sigma_\CC(i)\cdot i.$$
\end{rem}

Similarly to spectral theorems for $\sigma$-symmetric elements of a semisimple, Hermitian algebra $(A,\sigma)$, the spectral theorem for $\bar\sigma_\CC$-normal elements of $(A_\CC,\bar\sigma_\CC)$ can be proven.

\begin{teo}[Spectral theorem for normal elements, first version]\label{th:spec_normal1}
Let $(A_\CC,\bar\sigma_\CC)$ be Hermitian. For every $\bar\sigma_\CC$-normal element $b\in A_\CC$, there exist a unique $k\in\N$, unique elements $\lambda_1,\dots,\lambda_k\in\K_\CC$, all distinct, and a unique complete system of orthogonal idempotents $c_1,\dots,c_k\in \K_\CC[b]\cap A_\CC^{\bar\sigma_\CC}$ such that
$$b=\sum_{i=1}^k\lambda_ic_i.$$
We call this the \defin{spectral decomposition} of $b$.
\end{teo}

\begin{teo}[Spectral theorem for normal elements, second version]\label{th:spec_normal2}
Let $(A_\CC,\bar\sigma_\CC)$ be Hermitian. Suppose that the Jordan algebra $A_\CC^{\bar\sigma_\CC}$ has rank $n$. For every $\bar\sigma_\CC$-normal element $b\in A_\CC$ there exist a Jordan frame $(e_1,\dots,e_n)$ of $A_\CC^{\bar\sigma_\CC}$ and an $n$-tuple of elements  $(\lambda_1(b),\dots,\lambda_n(b))\in\K^n_\CC$ such that
$$b=\sum_{i=1}^n\lambda_i(b)e_i.$$
The tuple $(\lambda_1(b),\dots,\lambda_n(b))$ is uniquely defined by $b$ up to permutations. The elements $\lambda_1(b),\dots,\lambda_n(b)\in\K$ (with their multiplicities) are called the \defin{eigenvalues} of $b$. In particular, they do not depend on the Jordan frame $e_1,\dots,e_n\in A_\CC^{\bar\sigma_\CC}$.
\end{teo}

\begin{rem}
A Jordan frame $e_1,\dots,e_n\in A_\CC^{\bar\sigma_\CC}$ associated to a $\bar\sigma_\CC$-normal element $b\in A_\CC$ by Theorem~\ref{th:spec_normal2} is in general not unique and elements $e_1,\dots,e_n\notin \K_\CC[b]$, in contrast to the complete system of orthogonal idempotents from  Theorem~\ref{th:spec_normal1}.
\end{rem}

\begin{cor}
Let $(A,\sigma)$ be semisimple and Hermitian. Let $a_1,a_2\in A^\sigma$ that commute. Then there exist a complete orthogonal system of idempotents $c_1,\dots c_k\in A^\sigma$ and elements $\lambda_1,\dots,\lambda_k,\mu_1,\dots,\mu_k\in \K$ such that
$$a_1=\sum_{i=1}^k\lambda_ic_i,\;\;a_2=\sum_{i=1}^k\mu_ic_i.$$
Moreover, there exist an $a\in A^\sigma$ such that $c_1,\dots,c_k\in\K[a]$. In particular, $a_1,a_2\in \K[a]$.
\end{cor}

\begin{cor}
If $a\in A_\CC$ is $\bar\sigma_\CC$-normal, then $\sigma(a)\in\K[a]$.
\end{cor}

\begin{cor}\label{Spec_teo_U}
Let $(A_\CC,\bar\sigma_\CC)$ be Hermitian. If $a\in A^{-\bar\sigma_\CC}$, then all the eigenvalues of $a$ are purely imagine. If $a\in U_{(A_\CC,\bar\sigma_\CC)}$, then all the eigenvalues are in $S^1_\K=\{z\in\K_\CC\mid |z|=1\}$.
\end{cor}

\begin{cor}[Spectral theorem for $U_{(A,\sigma)}$]\label{Spec_teo_UR}
For every $u\in U_{(A,\sigma)}$, there exist unique $r\in\N\cup\{0\}$, $s\in\{0,1,2\}$ with $r+s>0$, unique systems of idempotents $c_1,\dots,c_r\in A^{\bar\sigma_\CC}_\CC$ and $c_1',\dots,c_s'\in A^{\sigma}$ such that $c_1,\dots,c_r,\bar c_1,\dots,\bar c_r,c_1',\dots,c_s'$ is a complete orthogonal system of idempotents of $A^{\bar\sigma_\CC}_\CC$ and unique elements $\zeta_1,\dots,\zeta_r\in S^1_\K$ with $\Imm(\zeta_i)>0$ for all $i\in\{1,\dots,r\}$ all distinct and $\varepsilon_1,\dots\varepsilon_s\in\{1,-1\}$ all distinct such that
$$u=\sum_{j=1}^r (\zeta_j c_j+\bar\zeta_j\bar c_j)+\sum_{j=1}^s\varepsilon_jc_j'.$$
\end{cor}

\begin{proof}
For an element $u\in U(A_\CC,\bar\sigma_\CC)$, $u\in U(A,\sigma)$ if and only if $u=\bar u$. We take the spectral decompositions of $u$ and $\bar u$:
$$u=\sum_{j=1}^k \zeta_j c_j,\;\;\bar u=\sum_{j=1}^k \bar\zeta_j\bar c_j$$
where all $\zeta_j\in S_1$, $(c_j)$ is a complete orthogonal system of idempotents of $A^{\bar\sigma_\CC}_\CC$.
Notice, all $\bar c_1,\dots,\bar c_k$ is a complete orthogonal system of idempotents of $A^{\bar\sigma_\CC}_\CC$ because $\bar c_i\bar c_j=\overline{c_ic_j}$.
If $u=\bar u$, then, because of uniqueness of the spectral decomposition, for every $j\in\{1,\dots, k\}$ there exists $j'\in\{1,\dots, k\}$ such that $\zeta_j c_j=\bar\zeta_j \bar c_{j'}$. There can be to cases:
\begin{enumerate}
\item $j=j'$ then $\zeta_j\in\R$ i.e. $\zeta_j\in\{1,-1\}$.
\item $j\neq j'$ then $c_{j'}=\bar c_j$ and $\zeta_j\notin \R$, i.e. $\Imm(\zeta_j)\neq 0$.
\end{enumerate}
Because all $\zeta_j$ are distinct, there can be at most one $j$ such that $\zeta_j=1$ and at most one $j$ with $\zeta_j=-1$. For such $j$, $c_j\in A^\sigma$. So we obtain
$$u=\sum_{j=1}^r (\zeta_j c_j+\bar\zeta_j\bar c_j)+\sum_{j=1}^s\varepsilon_jc_j'.$$
for appropriate $r,s\in\N\cup\{0\}$, $s\leq 2$. Finally, we can rename the idempotents so that $\Imm(\zeta_j)>0$ for all $j\in\{1,\dots r\}$.
\end{proof}

\begin{df}
The \defin{determinant map} on $U_{(A_\CC,\bar\sigma_\CC)}$ is given by:
$$\det(u)=\prod_{j=1}^n \zeta_j\in S^1_\K\subset \K_\CC$$
where $u=\sum_{j=1}^n \zeta_je_j$ for some Jordan frame $e_1,\dots,e_n\in A^{\bar\sigma_
\CC}_\CC$.
\end{df}

\begin{prop}
Let $\K=\R$. The fundamental group of $U_{(A_\CC,\bar\sigma_\CC)}$ contains a copy of $(\Z,+)$.
\end{prop}

\begin{proof}
The determinant map
$$\det\colon U_{(A_\CC,\bar\sigma_\CC)}\to S^1$$ is continuous and surjective. It induces the homomorphism of fundamental groups:
$$(\det)_*\colon\pi_1(U_{(A_\CC,\bar\sigma_\CC)}, 1)\to \pi_1(S^1,1).$$
This homomorphism is surjective because the curve $u(t)=e^{it}e_1+\sum_{j=2}^ne_j$ for $t\in[0,2\pi]$ for some Jordan frame $e_1,\dots,e_n$ maps to the curve $\det_*u(t)=e^{it}$ which generates $\pi_1(S^1,1)$. Therefore $u(t)$ generates in $\pi_1(U_{(A_\CC,\bar\sigma_\CC)}, 1)$ a subgroup that is isomorphic to $(\Z,+)$.
\end{proof}

\begin{teo}\label{Transit_action_A_CC}
Let $(A,\sigma)$ be a semisimple Hermitian algebra and $(A_\CC,\sigma_\CC)$ be its complexification. Assume that the Jordan algebra $A_\CC^{\bar\sigma_\CC}$ is simple.
The following action of $A_\CC^\times$ on $(A_\CC^{\sigma_\CC})^\times$
$$\begin{array}{cccl}
\psi_\CC\colon & A_\CC^\times\times (A_\CC^{\sigma_\CC})^\times & \to & (A_\CC^{\sigma_\CC})^\times \\
 & (a,b) & \mapsto & \sigma(a)ba
\end{array}$$
is transitive.
\end{teo}

\begin{proof} We will show that for every $b\in (A_\CC^{\sigma_\CC})^\times$ there exists an $a\in A_\CC^\times$ such that $\psi_\CC(a,1)=b$. We take $b\in (A_\CC^{\sigma_\CC})^\times$ and consider its polar decomposition of $b=ub'$ where $u\in U_{(A_\CC,\bar\sigma_\CC)}$, $b'\in (A_\CC^{\bar\sigma_\CC})^\times$. Using the spectral theorem we obtain:
$b'=\sum_{i=1}^k\lambda_ic_i$
where $(c_i)_{i=1}^k$ is a Jordan frame of $A_\CC^{\bar\sigma_\CC}$ and $\lambda_i>0$ for all $i\in\{1,\dots, k\}$.

We fix a Jordan frame $(x_i)_{i=1}^k$ of $A^\sigma$. This is also a Jordan frame of $A_\CC^{\bar\sigma_\CC}$. By~\ref{transit_JF}, the group $U_{(A_\CC,\bar\sigma_\CC)}$ acts transitively on Jordan frames of $A_\CC^{\bar\sigma_\CC}$. Therefore, there exists $u'\in U_{(A_\CC,\bar\sigma_\CC)}$ such that $x_i=\bar\sigma(u')c_i u'$ for all $i\in\{1,\dots,k\}$. Further
$$\sigma_\CC(u')bu'=\sigma_\CC(u')u\bar\sigma_\CC(u')^{-1}\bar\sigma_\CC(u')b'u'=\sigma_\CC(u')u\bar\sigma_\CC(u')^{-1}\sum_{i=1}^k\lambda_ix_i=:u''b''$$
where $u''=\sigma_\CC(u')u\bar\sigma_\CC(u')^{-1}\in U_{(A_\CC,\bar\sigma_\CC)}$, $b''=\sum_{i=1}^k\lambda_ix_i\in A^\sigma$. Since $u''b''=\sigma(u')bu'\in A_\CC^{\sigma_\CC}$, $$u''b''=\sigma_\CC(u''b'')=b''\sigma_\CC(u'')=b''(\bar u'')^{-1}.$$ Therefore, $b''=u''b''\bar u''$. By induction, we obtain $b''=(u'')^n b'' (\bar u'')^n$ for all $n\in\Z$, or equivalently $(u'')^nb''=b'' (\bar u'')^{-n}$.

Now assume $\K=\R$. Since $(u'')^nb''=b'' (\bar u'')^{-n}$ holds for every $n\in\Z$, the following equality $f(u'')b''=b''f((\bar u'')^{-1})$ holds for every function $f$ that can be expressed as a convergent Taylor series at $u''$. Since $u''\in U_{(A_\CC,\bar\sigma_\CC)}$, $u''\neq 0$. By~\ref{Spec_teo_U}, there exists $w\in U_{(A_\CC,\bar\sigma_\CC)}$ such that $(u'')=w^2$. So we can take as $f$ the branch of the square root such that $f(u'')=w$. Then we obtain $wb''=b''(\bar w)^{-1}$. Since any branch of the square root is a semi-algebraic function, using the Tarski–Seidenberg transfer principle (see~\cite[Proposition~5.2.3]{real_algebraic}), this holds for every real closed field $\K$.

This implies:
$$u''b''=wb''\bar w^{-1}=(wb^{\frac{1}{2}})b^{\frac{1}{2}}\bar w^{-1}=\sigma_\CC(b^{\frac{1}{2}}\bar w^{-1})b^{\frac{1}{2}}\bar w^{-1}=\psi_\CC(b^{\frac{1}{2}}\bar w^{-1},1).$$
Finally,
$b=\sigma_\CC((u')^{-1})u''b''(u')^{-1}=\psi_\CC(b^{\frac{1}{2}}\bar w^{-1}(u')^{-1},1)$.
\end{proof}

\subsection{Quaternionic extensions of algebras}\label{Quatern_ext}

If $A$ is an algebra over a real closed field $\K$, then we call $A_\HH=A\otimes_\K\K_\HH$ the \defin{quaternionification} of $A$.

Sometimes, to make a construction more precise, we will
write $\K_\HH\{i,j,k\}$ to emphasize the imaginary units. The multiplication rule is then $ij=-ji=k$. Sometimes, we will also write $\K_\CC\{\kappa\}$ to emphasize the imaginary unit $\kappa$ of the complexified field $\K_\CC$.

If $A_\CC=A\otimes_\K\K_\CC$ is the complexification of some real algebra $A$, then it can be
embedded into $A\otimes_\K\K_\HH$ in many different ways. If we write $A_\CC:=A\otimes_\K\K_\CC\{i\}$,
$A_\HH:=A\otimes_\K\K_\HH\{i,J,K\}$, it means that $A_\CC$ is embedded into $A_\HH$ by the map induced by
the identification $A_\CC\ni i\mapsto i\in A_\HH$.

Let $A$ be a $\K_\CC$-algebra, $A_0\subset A$ be $\K$-subalgebra of $A$ with the property that there is a central element $I\in Z(A)$ such that
$I^2=-1$ and $A=A_0\oplus A_0I$. Then we say that $A_0$ is a \defin{real locus} of $A$ with respect to the
imaginary unit $I$. In this case, $A$ is isomorphic to $A_0\otimes_\K\K_\CC\{I\}$ as $\K_\CC\{I\}$-algebra. We
take the following $\K_\HH$-algebra:
$$\HH[A,A_0,I,J]:=A_0\otimes_\K\K_\HH\{I,J,K\}.$$
The algebra $A$ sits inside $\HH[A,A_0,I,J]$ as described above.
\begin{df}
We call $\HH[A,A_0,I,J]$ the \defin{quaternionification} of $A$ with respect to the real locus
$A_0$ and the imaginary unit $I$.
\end{df}

The algebra $A_\HH$ can be embedded into $\Mat_2(A_\CC)$ in the following way:
\begin{equation}\label{Upsilon_H}
\begin{matrix}
\Upsilon_\HH \colon& A_\HH & \to & \Mat_2(A_\CC)\\
& x+yj & \mapsto & \begin{pmatrix}
x & y \\
-\bar y & \bar x
\end{pmatrix}.
\end{matrix}
\end{equation}
This map is an injective homomorphism of $\CC\{i\}$-algebras. Moreover, the anti-involution $\sigma_1$ on $A_\HH$ defined as follows:
$$\sigma_1(x+yj)=\bar\sigma_\CC(x)-\sigma_\CC(y)j$$
for $x,y\in A_\CC$, corresponds under this embedding to $\bar\sigma^T$. By Corollary~\ref{sub-Herm_A}, we obtain:

\begin{cor}\label{AH-Herm_A} Let $(A,\sigma)$ be Hermitian and semisimple. 
\begin{itemize}
    \item The algebra $(A_\HH,\sigma_1)$ is Hermitian and semisimple.
    \item If $\K=\R$, then the group $U_{(A_\HH,\sigma_1)}=\{z\in A_\CC\mid \sigma_1(z)z=1\}$ is a maximal compact subgroup of $A_\HH^\times$.
\end{itemize}
\end{cor}

\begin{rem}There is another anti-involution $\sigma_0$ on $A_\HH$ defined as follows:
$$\sigma_0(x+yj)=\sigma_\CC(x)+\bar\sigma_\CC(y)j$$
for $x,y\in A_\CC$. The algebra $(A_\HH,\sigma_0)$ is never Hermitian.
\end{rem}

\subsection{Some properties of real Hermitian algebras}

In this section, we assume that $(A,\sigma)$ is a semisimple real Hermitian algebra. We will state some properties that depend on the fact that $\R$ is locally compact, contrary to the other real closed fields.

\begin{prop}\label{Max_Comp_A}
The group $U_{(A,\sigma)}$ is a maximal compact subgroup of $A^\times$.
\end{prop}

\begin{proof}
The group  $U_{(A,\sigma)}\subset\Isom(\beta)$ is compact by~\ref{cb_sets}. By~\ref{U_def_retr}, $U_{(A,\sigma)}$ is a deformation retract of $A^\times$. Hence, it is a maximal compact subgroup of $A^\times$.
\end{proof}

\begin{cor}
$\UU_n(A,\sigma)$ is a maximal compact subgroup of $\Mat_n^\times(A)$
\end{cor}

\begin{cor}[Polar decomposition, second version] The following map:
$$\begin{matrix}
\bpol\colon & U_{(A,\sigma)}\times A^\sigma_{\geq 0} & \to & A \\
& (u,b) & \mapsto & ub.
\end{matrix}$$
is proper, surjective.
\end{cor}

\begin{proof}
Since the group $U_{(A,\sigma)}$ is compact, we can take a closure in the polar decomposition and the map stays surjective. Let $K\subset A$ be a compact subset, then
$$\bpol^{-1}(K)\subseteq\{(\sigma(a)a)^{\frac{1}{2}}\mid a\in K\}\times U_{(A,\sigma)}$$
is compact.
\end{proof}

\begin{prop} If $A$ is Hermitian, the map $\theta\colon A\to A^\sigma_{\geq 0}$ is proper.
\end{prop}

\begin{proof}
Let $K\subset A^\sigma_{\geq 0}$ be a compact subset. Then
$$\theta^{-1}(K)=\{ub_+^{\frac{1}{2}}\mid u\in U_{(A,\sigma)}, b_+\in K\}=
\bpol(U_{(A,\sigma)}\times K).$$
Since $U_{(A,\sigma)}\times K$ is compact in $U_{(A,\sigma)}\times A^\sigma_{\geq 0}$ and $\bpol$ is continuous, $\theta^{-1}(K)$ is compact.
\end{proof}

\begin{prop}\label{comp_disc}
Let $(A,\sigma)$ be Hermitian. The domain
$$D(A,\sigma):=\{a\in A\mid 1-\sigma(a)a\in A^\sigma_{\geq 0}\}$$
is compact.

In particular, for every vector subspace $V$ of $A$, the domain
$$D(V,\sigma):=\{a\in V \mid 1-\sigma(a)a\in A^\sigma_{\geq 0}\}=D\cap V$$
is compact.
\end{prop}

This follows directly from~\ref{cb_sets}

\section{Symplectic groups over noncommutative algebras}\label{sec:symplectic}
In this section we introduce symplectic groups over unital rings with anti-involution. 

\subsection{Sesquilinear forms on \texorpdfstring{$A$}{A}-modules and their groups of symmetries}

Let $A$ be a unital associative ring with an anti-involution $\sigma$.

\begin{df}\label{osp}
A \defin{$\sigma$-sesquilinear form} $\omega$ on a right $A$-module $V$ is a map
$$\omega\colon V\times V\to A$$
such that
$$\omega(x+y,z)=\omega(x,z)+\omega(y,z)$$
$$\omega(x,y+z)=\omega(x,y)+\omega(x,z)$$
$$\omega(x_1r_1,x_2r_2)=\sigma(r_1)\omega(x_1,x_2)r_2$$

We denote by $$\Aut(\omega):=\{f\in\Aut(V)\mid \forall x,y\in V:\omega(f(x),f(y))=\omega(x,y)\}$$ the group of symmetries of $\omega$. We also define the corresponding Lie algebra:
$$\End(\omega):=\{f\in\End(V)\mid \forall x,y\in V:\omega(f(x),y)+\omega(x,f(y))=0\}$$
with the usual Lie bracket $[f,g]=fg-gf$.
\end{df}

We now set  $V=A^2$. We view $V$ as the set of columns and endow it with the structure of a right $A$-module.

\begin{df}$ $
We make the following definitions: 
\begin{enumerate}
\item A pair $(x,y)$ for $x,y\in A^2$ is called \defin{basis} of $A^2$ if for every $z\in A^2$ there exist $a,b\in A$ such that $z=xa+yb$.
\item The element $x\in A^2$ is called \defin{regular} if there exists $y\in A^2$ such that $(x,y)$ is a basis of $A^2$.
\item $l\subseteq A^2$ is called a \defin{line} if $l=xA$ for a regular $x\in A^2$. We denote the space of lines of $A^2$ by $\PP(A^2)$.
\item Two regular elements $x,y\in A^2$ are called \defin{linearly independent} 
if $(x,y)$ is a basis of $A^2$.
\item Two lines $l,m$ are called \defin{transverse} if $l=xA$, $m=yA$ for linearly independent $x,y\in A^2$.
\item An element $x\in A^2$ is called \defin{isotropic} with respect to $\omega$ if $\omega(x,x)=0$. The set of all isotropic regular elements of $(A^2,\omega)$ is denoted by $\Is(\omega)$.
\item A line $l$ is called isotropic if $l=xA$ for an regular isotropic $x\in A^2$. The set of all isotropic lines of $(A^2,\omega)$ is denoted by $\PP(\Is(\omega))$.
\end{enumerate}
\end{df}

\begin{df} $ $ 
We consider a form $\omega$ and say 
\begin{enumerate}
\item The form 
$\omega$ is \defin{non-degenerate} if for every regular $x\in A^2$ there exists a $y\in A^2$ such that $\omega(x,y)\in A^\times$.

\item The form is \defin{$\sigma$-symmetric} if $\omega(x_2,x_1)=\sigma(\omega(x_1,x_2))$ for all $x_1,x_2\in A^2$.

\item The form is \defin{$\sigma$-skew-symmetric } if $\omega(x_2,x_1)=-\sigma(\omega(x_1,x_2))$ for all $x_1,x_2\in A^2$.

\item When $A$ is  Hermitian, a $\sigma$-symmetric form is called \defin{$\sigma$-inner product} if $\omega(x,x)\in A^\sigma_+$ for all regular $x\in A^2$.
\end{enumerate}
\end{df}

We can now introduce the symplectic group $\Sp_2(A,\sigma)$ over $(A,\sigma)$.

\begin{df}
Let $(A,\sigma)$ be a unital ring with anti-involution. 
Let $\omega(x,y):=\sigma(x)^T\Omega y$ with $\Omega=\Ome{1}$. The form $\omega$ is called the \defin{standard symplectic form} on $A^2$.
The group $\Sp_2(A,\sigma):=\Aut(\omega)$ is the \defin{symplectic group} $\Sp_2$ over $(A,\sigma)$. Its Lie algebra is $\spp_2(A,\sigma):=\End(\omega)$.
\end{df}

We have
$$\Sp_2(A,\sigma)=\left\{\begin{pmatrix}
    a & b \\
    c & d
    \end{pmatrix}\midwd \sigma(a)c,\,\sigma(b)d\in A^\sigma,\,\sigma(a)d-\sigma(c)b=1\right\}\subseteq \Mat_2^\times(A)$$

$$\spp_2(A,\sigma)=\left\{\begin{pmatrix}
    x & z \\
    y & -\sigma(x)
    \end{pmatrix}\midwd x\in A,\;y,z\in A^\sigma\right\}\subseteq \Mat_2(A).$$

From now on, we assume $\omega(x,y):=\sigma(x)^T\Omega y$ on $A^2$.

\begin{df}
A basis $(x,y)$ of $A^2$ is called \defin{symplectic} if $x,y$ are isotropic and $\omega(x,y)=1$.
\end{df}

\begin{prop}\label{regular-extend}
For every basis $(x,y)$ of $A^2$ and for every $z\in A^2$ there exist unique $a,b\in A$ such that $z=xa+yb$. Moreover, for every regular $x\in A^2$, the map
$$\begin{matrix}
A & \to & xA\\
a & \mapsto & xa
\end{matrix}$$
is an isomorphism of right $A$-modules.
\end{prop}

\begin{proof} Take a basis $(x,y)$ of $A^2$.
Consider the following $A$-homomorphism of right $A$-modules:
$$\begin{matrix}
A^2 & \to & A^2\\
(a,b) & \mapsto & z= xa+yb.
\end{matrix}$$
This is also a surjective $\R$-homomorphism of vector spaces of the same dimension. Therefore, it is injective, i.e. $(a,b)$ is uniquely defines by $z$. The restriction of this homomorphism to $A\times\{0\}$ is an isomorphism $A\to xA$ of right $A$-modules.
\end{proof}


\begin{prop}
The form $\omega$ is non-degenerate.
\end{prop}

\begin{proof}
Let $x=(x_1,x_2)^T\in A^2$ regular. We want to find $y\in A^2$ such that $\omega(y,x)=1$. Since $x$ is regular, there exists $x'=(x_1',x_2')^T\in A^2$ such that $(x,x')$ is a basis. That means that the matrix
$X:=\begin{pmatrix}
x_1 & x_1'\\
x_2 & x_2'
\end{pmatrix}$ is invertible, i.e. there exists the inverse matrix
$X^{-1}=\begin{pmatrix}
a_1 & a_2\\
a_1' & a_2'
\end{pmatrix}$. Therefore, $a_1x_1+a_2x_2=1$. We take $y:=(\sigma(a_2),-\sigma(a_1))^T$, then
$$\omega(y,x)=(a_2,-a_1)\Ome{1}\Clm{x_1}{x_2}=(a_1,a_2)\Clm{x_1}{x_2}=1.$$
So $\omega$ is non-degenerate.
\end{proof}

\begin{prop}\label{krit_isotr}
An element $x=(x_1,x_2)^T\in A^2$ is isotropic if and only if $\sigma(x_1)x_2\in A^\sigma$.
\end{prop}

\begin{proof}
The proof follows by direct computation.
\end{proof}

\begin{prop}
If $x,y\in A^2$ are isotropic and $\omega(x,y)=1$, then $(x,y)$ is a basis.
\end{prop}

\begin{proof}
Let $x,y\in A^2$ are isotropic and $\omega(x,y)=1$. Consider the map
$$\begin{matrix}
A^2 & \to & A^2\\
(a,b) & \mapsto & xa+yb.
\end{matrix}$$
To see that this map is an isomorphism, it is enough to check that it is injective. Assume $xa+yb=0$ for some $a,b\in A$, then
$$0=\omega(x,xa+yb)=\omega(x,y)b=b,$$
$$0=\omega(y,xa+yb)=-\omega(x,y)a=-a.$$
So $a=b=0$.
\end{proof}

\begin{cor}
$$\Sp_2(A,\sigma)=\left\{\begin{pmatrix}
a & b \\
c & d
\end{pmatrix}\midwd \left(\Clm{a}{c},\Clm{b}{d}\right)\text{ is a symplectic basis}
\right\} $$
\end{cor}

\begin{prop}
Let $x\in A^2$ regular isotropic, $y\in A^2$ and $\omega(x,y)\in A^\times$. Then $(x,y)$ is a basis of $A^2$. In particular, $y$ is regular.
\end{prop}

\begin{proof}
To see that $(x,y)$ is a basis, it is enough to check that the map
$$\begin{matrix}
A^2 & \to & A^2\\
(a,b) & \mapsto & xa+yb
\end{matrix}$$
is injective. Assume $xa+yb=0$ for some $a,b\in A$, then
$$0=\omega(x,xa+yb)=\omega(x,y)b.$$
Since $\omega(x,y)\in A^\times$, $b=0$.

The element $x\in A^2$ is regular, therefore, by Proposition~\ref{regular-extend}, if $xa=0$, then $a=0$. So, we obtain $(a,b)=(0,0)$, i.e. the map above is an isomorphism.
\end{proof}

\begin{prop}
For every regular isotropic $x\in A^2$, there exists an isotropic $y\in A^2$ such that $(x,y)$ is a symplectic basis.
\end{prop}

\begin{proof}
Since $\omega$ is non-degenerate, there exists $y'\in A^2$ such that $\omega(x,y')\in A^{\times}$ and $(x,y')$ is a basis. We take $y'':=y'-\frac{x}{2}\omega(y',x)^{-1}\omega(y,y)$, then
$$\omega(y,y)=\omega(y',y')-\frac{1}{2}\omega(y',x)\omega(y',x)^{-1}\omega(y,y)-$$
$$-\frac{1}{2}\sigma(\omega(y',x)^{-1}\omega(y,y))\omega(x,y')=0.$$
Since $\omega(x,y')=\omega(x,y'')$, if we take $y:=y''\omega(x,y')^{-1}$, we obtain $\omega(x,y)=1$ and $x,y$ are isotropic, so $(x,y)$ is a symplectic basis.
\end{proof}

\begin{cor}
The group $\Sp_2(A,\sigma)$ acts transitively on regular isotropic elements of $(A^2,\omega)$.
\end{cor}

\begin{proof} If $x=(x_1,x_2)^T\in A^2$ is regular isotropic, then there exists $y=(y_1,y_2)^T\in A^2$ regular isotropic such that $(x,y)$ is a symplectic basis. Then
$$g:=\begin{pmatrix}
x_1 & y_1 \\
x_2 & y_2
\end{pmatrix}\in \Sp_2(A,\sigma)$$
and $g(1,0)^T=x$.
\end{proof}


Let now $\K$ be a real closed field and $(A,\sigma)$ be a Hermitian $\K$-algebra. We consider the following sesquilinear form $b\colon A^2\times A^2 \to A$, $b(x,y):=\sigma(x)^Ty$ for $x,y\in A^2$.

\begin{prop}
$b(x,x)=0$ for an element $x\in A^2$ if and only if $x=0$.
\end{prop}

\begin{proof}
For $x=(x_1,x_2)^T\in A^2$, $b(x,x)=\sigma(x_1)x_1+\sigma(x_2)x_2$ with $\sigma(x_1)x_1,\sigma(x_2)x_2\in A^\sigma_{\geq 0}$. If $b(x,x)=0$, then $\sigma(x_1)x_1=\sigma(x_2)x_2=0$.
\end{proof}

\begin{prop}\label{regular-sesq}
If $x\in A^2$ is regular then $b(x,x)\in A^\sigma_+$.
\end{prop}

\begin{proof}
Let $x=(x_1, x_2)^T$ and assume $b(x,x)=\sigma(x)x=\sigma(x_1)x_1+\sigma(x_2)x_2\in A^\sigma_{\geq 0}$ is not invertible. By the Lemma~\ref{zero-divisor}, there exists $c\in\R[a]\subseteq A^\sigma$ such that $0=cb(x,x)c=b(xc,xc)$. Therefore, $xc=0$ and so by Proposition~\ref{regular-extend}, the map $A\to xA$, $a\mapsto xa$ is not invertible. In particular, $x$ is not regular.
\end{proof}

\begin{prop}[Gram-Schmidt orthogonalization]\label{Gram}
Let $x\in A^2$ be regular, then there exists $a\in A$ and $y\in A^2$ such that $(xa,y)$ is an \defin{orthonormal with respect to $b$ basis} of $A^2$, i.e. $b(xa,xa)=b(y,y)=1$, $b(xa,y)=0$.
\end{prop}

\begin{proof}
Since $x$ is regular, $b(x,x)\in A^\sigma_+$. Take $a=b(x,x)^{-\frac{1}{2}}$, then $b(xa,xa)=1$. Moreover, there exists $z\in A^2$ such that $(x,z)$ is a basis of $A^2$. Consider $y':=z-xab(xa,z)$, obviously $(xa,y')$ is a basis as well. Then $b(xa,y')=b(xa,z-xab(xa,z))=b(xa,z)-b(xa,xa)b(xa,z)=0$. Since $y'$ is regular and so $b(y',y')\in A^\sigma_+$, we can take $y:=y'b(y',y')^{-1}$.
\end{proof}

\subsection{Maximal compact subgroup of \texorpdfstring{$\Sp_2(A,\sigma)$}{Sp2(A,sigma)} over Hermitian algebras}\label{Max_Comp_R}

We now assume $\K=\R$ and $(A,\sigma)$ to be Hermitian, semisimple $\R$-algebra. We describe a maximal compact subgroup of $\Sp_2(A,\sigma)$.

\begin{df}
We denote:
$$\UU_2(A,\sigma):=\{M\in\Mat_2(A)\mid \sigma(M)^TM=\Id\};$$
$$\KSp_2(A,\sigma):=\Sp_2(A,\sigma)\cap\UU_2(A,\sigma).$$
\end{df}

\begin{prop}\label{KSp-shape}
$\KSp_2(A,\sigma)=\left\{\begin{pmatrix}
a & b \\
-b & a
\end{pmatrix}\left|\,\begin{matrix}
\sigma(a)a+\sigma(b)b=1\\
\sigma(a)b-\sigma(b)a=0
\end{matrix}\right.,\,a,b\in A\right\}.$\end{prop}

\begin{proof} Take
$M:=\begin{pmatrix}
a & b \\
c & d
\end{pmatrix}\in\KSp_2(A,\sigma)$. On one hand, $M\in\Sp_2(A,\sigma)$, therefore,
$$M^{-1}= -\Om \sigma(M)^T\Om=\begin{pmatrix}
\sigma(d) & -\sigma(b) \\
-\sigma(c) & \sigma(a)
\end{pmatrix}.$$
On the other hand, $M\in\UU_2(A,\sigma)$, therefore,
$$M^{-1}=\sigma(M)^T=\begin{pmatrix}
\sigma(a) & \sigma(c) \\
\sigma(b) & \sigma(d)
\end{pmatrix}.$$
So we obtain, $a=d$ and $b=-c$.
\end{proof}

\begin{teo}\label{maxcomp-Sp_R} The group 
$\KSp_2(A,\sigma)$ is a maximal compact subgroup of $\Sp_2(A,\sigma)$.
\end{teo}

\begin{proof} The group $\KSp_2(A,\sigma)$ is a closed subgroup of the compact group $\UU_2(A,\sigma)$, so it is compact.

Now, we show that $\KSp_2(A,\sigma)$ is a maximal compact subgroup of $\Sp_2(A,\sigma)$. For this let $K$ be a compact subgroup containing $\KSp_2(A,\sigma)$ as a proper subgroup. We consider the following decomposition of $\spp_2(A,\sigma)$:
$$\spp_2(A,\sigma)=\ksp_2(A,\sigma)\oplus S$$
where
$$\ksp_2(A,\sigma)=\Lie(\KSp_2(A,\sigma))=\left\{
\begin{pmatrix}
a & b \\
-b & a
\end{pmatrix}
\midwd \sigma(a)=-a\in A, b\in A^\sigma
\right\},$$
$$S=\left\{
\begin{pmatrix}
c & d \\
d & -c
\end{pmatrix}
\midwd c,d\in A^\sigma
\right\}.$$
By our assumption, $\Lie(K)$ contains $\ksp_2(A,\sigma)$ and has nontrivial intersection with $S$. Take some matrix
$$\begin{pmatrix}
c & d \\
d & -c
\end{pmatrix}\in\Lie(K)\cap S,\;c,d\in A^\sigma.$$
The matrix
$$\begin{pmatrix}
0 & d \\
-d & 0
\end{pmatrix}\in\ksp_2(A,\sigma)\subset\Lie(K),$$
therefore,
$$\begin{pmatrix}
c & 2d \\
0 & -c
\end{pmatrix}=\begin{pmatrix}
c & d \\
d & -c
\end{pmatrix}+\begin{pmatrix}
0 & d \\
-d & 0
\end{pmatrix}\in\Lie(K)\bs\ksp_2(A,\sigma).$$
Using the exponential map of $\spp_2(A,\sigma)$ restricted to $\Lie(K)$, we obtain that there exits a matrix $M:=\begin{pmatrix}
g & gx \\
0 & g^{-1}
\end{pmatrix}\in K\bs\KSp_2(A,\sigma)$ where $g=\exp(c)\in A^\sigma$, $x\in A^\sigma$. We consider the spectral decomposition of $g=\sum_{i=1}^k\lambda_ic_i$ with $\lambda_i>0$ and $(c_i)_{i=1}^k$ a complete orthogonal system of idempotents. Take a sequence $\{M^r\}\subseteq K$. Then $$M^r_{11}=g^k=\sum_{i=1}^k\lambda_i^rc_i,\;
M^r_{22}=g^{-k}=\sum_{i=1}^k\lambda_i^{-r}c_i.$$
Assume that there exists $s\in\{1,\dots,k\}$ such that $\lambda_s\neq\pm 1$. Then either $0<|\lambda_s|<1$ or $0<|\lambda_s^{-1}|<1$. Without loss of generality, we may assume $0<|\lambda_s|<1$. Since $K$ is compact, $\{M^r\}\subseteq K$ has a convergent subsequence $\{M^{r_j}\}\subseteq K$:
$$\lim M^{r_j}_{11}=\lim\sum_{i=1}^k\lambda_i^{r_j}c_i=\sum_{i=1}^k\hat\lambda_ic_i$$
where $\hat\lambda_i=\lim\lambda_i^{r_j}$. But $\hat\lambda_s=\lim\lambda_s^{r_j}=0$ for any subsequence $\{r_j\}$. Therefore $\lim M^{r_j}_{11}$ is not invertible and so $\lim M^{r_j}$ is not invertible as well. Therefore, all $\lambda_i=\pm 1$ and $g^2=1$. The element $L:=\begin{pmatrix}
g & 0\\
0 & g^{-1}
\end{pmatrix}\in\KSp_2(A,\sigma)\subset K$. Then $ML=\begin{pmatrix}
1 & x \\
0 & 1
\end{pmatrix}\in K$. Take $(ML)^r=\begin{pmatrix}
1 & rx \\
0 & 1
\end{pmatrix}\in K$. This sequence does not have any convergent subsequence unless $x=0$. So we get $M=L\in\KSp_2(A,\sigma)$. This contradicts the assumption $M\notin\KSp_2(A,\sigma)$ and we obtain that $\KSp_2(A,\sigma)$ is a maximal compact subgroup of $\Sp_2(A,\sigma)$.
\end{proof}

\begin{cor}
The embedding $$\Upsilon \colon A_\CC  \to  \Mat_2(A)$$ (see equation~\ref{Upsilon}) maps $U_{(A_\CC,\bar\sigma_\CC)}$ isomorphically to $\KSp_2(A,\sigma)$. In particular, the fundamental group of $\Sp_2(A,\sigma)$ is infinite.
\end{cor}

\subsection{Maximal compact subgroups of \texorpdfstring{$\Sp_2(A_\CC,\sigma_\CC)$}{Sp2(A,sigma)} over complexified algebras}\label{Max_Comp_C}

Let $\K=\R$ and $(A,\sigma)$ be a Hermitian, semisimple $\R$-algebra. In Section~\ref{CH-ext} we have seen that $(A_\CC,\bar\sigma_\CC)$ is also Hermitian and semisimple and $(A_\CC,\sigma_\CC)$ is never Hermitian.

\begin{df}\label{KSpc}
We set 
$$\KSp^c_2(A_\CC,\sigma_\CC):=\Sp_2(A_\CC,\sigma_\CC)\cap\UU_2(A_\CC,\bar\sigma_\CC).$$
\end{df}

\begin{prop}\label{KSpc-shape} The group $\KSp^c_2(A_\CC,\sigma_\CC)$ is given by 
$$\KSp^c_2(A_\CC,\sigma_\CC)=\left\{\begin{pmatrix}
a & b \\
-\bar b & \bar a
\end{pmatrix}\left|\;\begin{matrix}
\bar\sigma_\CC(a)a+\sigma_\CC(b)\bar b=1\\
\bar\sigma_\CC(a)b-\sigma_\CC(b)\bar a=0
\end{matrix}\right.,\;a,b\in A_\CC\right\}.$$
\end{prop}

\begin{proof} Take
$M:=\begin{pmatrix}
a & b \\
c & d
\end{pmatrix}\in\KSp^c_2(A_\CC,\sigma_\CC)$. On one hand, $M\in\Sp_2(A_\CC,\sigma_\CC)$, therefore,
$$M^{-1}= -\Om \sigma_\CC(M)^T\Om=\begin{pmatrix}
\sigma_\CC(d) & -\sigma_\CC(b) \\
-\sigma_\CC(c) & \sigma_\CC(a)
\end{pmatrix}.$$
On the other hand, $M\in\UU_2(A_\CC,\bar\sigma_\CC)$, therefore,
$$M^{-1}=\bar\sigma(M)_\CC^T=\begin{pmatrix}
\bar\sigma_\CC(a) & \bar\sigma_\CC(c) \\
\bar\sigma_\CC(b) & \bar\sigma_\CC(d)
\end{pmatrix}.$$
So we obtain, $d=\bar a$ and $c=-\bar b$.
\end{proof}

\begin{teo}\label{maxcomp-Sp_C} The group 
$\KSp^c_2(A_\CC,\sigma_\CC)$ is a maximal compact subgroup of $\Sp_2(A_\CC,\sigma_\CC)$.
\end{teo}

\begin{proof} The proof follows the same strategy as the proof of Theorem~\ref{maxcomp-Sp_R}.

By definition, $\KSp^c_2(A_\CC,\sigma_\CC)$ is closed subgroup of $\UU_2(A_\CC,\bar\sigma_\CC)$ which is compact, so $\KSp^c_2(A_\CC,\sigma_\CC)$ is compact as well.

To show that $\KSp^c_2(A_\CC,\sigma_\CC)$ is a maximal compact subgroup of $\Sp_2(A_\CC,\sigma_\CC)$, we assume $K$ to be a compact subgroup of $\Sp_2(A_\CC,\sigma_\CC)$ containing $\KSp^c_2(A_\CC,\sigma_\CC)$ as a proper subgroup. We consider the following decomposition of $\spp_2(A_\CC,\sigma_\CC)$:
$$\spp_2(A_\CC,\sigma_\CC)=\ksp^c_2(A_\CC,\sigma_\CC)\oplus S$$
where
$$\ksp^c_2(A_\CC,\sigma_\CC)=\Lie(\KSp^c_2(A_\CC,\sigma_\CC))=\left\{
\begin{pmatrix}
a & b \\
-\bar b & \bar a
\end{pmatrix}
\midwd \bar\sigma_\CC(a)=-a\in A_\CC, b\in A^{\bar\sigma}_\CC
\right\},$$
$$S=\left\{
\begin{pmatrix}
c & d \\
\bar d & -\bar c
\end{pmatrix}
\midwd c,d\in A^{\bar\sigma}_\CC
\right\}.$$
By our assumption, $\Lie(K)$ contains $\ksp^c_2(A_\CC,\sigma_\CC)$ and has nontrivial intersection with $S$. So we can take a matrix
$$\begin{pmatrix}
c & d \\
\bar d & -\bar c
\end{pmatrix}\in\Lie(K)\cap S,\; c,d\in A^{\bar\sigma}_\CC.$$ Then 
$$\begin{pmatrix}
0 & d \\
-\bar d & 0
\end{pmatrix}\in\ksp^c_2(A_\CC,\sigma_\CC)\subset\Lie(K),$$
therefore,
$$\begin{pmatrix}
c & 2d \\
0 & -\bar c
\end{pmatrix}=\begin{pmatrix}
c & d \\
\bar d & -\bar c
\end{pmatrix}+\begin{pmatrix}
0 & d \\
-\bar d & 0
\end{pmatrix}\in\Lie(K)\bs\ksp^c_2(A_\CC,\sigma_\CC).$$
Using the exponential map of $\spp_2(A_\CC,\sigma_\CC)$ restricted to $\Lie(K)$, we obtain that there exits a matrix $M:=\begin{pmatrix}
g & gx \\
0 & \bar g^{-1}
\end{pmatrix}\in K\bs\KSp^c_2(A_\CC,\sigma_\CC)$ where $g=\exp(c)\in A^{\bar\sigma}_\CC$, $x\in A^{\bar\sigma}_\CC$. We now consider the spectral decomposition of $g=\sum_{i=1}^k\lambda_ic_i$ with $\lambda_i>0$ and $(c_i)_{i=1}^k$ a complete orthogonal system of idempotents. Take a sequence $\{M^r\}\subseteq K$. Then
$$M^r_{11}=g^k=\sum_{i=1}^k\lambda_i^rc_i,\;
M^r_{22}=\bar g^{-k}=\sum_{i=1}^k\lambda_i^{-r}c_i.$$
We assume, there exists $s\in\{1,\dots,k\}$ such that $\lambda_s\neq\pm 1$. Then either $0<|\lambda_s|<1$ or $0<|\lambda_s^{-1}|<1$. Without loss of generality, we assume $0<|\lambda_s|<1$. Since $K$ is compact, $\{M^r\}\subseteq K$ has a convergent subsequence $\{M^{r_j}\}\subseteq K$:
$$\lim M^{r_j}_{11}=\lim\sum_{i=1}^k\lambda_i^{r_j}c_i=\sum_{i=1}^k\hat\lambda_ic_i$$
where $\hat\lambda_i=\lim\lambda_i^{r_j}$. But $\hat\lambda_s=\lim\lambda_s^{r_j}=0$ for any subsequence $\{r_j\}$. Therefore $\lim M^{r_j}_{11}$ is not invertible and so $\lim M^{r_j}$ is not invertible as well. Therefore, all $\lambda_i=\pm 1$ and $g^2=1$. The element $L:=\begin{pmatrix}
g & 0\\
0 & g^{-1}
\end{pmatrix}\in\KSp^c_2(A_\CC,\sigma_\CC)\subset K$. Then $ML=\begin{pmatrix}
1 & x \\
0 & 1
\end{pmatrix}\in K$. Take $(ML)^r=\begin{pmatrix}
1 & rx \\
0 & 1
\end{pmatrix}\in K$. This sequence does not have any convergent subsequence unless $x=0$. So we get $M=L\in\KSp^c_2(A_\CC,\sigma_\CC)$. This contradicts the assumption $M\notin\KSp^c_2(A_\CC,\sigma_\CC)$ and we obtain that $\KSp^c_2(A_\CC,\sigma_\CC)$ is a maximal compact subgroup of $\Sp_2(A_\CC,\sigma_\CC)$.
\end{proof}

\begin{cor}
Let $(A,\sigma)$ be a real Hermitian algebra, $A_\CC$ be the complexification of $A$, and $\sigma_\CC$ be the complex linear extension of $\sigma$. The embedding 
$$\Upsilon_\HH \colon A_\HH \to \Mat_2(A_\CC)$$ from~\ref{Upsilon_H} maps $U_{(A_\HH,\sigma_1)}$ isomorphically to $\KSp^c_2(A_\CC,\sigma_\CC)$.
\end{cor}

\begin{rem}
Notice, that the group $\KSp_2(A_\CC,\sigma_\CC)$ is never compact because it is a complexification of the real group $\KSp_2(A_\CC,\sigma_\CC)$. 
\end{rem}

\subsection{Realization of classical Lie groups as \texorpdfstring{$\Sp_2(A,\sigma)$}{Sp2(A,sigma)}}\label{ex1}

In the case when $(A, \sigma)$ is a Hermitian algebra so that $(A^\sigma,\circ)$ is a Jordan algebra, the symplectic groups $\Sp_2(A,\sigma)$ are isomorphic to classical Hermitian Lie groups of tube type. 
There are also other groups that can be realized as $\Sp_2(A,\sigma)$ for algebras with anti-involution which are not Hermitian. 

\begin{enumerate}
\item {\bf Real symplectic group $\Sp(2n,\R)$}\\
In order to realize the real symplectic group $\Sp(2n,\R)$, we take
$A=\Mat(n,\R)$ to the be algebra of $n\times n$ matrices over $\R$ and consider the involution $\sigma: A \to A$ given by $\sigma(r)=r^T$, $r\in A$.\\ 
Then $\Sp_2(A,\sigma)$ is isomorphic to $\Sp(2n,\R)$.
    The maximal compact subgroup is
    $$\KSp_2(A,\sigma)=\Upsilon(\UU(n))\cong \UU(n).$$

\item {\bf Indefinite unitary group $\UU(n,n)$}\\ 
To realize the unitary group $\UU(n,n)$ of an indefinite Hermitian form of signature $(n,n)$ we consider $A=\Mat(n,\CC)$, and the involution $\bar\sigma: A \to A$ given by $\bar\sigma(r)=\bar r^T$. 
Then $\Sp_2(A,\bar\sigma)$ is isomorphic to $\UU(n,n)$. To see this, we notice that the standard Hermitian form $h$ of signature $(n,n)$ on $\CC^{2n}$ is given by $h(x,y):=i\omega(xT,yT)$ where $T=\diag(\Id_n,-i\Id_n)$). The complexification $A_\CC$ is isomorphic to $\Mat(n,\CC)\times\Mat(n,\CC)$ (see Section~\ref{Isom_chi}). Therefore,
    $$\KSp_2(A,\sigma)\cong \UU(n)\times\UU(n).$$
Note that we cannot realize the special unitary group $\mathrm{SU}(n,n)$ as $\Sp_2(A,\sigma)$. 

\item {\bf The group $\SO^*(4n)$ }\\
The groups $\SO^*(2n)$ are Hermitian Lie groups, but they are of tube type only if $n$ is even. 
In order to realize $\SO^*(4n)$ as $\Sp_2(A,\sigma)$ we consider  $A=\Mat(n,\HH)$ and the involution $\sigma_1: A \to A$, given by $\sigma_1(r)=\bar r^T=\bar\sigma(r_1)-\sigma(r_2)j$ for $r=r_1+r_2j$ and $r_1,r_2\in \Mat(n,\CC)$. 
Then $\Sp_2(A,\sigma_1)$ is isomorphic $\SO^*(4n)$ (some authors also use the notation $\OO(2n,\HH)$) considered as the group of isometries of the  quaternionic form $\beta$ on $\HH^{2n}$ defined by 
    $$\beta(x,y)=\sum_{i=1}^{2n} \bar x_i j y_i=\bar x^T(\Id_{2n}j)y.$$
    To see this, we notice that
    $$\Id_{2n} j=\sigma_1(T)\Ome{\Id_n} T$$
    for
    $$T=\frac{1}{\sqrt{2}}\begin{pmatrix}\Id_n & -\Id_n j \\-\Id_n j & \Id_n\end{pmatrix}.$$
    In this case $A_\CC$ is isomorphic to $\Mat(2n,\CC)$ (see Section~\ref{Isom_psi}). Therefore.
    $$\KSp_2(A,\sigma)\cong \UU(2n).$$
\end{enumerate}

\begin{rem}
The Hermitian Lie group of tube type $\SO_0(2,n)$ cannot be realized in the same way as $\Sp_2(A,\sigma)$. The reason is that the Jordan algebra $B_n$ (see the classification~\ref{Jclassif}) cannot be seen as $A^\sigma$ for an appropriate Hermitian algebra $(A,\sigma)$. Nevertheless, since $B_n$ can be seen as Jordan subalgebra of an appropriate Clifford algebra, the group $\SO_0(2,n)$ (or more precisely, its double cover $\Spin_0(2,n)$) can be realized as $\Sp_2$ over some more complicated object which we do not discuss in this paper.
\end{rem}

\begin{rem}\label{rem:exceptional}
The exceptional Hermitian Lie group of tube type $E_7$ cannot be realized as $\Sp_2(A,\sigma)$. The reason is that the Jordan algebra $\Herm(3,\Oc)$ (see the classification~\ref{Jclassif}) cannot be embedded into any associative algebra, in particular it cannot be realized as $A^\sigma$ for a Hermitian algebra $(A,\sigma)$.
\end{rem}

Some other Lie groups can also be realized as $\Sp_2(A,\sigma)$ for algebras with anti-involution $(A,\sigma)$ that is not Hermitian.

\subsubsection{Other examples}\label{other_examples}
There are other interesting cases of classical groups that can be realized as $\Sp_2(A,\sigma)$ where $(A,\sigma)$ is not Hermitian. 
\begin{enumerate}
\item To realize the {\bf symplectic group over any field $\K$} we  consider $A=\Mat(n,\K)$ with the anti-involution $\sigma(r):=r^T$. Then $\Sp_2(A,\sigma)$ is isomorphic to $\Sp(2n,\K)$. \\
    If $\K=\CC$, $A$ is the complexification of $\Mat(n,\R)$, and the maximal compact subgroup is
    $$\KSp^c_2(A,\sigma)\cong \Sp(n).$$

\item The {\bf indefinite symplectic group} can be realized as $\Sp_2(A,\sigma)$. For this consider $A=\Mat(n,\HH)$ with involution  $\sigma_0(r)=\sigma(r_1)+\bar\sigma(r_2)j$ for $r=r_1+r_2j$ and $r_1,r_2\in \Mat(n,\CC)$. Note that the algebra $(A,\sigma_0)$ is not Hermitian and also does not appear as a complexification of a Hermitian algebra.\\
Then $\Sp_2(A,\sigma_0)$ is isomorphic $\Sp(n,n)$ considered as the group of isometries of the quaternionic form $\omega$ on $\HH^{2n}$ given by 
    $$\omega(x,y)=\sum_{i=1}^{2n} \bar x_i y_i=\bar x^T\Dia{-\Id_n}{\Id_n}y=$$
    $$=k\left(\sigma_0(x)\Dia{\Id_n k}{-\Id_n k}y\right).$$
    To see this, we notice that
    $$\Dia{\Id_n k}{-\Id_n k}=\sigma_0(T)\Ome{\Id_n} T$$
    for
    $$T=\frac{1}{\sqrt{2}}\begin{pmatrix}\Id_n & \Id_n k \\ \Id_n k & \Id_n\end{pmatrix}.$$
    The maximal compact subgroup of $\Sp(n,n)$ is $\Sp(n)\times\Sp(n)$. One can see this using the machinery developed in the Section~\ref{SymSp_O(h)}.\\
    The subgroup $\KSp_2(A,\sigma_0)$ is isomorphic to $\GL_n(\HH)$ which is not compact because $(A,\sigma_0)$ is not Hermitian.
\end{enumerate}

\section{Space of isotropic lines}\label{A-islines}

We assume $(A,\sigma)$ to be a unital ring with an anti-involution. We denote by $\PP(A^2)$ the space of lines in $A^2$, i.e. 
$$\PP(A^2):=\{xA\mid x\in A^2\text{ regular}\}.$$
The group $\Mat_2^\times(A)$ acts on $\PP(A^2)$.

If $A^2$ is equipped with the standard symplectic form $\omega$, we denote by $\PP(\Is(\omega))$ the space of all isotropic (with respect to $\omega$) lines:
$$\PP(\Is(\omega))=\{xA\subset A^2\mid \omega(x,x)=0,\,x\text{ regular}\}.$$
This space is a closed subspace of $\PP(A^2)$. The group $\Sp_2(A,\sigma)$ acts on the space of isotropic lines.

\subsection{Space of lines as a homogeneous space} 

Now assume $\K$ to be real closed field and $(A,\sigma)$ is an algebra with an anti-involution over $\K$. We show that $\PP(A^2)$ and $\PP(\Is(\omega))$ can be seen as homogeneous (even symmetric) spaces that is compact in case $\K=\R$.

\begin{prop}\label{PP_compact}
Let $(A,\sigma)$ be Hermitian. The group $\UU_2(A,\sigma)$ acts continuously  and transitively on $\PP(A^2)$ with 
$$\Stab_{\UU_2(A,\sigma)}\Clm{1}{0}A=\left\{\begin{pmatrix}u_1 & 0 \\ 0 & u_2\end{pmatrix}\midwd u_1,u_2\in U_{(A,\sigma)}\right\}\cong U_{(A,\sigma)}\times U_{(A,\sigma)}.$$
\end{prop}

\begin{proof}
Let $xA$ be a line. Since $x\in A^2$ is regular, by Proposition~\ref{regular-sesq}, we can assume $b(x,x):=\sigma(x)^Tx=1$. By Proposition~\ref{Gram}, there exists $y$ such that $(x,y)$ is an orthonormal basis, i.e. the matrix
$$\begin{pmatrix}
x_1 & y_1\\
x_2 & y_2
\end{pmatrix}\in \UU_2(A,\sigma),$$
where $x=(x_1,x_2)^T$, $y=(y_1,y_2)^T$. Moreover
$$\begin{pmatrix}
x_1 & y_1\\
x_2 & y_2
\end{pmatrix}\Clm{1}{0}=x.$$
That means $\UU_2(A,\sigma)$ acts transitively on $\PP(A^2)$.

Now compute the stabilizer of $(1,0)^TA$. Take $U:=\begin{pmatrix}
x_1 & y_1\\
x_2 & y_2
\end{pmatrix}\in \UU_2(A,\sigma).$ Since $U(1,0)=(a,0)$ for some $a\in A$, $x_2=0$. Further
$$\Id_2=\sigma(U)^TU= \begin{pmatrix}
\sigma(x_1) & 0 \\
\sigma(y_1) & \sigma(y_2)
\end{pmatrix}\begin{pmatrix}
x_1 & y_1\\
0 & y_2
\end{pmatrix}=\begin{pmatrix}
\sigma(x_1)x_1 & \sigma(x_1)y_1 \\
\sigma(y_1)x_2 & \sigma(y_1)y_1+\sigma(y_2)y_2
\end{pmatrix}.$$
Therefore, $\sigma(x_1)x_1=1$, i.e. $x_1\in U_{(A,\sigma)}$. Further, $\sigma(x_1)y_1=0$. Since $x_1\in U_{(A,\sigma)}$, it is invertible and so $y_1=0$. Finally, $\sigma(y_1)y_1+\sigma(y_2)y_2=\sigma(y_2)y_2=1$. Therefore, $y_2\in U_{(A,\sigma)}$.
\end{proof}

\begin{cor}\label{PP_homog}
The space $\PP(A^2)$ is homeomorphic to the quotient space:
$$\UU_2(A,\sigma)/U_{(A,\sigma)}\times U_{(A,\sigma)}$$
where the group $U_{(A,\sigma)}\times U_{(A,\sigma)}$ is embedded into $\UU_2(A,\sigma)$ diagonally.

In particular, if $\K=\R$, $\PP(A^2)$ is compact and $\PP(\Is(\omega))$ is compact for any sesquilinear form $\omega\colon A^2\times A^2\to A$.
\end{cor}

\begin{cor}
For Hermitian algebras and complexifications of Hermitian algebras, we can describe the space $\PP(\Is(\omega))$ more precisely. Similarly to the Proposition~\ref{PP_compact} and Corollary~\ref{PP_homog}, one can prove the following two statements:
\begin{enumerate}
    \item Let $(A,\sigma)$ be Hermitian and $\omega$ be the standard symplectic form on $A^2$. Then $\KSp_2(A,\sigma)$ acts transitively on $\PP(\Is(\omega))$ with
        $$\Stab_{\KSp_2(A,\sigma)}\Clm{1}{0}A=\left\{\begin{pmatrix}u & 0 \\ 0 & u\end{pmatrix}\midwd u \in U_{(A,\sigma)}\right\}\cong U_{(A,\sigma)}.$$
        In particular, $\PP(\Is(\omega))$ is homeomorphic to the quotient space:
        $$\KSp_2(A,\sigma)/U_{(A,\sigma)}$$ 
        where the group $U_{(A,\sigma)}$ is embedded into $\KSp_2(A,\sigma)$ in the diagonal way.
\item Let $(A_\CC,\sigma_\CC)$ be the complexification of a Hermitian algebra $(A,\sigma)$ and $\omega$ be the standard symplectic form on $A_\CC^2$. Then $\KSp^c_2(A_\CC,\sigma_\CC)$ acts transitively on $\PP(\Is(\omega))$ with
        $$\Stab_{\KSp^c_2(A_\CC,\sigma_\CC)}\Clm{1}{0}A=\left\{\begin{pmatrix}u & 0 \\ 0 & u\end{pmatrix}\midwd u \in U_{(A_\CC,\bar\sigma_\CC)}\right\}\cong U_{(A_\CC,\bar\sigma_\CC)}.$$
        In particular, $\PP(\Is(\omega))$ is homeomorphic to the quotient space:
        $$\KSp^c_2(A_\CC,\sigma_\CC)/U_{(A_\CC,\bar\sigma_\CC)}$$
        where the group $U_{(A_\CC,\bar\sigma_\CC)}$ is embedded into $\KSp^c_2(A_\CC,\sigma_\CC)$ in the diagonal way.
\end{enumerate}
\end{cor}


\begin{prop}\label{stab1_A} $\Sp_2(A,\sigma)$ acts transitively on $\PP(\Is(\omega))$.
$$\Stab_{\Sp_2(A,\sigma)}\left(\Clm{1}{0}A\right):=\left\{
\begin{pmatrix}
x & xy \\
0 & \sigma(x)^{-1}
\end{pmatrix}
\midwd
x\in A^\times, y\in A^\sigma
\right\}$$

$$\Stab_{\Sp_2(A,\sigma)}\left(\Clm{0}{1}A\right):=\left\{
\begin{pmatrix}
x & 0 \\
zx & \sigma(x)^{-1}
\end{pmatrix}
\midwd
x\in A^\times, z\in A^\sigma
\right\}$$

\end{prop}

\begin{proof} $\Sp_2(A,\sigma)$ acts transitively on the space of isotropic lines since it acts transitively on $\Is(\omega)$.

We prove only the statement for the first stabilizer. The second one can be proved analogously.

Since
$$\begin{pmatrix}
x & a \\
b & t
\end{pmatrix}
\begin{pmatrix}
1 \\ 0
\end{pmatrix}=\begin{pmatrix}
x \\ b
\end{pmatrix},$$
$x\in A^\times$ and $b=0$. Furthermore,
$$\sigma\left(\begin{pmatrix}
x & a \\
0 & t
\end{pmatrix}\right)^T\Ome{1}
\begin{pmatrix}
x & a \\
0 & t
\end{pmatrix}=\begin{pmatrix}
0 & \sigma(x)t \\
-\sigma(t)x & -\sigma(t)a+\sigma(a)t
\end{pmatrix}=\Ome{1},$$
we obtain $t=\sigma(x)^{-1}$, $a=xy$ for $y\in A^\sigma$.
\end{proof}

\subsection{Action of \texorpdfstring{$\Sp_2(A,\sigma)$}{Sp2(A,sigma)} on pairs of isotropic lines}\label{subsect_triples}

Let $(A,\sigma)$ be a unital ring with an anti-involution. 

\begin{prop}
Two elements $u,v\in\Is(\omega)$ are linearly independent if and only if, up to action of $\Sp_2(A,\sigma)$, $u=(1,0)^T$, $v=(a,b)^T$ with $b\in A^\times$. Moreover, if $\omega(u,v)=1$, then $a\in A^\sigma$, $b=1$.
\end{prop}

\begin{proof}
$\Sp_2(A,\sigma)$ acts transitively on $\Is(\omega)$, therefore, up to $\Sp_2(A,\sigma)$-action, we can assume that $u=(1,0)^T$. Since $u$ and $v$ are linearly independent, $b\in A^\times$. If $\omega(u,v)=1=b$, then $v=(a,1)^T$ is isotropic, i.e.
$$\omega(v,v)=\sigma(a)-a=0$$
So $a\in A^\sigma$.
\end{proof}

\begin{cor}
If $x,y\in\Is(\omega)$ linearly independent, then $\omega(x,y)\in A^\times$.
\end{cor}

\begin{prop}\label{trans_bas_A}
If $(x,y)$ is a symplectic basis then there exists the unique $g\in \Sp_2(A,\sigma)$ such that $g(1,0)^T=x$, $g(0,1)^T=y$. In particular, $\Sp_2(A,\sigma)$ acts transitively on symplectic bases.
\end{prop}

\begin{proof}
We can assume, $x=(1,0)^T$, $y=(a,1)^T$ and $a\in A^\sigma$. Take $g:=\begin{pmatrix}1 & a \\ 0 & 1\end{pmatrix}$, then $gx=x$, $gy=(0,1)^T$.
\end{proof}

\begin{cor}\label{trans2_A}
Let $xA$, $yA$ be two transverse isotropic lines with $x,y\in\Is(\omega)$. Then there exist $M\in \Sp_2(A,\sigma)$ and $y'\in\Is(\omega)$ such that $y'A=yA$ and $Mx=(1,0)^T$, $My'=(0,1)^T$. In particular, $\omega(x,y')=1$.
\end{cor}

\begin{prop}\label{stab2_A}
$\Sp_2(A,\sigma)$ acts transitively on pairs of transverse isotropic lines.
$$\Stab_{\Sp_2(A,\sigma)}\left(\Clm{1}{0}A,\Clm{0}{1}A\right):=\left\{
\begin{pmatrix}
x & 0 \\
0 & \sigma(x)^{-1}
\end{pmatrix}
\mid
x\in A^\times\right\}\cong A^\times.$$
\end{prop}

\begin{proof} By the Corollary~\ref{trans2_A}, every pair of transverse isotropic lines can be mapped to $((1,0)^TA,(0,1)^TA)$ by an element of $\Sp_2(A,\sigma)$. So $\Sp_2(A,\sigma)$ acts transitively on pairs of transverse isotropic lines.

By the Proposition~\ref{stab1_A},
$$\Stab_{\Sp_2(A,\sigma)}\left(\Clm{1}{0}A,\Clm{0}{1}A\right)=$$
$$=\Stab_{\Sp_2(A,\sigma)}\left(\Clm{1}{0}A\right)\cap \Stab_{\Sp_2(A,\sigma)}\left(\Clm{0}{1}A\right)=$$
\begin{equation*}
=\left\{
\begin{pmatrix}
x & 0 \\
0 & \sigma(x)^{-1}
\end{pmatrix}
\mid
x\in A\right\}.\qedhere
\end{equation*}
\end{proof}

\subsection{Action of \texorpdfstring{$\Sp_2(A,\sigma)$}{Sp2(A,sigma)} on triples of isotropic lines}

Let $(A,\sigma)$ be a unital ring with an anti-involution. Let $(x_1A,x_3A,x_2A)$ be a triple of pairwise transverse isotropic lines where all $x_i\in \Is(\omega)$. Because of transversality of $x_1A$ and $x_2A$, we can assume $\omega(x_1,x_2)=1$. Up to action of $\Sp_2(A,\sigma)$, we can assume $x_1=(1,0)^T$, $x_2=(0,1)^T$. We can also normalize $x_3$ so that $\omega(x_1,x_3)=1$. Then $x_3=(b,1)^T$, $b=\omega(x_3,x_2)\in (A^\sigma)^\times$.

\begin{prop}\label{pr:action_on_triples}
Orbits of the action of $\Sp_2(A,\sigma)$ on triples of pairwise transverse isotropic lines are in 1-1 correspondence with orbits of the following action of $A^\times$ on $(A^\sigma)^\times$:
$$\begin{array}{cccl}
\psi\colon & A^\times\times (A^\sigma)^\times & \mapsto & (A^\sigma)^\times\\
 & (a,b) & \mapsto & ab\sigma(a).
\end{array}$$
\end{prop}

\begin{proof} Let $(l_1,l_3,l_2)$ be a triple pairwise transverse of isotropic lines.
Up to $\Sp_2(A,\sigma)$-action, we can assume $l_i=x_iA$ for $x_1=(1,0)^T$, $x_2=(0,1)^T$, $x_3=(1,b)^T$ with $b\in (A^\sigma)^\times$. The stabilizer $\Stab_{\Sp_2(A,\sigma)}((1,0)^TA,(0,1)^TA)\cong A^\times$ acts on $x_3$ in the following way:
$$\diag(a,\sigma(a)^{-1})x_3=(ab,\sigma(a)^{-1})^T=(ab\sigma(a),1)^Ta^{-1}$$
i.e. $\diag(a,\sigma(a)^{-1})(b,1)^TA=(ab\sigma(a),1)^TA$

So we see that in the orbit of $(b,1)^TA$ are exactly all isotropic lines of the form $(b',1)^TA$ where $b'$ is from the orbit of $b$ under $\psi$.
\end{proof}

%

\begin{df}\label{Maslov} Let $(A,\sigma)$ is a semisimple Hermitian algebra over a real closed field such that $A^\sigma$ is a simple Jordan algebra. The \defin{Kashiwara-Maslov index} of a triple of pairwise transverse isotropic lines $(l_1,l_3,l_2)$ is a signature of an element $b\in A^\sigma$ such that up to the action of $\Sp_2(A,\sigma)$, $l_i=x_iA$, $x_1=(1,0)^T$, $x_2=(0,1)^T$, $x_3=(b,1)^T$.  

We call a triple $(l_1,l_3,l_2)$ of isotropic lines \defin{positive} if its Kashiwara-Maslov index is $(n,0)$ where $n$ is the rank of $A^\sigma$. (or equivalently a corresponding element $b$ is in $A^\sigma_+$. 
\end{df}

By the Sylvester's law of inertia~\ref{Sylvester}, the action $\psi$ form~\ref{pr:action_on_triples} preserves the signature. Therefore, the Kashiwara-Maslov index is well defined for every triple of pairwise transverse isotropic lines. Moreover, because of~\ref{action_on_Asigma}, the positivity of a triple is well defined for every semisimple Hermitian algebra (not only if $A^\sigma$ is a simple Jordan algebra).

The following properties of the Kashiwara-Maslov index are well known (see~\cite[Section~1.5]{Lion}, \cite[Section~5]{Clerc}):

\begin{prop}
\begin{enumerate}
    \item The Kashiwara-Maslov index $\mu$ is alternating and invariant for the diagonal action of $\Sp_2(A,\sigma)$ on the space of transverse triples of isotropic lines;
    \item $\mu$ takes values $\{-n,-n+2,\dots,n\}$ where $n$ is the rank of the Jordan algebra $A^\sigma$;
    \item $\mu$ is a cocycle: for 4-tuple $(l_1,l_2,l_3,l_4)$ of pairwise transverse isotropic lines, we have
    $$\mu(l_2,l_3,l_4)-\mu(l_1,l_3,l_4)+\mu(l_1,l_2,l_4)-\mu(l_1,l_2,l_3)=0.$$
\end{enumerate}
\end{prop}


\begin{prop}
If $(A,\sigma)$ a semisimple Hermitian algebra over a real closed field, then $\Sp_2(A,\sigma)$ acts transitively on positive triples of isotropic lines.

The stabilizer of the positive triple
$$\left(\begin{pmatrix}1 \\ 0\end{pmatrix}A,\begin{pmatrix}1 \\ 1\end{pmatrix}A,\begin{pmatrix}0 \\ 1\end{pmatrix}A\right)$$
in $\Sp_2(A,\sigma)$ coincides with the following subgroup:
$$\hat U:=\left\{
\begin{pmatrix}
u & 0 \\
0 & u
\end{pmatrix}
\midwd u\in U_{(A,\sigma)}\right\}\cong U_{(A,\sigma)}.$$

The stabilizer of every positive triple of isotropic lines is conjugate in $\Sp_2(A,\sigma)$ to $\hat U$.
\end{prop}

\begin{proof}
This follows by a direct computation from~\ref{action_on_Asigma} and~\ref{pr:action_on_triples}.
\end{proof}

\begin{prop}
If $(A,\sigma)$ a semisimple Hermitian algebra over a real closed field such that $A^\sigma$ is a simple Jordan algebra of rank $n$, then $\Sp_2(A,\sigma)$ acts transitively on triples of isotropic lines with a fixed Kashiwara-Maslov index.

Let $p,q\in\N\cup\{0\}$ with $p+q=n$. Let $(e_i)_{i=1}^n$ be a fixed Jordan frame in $A^\sigma$ and $o_{p,q}:=\sum_{i=1}^p e_i-\sum_{i=1}^q e_{p+i}$. The stabilizer of the triple
$$\left(\begin{pmatrix}1 \\ 0\end{pmatrix}A,\begin{pmatrix}o_{p,q} \\ 1\end{pmatrix}A,\begin{pmatrix}0 \\ 1\end{pmatrix}A\right)$$
in $\Sp_2(A,\sigma)$ coincides with the following subgroup:
$$\hat U(o_{p,q}):=\left\{
\begin{pmatrix}
u & 0 \\
0 & u
\end{pmatrix}
\midwd \sigma(u)o_{p,q}u=o_{p,q}\right\}.$$
\end{prop}

\begin{proof}
This follows by a direct computation from~\ref{Sylvester} and~\ref{pr:action_on_triples}.
\end{proof}

Finally, we consider the complexification $(A_\CC,\sigma_\CC)$ of a Hermitian algebra $(A,\sigma)$. In this case, the action of $\Sp_2(A_\CC,\sigma_\CC)$ on triples of isotropic $A_\CC$-lines is transitive if the Jordan algebra $A^{\bar\sigma_\CC}_\CC$ is simple.

\begin{prop}
Let $(A_\CC,\sigma_\CC)$ be the complexification of a Hermitian algebra $(A,\sigma)$ over a real closed field such that the Jordan algebra $A^{\bar\sigma_\CC}_\CC$ is simple. The group $\Sp_2(A_\CC,\sigma_\CC)$ acts transitively on triples of isotropic $A_\CC$-lines.

The stabilizer of the triple
$$\left(\begin{pmatrix}1 \\ 0\end{pmatrix}A,\begin{pmatrix}1 \\ 1\end{pmatrix}A,\begin{pmatrix}0 \\ 1\end{pmatrix}A\right)$$
in $\Sp_2(A_\CC,\sigma_\CC)$ coincides with the following subgroup:
$$\hat U_\CC:=\left\{
\begin{pmatrix}
u & 0 \\
0 & u
\end{pmatrix}
\midwd u\in U_{(A_\CC,\sigma_\CC)}\right\}\cong U_{(A_\CC,\sigma_\CC)}$$

The stabilizer of every positive triple of isotropic lines is conjugate in $\Sp_2(A,\sigma)$ to $\hat U$.
\end{prop}

\begin{proof}
This follows from~\ref{Transit_action_A_CC} and~\ref{pr:action_on_triples}.
\end{proof}

\subsection{Action of \texorpdfstring{$\Sp_2(A,\sigma)$}{Sp2(A,sigma)} on quadruples of isotropic lines -- the cross ratio}

Let $(A,\sigma)$ be a unital ring with an anti-involution. We consider the following subspace of $A$:
$$A_0:=\{bb'\mid b,b'\in (A^\sigma)^\times\}.$$
$A^\times$ acts on $A_0$ by conjugation because for $b,b'\in (A^\sigma)^\times$, $a\in A^\times$:
$$a(bb')a^{-1}=(ab\sigma(a))(\sigma(a)^{-1}b')a^{-1}\in A_0.$$

\begin{rem}
It is a well-known fact from linear algebra that for matrix algebras $A$ over $\R$, $\CC$ or $\HH$, we always have $A_0=A^\times$.
\end{rem}

\begin{prop}
Orbits of the action of $\Sp_2(A,\sigma)$ on quadruples of pairwise transverse isotropic lines are in 1-1 correspondence with orbits of the following action of $A^\times$ on $A_0$:
$$\begin{array}{cccl}
\eta\colon & A^\times\times A_0 & \mapsto & A_0\\
 & (a,b) & \mapsto & aba^{-1}.
\end{array}$$
\end{prop}

\begin{proof}
Let $(l_1,l_3,l_2,l_4)$ be a quadruple pairwise transverse of isotropic lines. Then up to action of $\Sp_2(A,\sigma)$, we can assume $l_1=(1,0)^TA$, $l_2=(0,1)^TA$, $l_3=(b,1)^TA$, $l_3=(1,b')^TA$ with $b,b'\in (A^\sigma)^\times$. Consider the action of the stabilizer of $(l_1,l_2)$:
$$\diag(a,\sigma(a)^{-1})(b,1)^TA=(ab\sigma(a),1)A,$$
$$\diag(a,\sigma(a)^{-1})(1,b')^TA=(1,\sigma(a)^{-1}b'a^{-1})A.$$

We consider the map $(l_1,l_3,l_2,l_4)\mapsto bb'\in A_0$. This map is well-defined, bijective and the action of the stabilizer of $(l_1,l_2)$ (that is isomorphic to $A^\times$) induces the action of $A^\times$ by conjugation on $A_0$. So we obtain that these two actions are isomorphic.
\end{proof}

\begin{df}
The conjugacy class of $A_0$ corresponding to the quadruple $(l_1,l_3,l_2,l_4)$ of pairwise transverse isotropic lines is called the \defin{cross ratio} of $(l_1,l_3,l_2,l_4)$.
\end{df}

\subsection{Examples of matrix algebras}

In this section, we construct explicitly examples of spaces of isotropic lines for classical matrix algebras. We will use the following notation: for complex numbers, we write $\CC\{I\}$ to emphasize that the imaginary unit is denoted by $I$. Similarly, for quaternions, we write $\HH\{I,J,K\}$ to emphasize that the imaginary units are denoted by $I$, $J$, $K$. The multiplication rule is then $IJ=K$.

\begin{ex}\label{Ex_Isotr_lines1}
Let $(A,\sigma)$ be $(\Mat(n,\R),\sigma)$, $(\Mat(n,\CC),\sigma)$, $(\Mat(n,\CC),\bar\sigma)$ or $(\Mat(n,\HH),\bar\sigma)$ where $\sigma$ is the transposition, $\bar\sigma$ the composition of transposition and complex/quaternionic conjugation. 

Every regular element of $x\in A^2$ can be seen as a $2n\times n$-matrix of maximal rank. Columns of this matrix are elements of $\K^{2n}$ considered as a right $\K$-module where $\K$ is $\R$, $\CC$ or $\HH$. If we take the $\K$-span of this columns, we obtain $n$-dimensional submodule of $\K^{2n}$ denoted by $\Span_\K(x)$. It is easy to see that the map:
$$\begin{array}{cccl}
L\colon & \PP(A^2) & \to & \Gr(n,\K^{2n})\\
& xA & \mapsto & \Span_\K(x)
\end{array}$$
where $\Gr(n,\K^{2n})$ is the space of all $n$-dimensional submodules of $\K^{2n}$ is a bijection.

We consider the following form (bilinear or sesquilinear depending on $\sigma$) on $\K^{2n}$:
$$\tilde\omega(u,v):=\sigma(u)\Ome{\Id_n}v$$
for $u,v\in \K^{2n}$. Then $x\in\Is(\omega)$ if and only if $\Span_\K(x)$ is isotropic with respect to $\tilde\omega$, that means for all $u,v\in\Span_\K(x)$, $\tilde\omega(u,v)=0$. So we obtain that $L$ maps bijectively isotropic lines of $A^2$ to isotropic $n$-dimensional submodules of $\K^{2n}$. Such submodules are called \defin{Lagrangian} with respect to $\tilde\omega$. The space of all Lagrangian with respect to $\tilde\omega$ submodules are denoted by $\Lag(\K^{2n},\tilde\omega)$.
\end{ex}

\begin{ex}
Let $A=\Mat(n,\CC\{I\})\otimes\CC\{i\}$ with the anti-involution $\bar\sigma\otimes\Id$. We use the map $\chi: A \rightarrow A'$ (see Appendix~\ref{Isom_chi}) to identify $A$ with $A':= \Mat(n,\CC\{i\})\times\Mat(n,\CC\{i\})$.
The anti-involution $\sigma':=\chi\circ(\bar\sigma\otimes\Id)\circ\chi^{-1}$ induced by $\bar\sigma\otimes\Id$
on $\Mat(n,\CC\{i\})\times\Mat(n,\CC\{i\})$ acts in the following way:
$$(m_1,m_2)\mapsto(m^T_2,m^T_1).$$

The map $\chi$ can be extended componentwise to the map
$$\chi'\colon\Mat_2(A)\to\Mat_2(A').$$

\begin{prop}
$\Sp_2(A,\bar\sigma\otimes\Id)$ is isomorphic to $\GL(2n,\CC)$.
\end{prop}
\begin{proof} First, we note that
$$A^{\bar\sigma\otimes\Id}=\Sym(n,\CC\{i\})+\Skew(n,\CC\{i\})I.$$
It is enough, to identify $\spp_2(A,\bar\sigma\otimes\Id)$ and $\Mat_2(A_\R)=\Mat(2n,\CC)$ as Lie algebras.  First, we take the map $\chi'$ restricted to $\spp_2(A,\bar\sigma\otimes\Id)$:
$$\begin{matrix}
\chi'\colon & \spp_2(A,\sigma) & \to & \Mat(2n,\CC\{i\})\times\Mat(2n,\CC\{i\}) \\
& \begin{pmatrix}
a_1+a_2I & b_1+b_2I \\
c_1+c_2I & -a_1^T+a_2^TI
\end{pmatrix} & \mapsto & \left(
\begin{pmatrix}
a_1+a_2i & b_1+b_2i \\
c_1+c_2i & -a_1^T+ a_2^Ti
\end{pmatrix},
\begin{pmatrix}
a_1-a_2i & b_1-b_2i \\
c_1-c_2i & -a_1^T- a_2^Ti
\end{pmatrix}\right).\end{matrix}$$
where $a_1,a_2\in\Mat(n,\CC\{i\})$, $b_1,c_1\in \Sym(n,\CC\{i\})$, $b_2,c_2\in\Skew(n,\CC\{i\})$. This is an injective homomorphism of $\CC\{i\}$-Lie algebras as restriction of injective map. Finally, we take a projection to the first component:
$$\pi_1\colon \Mat(2n,\CC\{i\})\times\Mat(2n,\CC\{i\})\to \Mat(2n,\CC\{i\}).$$
Easy computation shows that $\pi_1\circ\chi'$ is an isomorphism.
\end{proof}

The set $(A')^2$ can be identified with the space of pairs $(x_1,x_2)^T$ such that $x_1,x_2\in\Mat(n,\CC\{i\})$. We define the sesquilinear form:
$$\omega((x_1,x_2)^T,(y_1,y_2)^T)=\sigma'(x_1,x_2)\Ome{(\Id_n,\Id_n)}(y_2,y_2)^T=$$
$$=(\sigma(x_2)\Ome{\Id_n}y_1,\sigma(x_1)\Ome{\Id_n}y_2).$$
Therefore,
$$\Is(\omega)=\{(l_1,l_2)\mid l=x_1\Mat(n,\CC\{i\}),\,l_2=x_2\Mat(n,\CC\{i\}),\,x_1,x_2\text{ regular},\,\omega(x_1,x_2)=0\}.$$
Since $\omega$ is non-degenerate, $l_2$ is uniquely determined by $l_1$. Therefore, we can identify:
$$\Is(\omega)\cong\{x\Mat(n,\CC\{i\})\mid x\text{ regular}\}.$$

As in the previous example, we can identify lines in $\Mat(n,\CC\{i\})^2$ with Lagrangian subspaces of $(\CC^{2n},\tilde\omega)$ where:
$$\tilde\omega(u,v)= u^T\Ome{\Id_n}v.$$
So the space $\Is(\omega)$ can be identified with
$$\Is(\omega)\cong\{(l_1,l_2)\in\Gr(n,\CC^{2n})^2\mid \tilde\omega(u,v)=0\text{ for all }u\in l_1,\,v\in l_2\}.$$

The form $\tilde\omega$ is a non-degenerate. Therefore, for $l\in \Gr(n,\CC^{2n})$ there exists exactly one $\tilde\omega$-orthogonal complement $l^\bot\in \Gr(n,\CC^{2n})$ such that for all $u\in l$, $v\in l^\bot$, $\tilde\omega(u,v)=0$. So we can identify
$$\Is(\omega)\cong \Gr(n,\CC^{2n})$$
and $\GL(2n,\CC)$ acts on $\Gr(n,\CC^{2n})$ in the standard way.
\end{ex}

\begin{ex}
Let $A=\Mat(n,\HH\{i,j,k\})\otimes\CC\{I\}$ with the anti-involution $\bar\sigma\otimes\Id$. We use the map $\psi: A \rightarrow A'$ (see Appendix~\ref{Isom_psi}) to identify $A$ with $A':=\Mat(2n,\CC)$. The anti-involution $\sigma':=\psi\circ(\bar\sigma\otimes\Id)\circ\psi^{-1}$ on $A'$ induced by $\bar\sigma\otimes\Id$
on $\Mat(2n,\CC)$ acts in the following way:
$$m\mapsto -\Ome{\Id}m^T\Ome{\Id}.$$
We define the following $\sigma'$-sesquilinear form on $(A')^2$: for $x,y\in (A')^2$
$$\omega(x,y)=\sigma'(x)^T\Ome{\Id_{2n}}y.$$

\begin{prop}
$\Sp_2(A,\bar\sigma\otimes\Id)$ is isomorphic to $\OO(4n,\CC)$.
\end{prop}

\begin{proof}
$M\in \Sp_2(A',\sigma')\cong\Sp_2(A,\bar\sigma\otimes\Id)$ if and only if
$$\sigma'(M)^T\begin{pmatrix}
0 & \Id_{2n}\\
-\Id_{2n} & 0
\end{pmatrix}M=\begin{pmatrix}
0 & \Id_{2n}\\
-\Id_{2n} & 0
\end{pmatrix},$$
i.e.
$$-\begin{pmatrix}
\COme{\Id_n} & 0\\
0 & \COme{\Id_n}
\end{pmatrix}M^T\begin{pmatrix}
\COme{\Id_n} & 0\\
0 & \COme{\Id_n}
\end{pmatrix}\cdot$$
$$\cdot\begin{pmatrix}
0 & \Id_{2n}\\
-\Id_{2n} & 0
\end{pmatrix}M=\begin{pmatrix}
0 & \Id_{2n}\\
-\Id_{2n} & 0
\end{pmatrix}$$
This is equivalent to:
$$M^T
\begin{pmatrix}
0 & \COme{\Id_n}\\
\MOme{\Id_n} & 0
\end{pmatrix}M=\begin{pmatrix}
0 & \COme{\Id_n}\\
\MOme{\Id_n} & 0
\end{pmatrix}.$$
So the group $\Sp_2(A,\sigma)$ is the group of symmetries of the symmetric bilinear form form
$$\begin{pmatrix}
0 & \COme{\Id_n}\\
\MOme{\Id_n} & 0
\end{pmatrix}$$
on $\CC^{4n}$. But all symmetric bilinear forms on $\CC^{4n}$ are conjugated. Therefore, $\Sp_2(A,\sigma)$ is isomorphic to $\OO(4n,\CC)$.
\end{proof}

Note that $\Is(\omega)=\Is(\omega')$ for
$$\omega'(x,y):=x^T\begin{pmatrix}
0 & \COme{\Id_n}\\
\MOme{\Id_n} & 0
\end{pmatrix}y.$$

As before, we can identify lines in $(A')^2$ with the space $\Gr(2n,\CC^{4n})$ of $2n$-dimensional subspaces of $\CC^{4n}$ using the map $L$ (see Example~\ref{Ex_Isotr_lines1}). Under this map, the space $\Is(\omega)$ goes to the space of all maximal $\tilde\omega$-isotropic subspaces where
$$\tilde\omega(u,v)=u^T\begin{pmatrix}
0 & \COme{\Id_n}\\
\MOme{\Id_n} & 0
\end{pmatrix}v$$
for $x,y\in\CC^{4n}$. The group $\OO(\tilde\omega)\cong\OO(4n,\CC)$ acts on the space of all maximal $\tilde\omega$-isotropic subspaces in the standard way.
\end{ex}

\section{Models for the symmetric space of \texorpdfstring{$\Sp_2(A,\sigma)$}{Sp2(A,sigma)} over Hermitian algebras}\label{AR-models}

The goal of this Chapter is to construct different models of the symmetric space for $\Sp_2(A,\sigma)$ for a real Hermitian algebra $(A,\sigma)$. In the case when $\Sp_2(A,\sigma)$ is a classical Hermitian Lie group of tube type, we recover many aspects of their well-known structure theory. We refer the reader to~\cite{Sateke, Koufany05} for standard works on Hermitian Lie groups and their symmetric spaces.  

In this Chapter, we discuss only groups $\Sp_2(A,\sigma)$ for semisimple Hermitian algebras $(A,\sigma)$ over the field $\R$. In this case, the corresponding symmetric spaces become naturally Riemannian manifolds. Actually, this construction works also for every real closed field. Although in general the topology a real closed field does not allow to define a structure of a Riemannian manifold on the corresponding symmetric spaces, we still have some ``Riemannian structure'' on them. For example, a $\Sp_2(A,\sigma)$-invariant $\K$-valued Riemannian metric is well-defined. Moreover, as we will see, all the models we are going to describe in this Chapter are semi-algebraic sets that are semi-algebraically isomorphic to each other. This allows us to define a natural semi-algebraic structure on symmetric spaces for $\Sp_2(A,\sigma)$ for every semisimple Hermitian algebra $(A,\sigma)$ over any real closed field $\K$.

\subsection{Complex structures model}\label{Comp_Str_Mod_A}
The first model we construct is the complex structures model.

\begin{df}
A \defin{complex structure} on a right $A$-module $V$ is an $A$-linear map $J\colon V\to V$ such that $J^2=-\Id$.
\end{df}

Let $V=A^2$ and $\omega$ be the standard symplectic form on $A^2$. For every complex structure $J$ on $A^2$, we define the following $\sigma$-sesquilinear form
$$\begin{matrix}
h_J\colon & A^2\times A^2 & \to & A \\
& (x,y) & \mapsto & \omega(J(x),y)
\end{matrix}$$

We remind the definition of a $\sigma$-inner product:

\begin{df}
A $\sigma$-sesquilinear form $h$ on $(A^2,\omega)$ is called \defin{$\sigma$-inner product} if $h$ is $\sigma$-symmetric and for all regular $v\in A^2$, $h(v,v)\in A^\sigma_+$.
\end{df}

We define the \defin{complex structures model} by setting 
$$\mathfrak C:=\left\{J\text{ complex structure on $A^2$}\mid h_J\text{ is an $\sigma$-inner product}
\right\}.$$
We show now that $\mathfrak C$ is a model of the symmetric space of $\Sp_2(A,\sigma)$.

\begin{prop}
Let $J\in\mathfrak C$ and $w\in \Is(\omega)$, then $J(w)\in\Is(\omega)$.
\end{prop}

\begin{proof}
For $w\in \Is(\omega)$,
$$\omega(J(w),J(w))=h_J(w,J(w))=\sigma(h_J(J(w),w))=\sigma(\omega(w,w))=0,$$
therefore, $J(w)\in \Is(\omega)$.
\end{proof}

\begin{prop}\label{CompStr-SympBas}
Let $J$ be a complex structure on $A^2$. Then $J\in\mathfrak C$ if and only if there exists $w\in \Is(\omega)$ such that $(J(w),w)$ is a symplectic basis of $A^2$.
\end{prop}

\begin{proof}
1. Let $J\in\mathfrak C$ and $w'\in\Is(\omega)$. Since $h_J(w',w')=b\in A^\sigma_+$, we can take $w:=w'b^{-\frac{1}{2}}$, then $h_J(w,w)=1$. Then:
$$\omega(J(w),J(w))=h_J(w,J(w))=\sigma(h_J(J(w),w))=\sigma(\omega(w,w))=0,$$
$$\omega(J(w),w)=h_J(w,w)=1.$$
Therefore, $(J(w),w)$ is a symplectic basis of $A^2$.

2. Let $w\in A^2$ such that $(J(w),w)$ is a symplectic basis of $A^2$. Then,
$$h_J(w,w)=\omega(J(w),w)=1$$
$$h_J(J(w),J(w))=\omega(J^2(w),J(w))=-\omega(w,J(w))=1,$$
$$h_J(J(w),w)=\omega(J^2(w),w)=-\omega(w,w)=0.$$
Therefore, $(J(w),w)$ is an orthonormal basis for $h_J$ and in this basis $h_J$ is the standard $\sigma$-inner product, so $h_J$ is an $\sigma$-inner product.
\end{proof}

\begin{df}
The \defin{standard complex} structure on $A^2$ is the map
$$\begin{matrix}
J_0\colon & A^2 & \to & A^2 \\
& (x,y) & \mapsto & (y,-x)
\end{matrix}$$
\end{df}

\begin{teo}
$\Sp_2(A,\sigma)$ acts on $\mathfrak C$ by conjugation. This action is transitive. The stabilizer of the standard complex structure is $\KSp_2(A,\sigma)$.

In particular, $\mathfrak C \cong \Sp_2(A,\sigma)/\KSp_2(A,\sigma)$ is a model of the symmetric space of $\Sp_2(A,\sigma)$.
\end{teo}

\begin{proof}
1. First, we prove that $\Sp_2(A,\sigma)$ acts on $\mathfrak C$ by conjugation. Let $J\in \mathfrak C$, $g\in\Sp_2(A,\sigma)$. Consider $J':=g^{-1}Jg$. $(J')^2=g^{-1}J^2g=-\Id$ so $J'$ is a complex structure on $A^2$. For $x\in\Is(\omega)$, $g(x)\in\Is(\omega)$ and we obtain
$$h_{J'}(x,x)=\omega(J'(x),x)=\omega(g^{-1}Jg(x),x)=$$
$$=\omega(Jg(x),g(x))=h_J(g(x),g(x))\in A^\sigma_+.$$
Therefore, $h_{J'}$ is a $\sigma$-inner product on $A^2$, i.e. $J'\in \mathfrak C$.

2. Second, we prove that this action is transitive. Let $J\in \mathfrak C$, take a symplectic basis $(J(w),w)$ from the Proposition~\ref{CompStr-SympBas}. Since $\Sp_2(A,\sigma)$ acts transitively on symplectic bases, there exists $g\in\Sp_2(A,\sigma)$ which maps the standard symplectic basis to $(J(w),w)$. That means, $g$ maps the standard complex structure $J_0$ to $J$. So the action is transitive.

3. Finally, compute the stabilizer of $J_0$. $g\in\Stab_{\Sp_2(A,\sigma)}(J_0)$ if and only if
$g\in\Sp_2(A,\sigma)$ and $g\in \OO(h_{J_0})=\UU_2(A,\sigma)$, i.e.
\begin{equation*}
g\in\Sp_2(A,\sigma)\cap\UU_2(A,\sigma)= \KSp_2(A,\sigma).\qedhere
\end{equation*}
\end{proof}

\subsection{Upper half-space model}\label{Upperhalf_Real}

Now we describe the upper half-space model that generalizes the well-known upper half-plane model for $\SL(2,\R)$.

We denote as before by $A_\CC$ the complexification of $A$, i.e. $A_\CC:=A\otimes_\R\CC$. We extend $\sigma$ to $A_\CC$ complex linearly, i.e. $\sigma_\CC(x+yi):=\sigma(x)+\sigma(y)i$.

Every element of $z\in A_\CC^{\sigma_\CC}$ can be uniquely written as $z=x+yi$ where $x,y\in A^\sigma$. We denote by $\Ree(z):=x$, $\Imm(z):=y$. We also have a complex conjugation on $A_\CC$ given by $\bar z=x-yi$.

\begin{df}
The \defin{upper half-space} is
$$\mathfrak U:=\{z\in A_\CC^{\sigma_\CC}\mid \Imm(z)\in A^\sigma_+\}$$
\end{df}

\begin{prop}
The group $\Sp_2(A,\sigma)$ acts transitively on $\mathfrak U$ via \defin{M\"obius transformations}
$$z\mapsto M.z=(az+b)(cz+d)^{-1}\text{, where } M=\begin{pmatrix}a & b \\ c & d \end{pmatrix}.$$
The stabilizer of $1i$ is $\KSp_2(A,\sigma)$.

In particular, $\mathfrak U$ is a model of the symmetric space of $\Sp_2(A,\sigma)$.
\end{prop}

\begin{proof} First, we show that the action is well defined.
Since $\Sp_2(A,\sigma)$ is generated by matrices
$$\begin{pmatrix}
    a & 0 \\
    0 & \sigma(a)^{-1}
    \end{pmatrix},
\begin{pmatrix}
    0 & 1 \\
    -1 & 0
    \end{pmatrix},
\begin{pmatrix}
    1 & b \\
    0 & 1
    \end{pmatrix}$$
where $a\in A^\times$, $b\in A^\sigma$, we proof that $M.z\in \mathfrak U$ for all $z\in\mathfrak U$ on these generators.

If $M:=\begin{pmatrix}
    1 & b \\
    0 & 1
    \end{pmatrix}$ with $b\in A^\sigma$, then $M.z=z+b\in A^{\sigma}_\CC$ and $\Imm(M.z)=\Imm(z)\in A^\sigma_+$.

If $M:=\begin{pmatrix}
    0 & 1 \\
    -1 & 0
    \end{pmatrix}$,
then $M.z=-z^{-1}\in A^{\sigma}_\CC$. If $z=x+iy$, then $$z^{-1}=y^{-1}x(y+xy^{-1}x)^{-1}-i(y+xy^{-1}x)^{-1},$$
i.e. $\Imm(M.z)=(y+xy^{-1}x)^{-1}$. For $y\in A^\sigma_+$, also $y^{-1}\in A^\sigma_+$.

Let $y^{-1}=\sigma(p)p$ for some $p\in A^\times$, then
$$y+xy^{-1}x=y+\sigma(px)px\in A^\sigma_+.$$
Therefore, $\Imm(M.z)=(y+xy^{-1}x)^{-1}\in A^\sigma_+.$

If $M:=\begin{pmatrix}
    a & 0 \\
    0 & \sigma(a)^{-1}
    \end{pmatrix}$
for $a\in A^\times$, then $M.z=az\sigma(a)\in A^{\sigma}_\CC$ because $A^\sigma$ is closed by action of $A^\times$. $\Imm(M.z)=a\Imm(z)\sigma(a)\in A^\sigma_+$ because $A^\sigma_+$ is closed by action of $A^\times$.

Now, we check the transitivity. 
Let $z=x+yi\in\mathfrak U$ then $y=u^2$ for some $u\in (A^\sigma)^\times$. Then
$$\pi\left(
\begin{pmatrix}
1 & x \\
0 & 1
\end{pmatrix}
\begin{pmatrix}
u & 0 \\
0 & u^{-1}
\end{pmatrix}\right)=
\pi\left(
\begin{pmatrix}
u & xu^{-1} \\
0 & u^{-1}
\end{pmatrix}\right)=x+yi=z$$

Finally, let us find the stabilizer of $1i$.
$M=\begin{pmatrix}
a & b \\
c & d
\end{pmatrix}$
stabilizes $1i$ if and only if
$$1i=M.1i=(ai+b)(ci+d)^{-1}=(ai+b)(-c+di)^{-1}i.$$
So, $a=d$ and $c=-b$, i.e. $M\in\KSp_2(A,\sigma)$.
\end{proof}

\subsection{Symmetric space of \texorpdfstring{$\OO(h)$}{O(h)} for an indefinite form \texorpdfstring{$h$}{h}}\label{SymSp_O(h)}

In order to describe other models of the symmetric space of $\Sp_2(A,\sigma)$, we consider the following $\sigma$-sesquilinear $\sigma$-symmetric forms on $A^2$:

\begin{df} A $\sigma$-sesquilinear $\sigma$-symmetric form $h$ on $A^2$ such that there exist a basis $(e_1,e_2)$ of $A^2$ such that $h(e_1,e_1)=-1$, $h(e_2,e_2)=1$, $h(e_1,e_2)=0$ is called \defin{indefinite}.

The \defin{standard indefinite form} $h_{st}$ is the $\sigma$-sesquilinear $\sigma$-symmetric form on $A^2$ given by the matrix $\begin{pmatrix}-1 & 0 \\ 0 & 1 \end{pmatrix}$ in the standard basis $((1,0)^T,(0,1)^T)$ of $A^2$.
\end{df}

We define the group of symmetries of $h$ by :
$$\OO(h):=\{g\in\Aut(A^2)\mid h(gx,gy)=h(x,y)\text{ for all }x,y\in A^2\},$$
and set 
$$\mathcal X_{\OO(h)}:=\{xA\mid x\in A^2\text{ such that }h(x,x)\in A^\sigma_+\},\;\; \mathcal X:=\mathcal X_{\OO(h_{st})}.$$

\begin{rem}The space $\mathcal X_{\OO(h)}$ is well defined because if $xA=yA$, i.e. there exists $a\in A^\times$ such that $y=xa$, then $$h(y,y)=\sigma(a)h(x,x)a=\sigma(a)\sigma(p)pa=\sigma(pa)pa\in A^\sigma_+$$
where $p\in A^\times$, $\sigma(p)p=h(x,x)\in A^\sigma_+$.
\end{rem}

\begin{rem}
Since $\Aut(A^2)$ acts transitively on bases of $A^2$, all $\OO(h)$ are isomorphic for indefinite $h$. Therefore, all $\mathcal X_{\OO(h)}$ are also isomorphic.
\end{rem}

\begin{prop}
$\OO(h_{st})$ acts transitively on $\mathcal X$ with stabilizer of $(0,1)^TA$ equal to $U_{(A,\sigma)}\times U_{(A,\sigma)}$ diagonally embedded into $\OO(h_{st})$.
\end{prop}

\begin{proof}
Since $h_{st}((0,1)^T,(0,1)^T)=1\in A^\sigma_+$, the line $(0,1)^TA\in\mathcal X$. Let $vA\in \mathcal X$ for some $v\in A^2$. Since $h_{st}(v,v)\in A^\sigma_+$, there exists $p\in A$ such that $h_{st}(v,v)=\sigma(p)p$. Let $v':=vp^{-1}$, then $h(v',v')=1$ and $v'A=vA$.

Consider the vector
$v':=(x,\sigma(v_2)^{-1}\sigma(v_1)x)^T$
where $v=(v_1,v_2)^T$, $x=(1+v_1\sigma(v_1))^\frac{1}{2}$. Then an easy calculation shows that $h(v',v')=-1$ and $h(v,v')=0$. So we can take the following matrix $M:=(v',v)\in \OO(h_{st})$. Since $M(0,1)^T=v$, we obtain $M(0,1)^TA=vA$, i.e. $\OO(h_{st})$ acts transitively on $\mathcal X$.

Now, we compute the stabilizer of $(0,1)^TA$. Let
$$M:=\begin{pmatrix}
a & b \\
c & d
\end{pmatrix}\in \OO(h_{st})$$
stabilize $(0,1)^TA$, then $M(0,1)^T=(b,d)$, i.e. $b=0$. Moreover
$$\begin{pmatrix}
-1 & 0 \\
0 & 1
\end{pmatrix}= \begin{pmatrix}
\sigma(a) & \sigma(c) \\
0 & \sigma(d)
\end{pmatrix}
\begin{pmatrix}
-1 & 0 \\
0 & 1
\end{pmatrix}
\begin{pmatrix}
a & 0 \\
c & d
\end{pmatrix}=\begin{pmatrix}
-\sigma(a) & \sigma(c) \\
0 & \sigma(d)
\end{pmatrix}\begin{pmatrix}
a & 0 \\
c & d
\end{pmatrix}=$$
$$=\begin{pmatrix}
-\sigma(a)a+\sigma(c)c & \sigma(c)d \\
0 & \sigma(d)d
\end{pmatrix}.$$
Therefore, $\sigma(d)d=1$, i.e. $d$ is invertible. So we obtain $c=0$ and $\sigma(a)a=1$, i.e. $a,d\in U_{(A,\sigma)}$.
\end{proof}

\begin{prop}
The group $U_{(A,\sigma)}\times U_{(A,\sigma)}$ is diagonally embedded into $\OO(h_{st})$ as a maximal compact subgroup of $\OO(h_{st})$.

In particular, $\mathcal X$ is a model of the symmetric space of $\OO(h_{st})$.
\end{prop}

\begin{proof}
First, note that the Lie algebra of $\OO(h_{st})$ is:
$$\oo(h_{st})=\left\{\begin{pmatrix}
a & b \\
\sigma(b) & d
\end{pmatrix}\midwd \sigma(a)=-a\in A, \sigma(d)=-d\in A, b\in A
\right\}.$$
Assume, $K$ is compact subgroup of $\OO(h_{st})$ that contains $U_{(A,\sigma)}\times U_{(A,\sigma)}$ as a proper subgroup. Then $\Lie(K)$ contains an element
$\begin{pmatrix}
a & b \\
\sigma(b) & d
\end{pmatrix}$ with $b\neq 0$. Therefore,
$$x:=\begin{pmatrix}
0 & b \\
\sigma(b) & 0
\end{pmatrix}=\begin{pmatrix}
a & b \\
\sigma(b) & d
\end{pmatrix}-\begin{pmatrix}
a & 0 \\
0 & d
\end{pmatrix}\in\Lie(K)$$
and
$$tx=\begin{pmatrix}
0 & tb \\
t\sigma(b) & 0
\end{pmatrix}\in\Lie(K)$$
for all $t\in\R$. Take a polar decomposition of $b=uy$ where $u\in U_{(A,\sigma)}$, $y\in A^\sigma$. We take the spectral decompositions of $y$:
$y=\sum_{i=1}^k\lambda_ic_i$ where $(c_i)$ is a complete system of orthogonal idempotents, $\lambda_1,\dots,\lambda_k\in\R$.

Further,
$$x^2=\begin{pmatrix}
b\sigma(b) & 0 \\
0 & \sigma(b)b
\end{pmatrix}.$$
Therefore,
$$b\sigma(b)=u\sum_{i=1}^k\lambda_i^2c_iu^{-1},\; \sigma(b)b=\sum_{i=1}^k\lambda_i^2c_i$$
and
$$\exp(tx)=\begin{pmatrix}
u\sum_{i=1}^k\cosh(t\lambda_i)c_iu^{-1} & u\sum_{i=1}^k\sinh(t\lambda_i)c_i \\
\sum_{i=1}^k\sinh(t\lambda_i)c_iu^{-1} & \sum_{i=1}^k\cosh(t\lambda_i)c_i
\end{pmatrix}\in K.$$
For $t$ going to infinity, $\exp(xt)$ does not converge even up to taking subsequence unless all $\lambda_i=0$. But this means that $b=0$, so we obtain $K=U_{(A,\sigma)}\times U_{(A,\sigma)}$. This contradicts the assumption that $U_{(A,\sigma)}\times U_{(A,\sigma)}$ is a proper subgroup of $K$.
\end{proof}

\begin{prop}\label{Precomp_Mod_O(h)}
The map
$$\begin{array}{cccl}
\Phi\colon & \mathcal X & \to & \mathring{D}(A,\sigma):=\{c\in A\mid 1-\sigma(c)c\in A^\sigma_+\}\\
 & (a,b)^TA & \mapsto & ab^{-1}
\end{array}$$
is a homeomorphism.
\end{prop}

\begin{proof}
Let $xA\in\mathcal X$ then $x=(a,b)^T$ with $-\sigma(a)a+\sigma(b)b\in A^\sigma_+$, i.e. there exists $p\in A^\times$ such that
$$-\sigma(a)a+\sigma(b)b=\sigma(p)p.$$
Therefore,
$$\sigma(b)b=\sigma(p)p+\sigma(a)a\in A^\sigma_+,$$
i.e. $b\in A^\times$. So for $c=ab^{-1}$, $xA=(c,1)^T A$. Moreover, for every line $xA\in\mathcal X$, the element $c\in A$ such that $xA=(c,1)^TA$ is well defined and $1-\sigma(c)c\in A^\sigma_+$.

For every $c\in\mathring{D}(A,\sigma)$, the line $(c,1)^TA\in\mathcal X$ because $$h_{st}((c,1)^T,(c,1)^T)=1-\sigma(c)c\in A^\sigma_+.$$
Therefore, $\Phi$ is a homeomorphism.
\end{proof}

\begin{cor}
$\OO(h_{st})$ acts on $\mathring{D}(A,\sigma)$ via M\"obius transformations 
$$z\mapsto M.z=(az+b)(cz+d)^{-1}\text{, where } M=\begin{pmatrix}a & b \\ c & d \end{pmatrix}.$$
This transformation is called \defin{M\"obius transformation}.
\end{cor}

\begin{rem}
Since $(A,\sigma)$ is Hermitian, by Proposition~\ref{comp_disc} the domain $\mathring{D}(A,\sigma)$ is precompact.
\end{rem}

\subsection{Projective model}\label{Proj_Mod_R}

Now we use $\mathcal X$ to construct the projective model of the symmetric space of $\Sp_2(A,\sigma)$. Let $A_\CC = A\otimes \CC$ and denote by $\sigma_\CC:A_\CC \rightarrow A_\CC$ the $\CC$-linear extension of $\sigma$, i.e. $$\sigma_\CC(x+iy)=\sigma(x)+i\sigma(y)$$
for every $x,y\in A$ and by $\bar\sigma_\CC$ the $\CC$-antilinear extension of $\sigma$, i.e. $$\bar\sigma_\CC(x+iy)=\sigma(x)-i\sigma(y)$$
for every $x,y\in A$.

As we have seen in the Corollary~\ref{AC-Herm_A}, $(A_\CC,\bar\sigma_\CC)$ is Hermitian. We extend $\omega$ in the following way:
$$\omega_\CC(x,y):=\sigma(x)^T\Om y.$$
The following $\bar\sigma$-sesquilinear form is an indefinite form on $A_\CC^2$:
$$h(x,y):=\bar\sigma_\CC(x)^T
\begin{pmatrix}
0 & i \\
-i & 0
\end{pmatrix}y=i\omega_\CC(\bar x,y).$$
Indeed,
$$h(y,x)=\bar\sigma_\CC(y)^T
\begin{pmatrix}
0 & i \\
-i & 0
\end{pmatrix}x=\bar\sigma_\CC\left(\sigma_1(x)^T
\begin{pmatrix}
0 & i \\
-i & 0
\end{pmatrix}y\right)=\bar\sigma_\CC(h(x,y)).$$
Then in the basis $e_1:=\left(\frac{1}{\sqrt{2}},\frac{i}{\sqrt{2}}\right)^T$, $e_2:=\left(\frac{1}{\sqrt{2}},-\frac{i}{\sqrt{2}}\right)^T$, the form $h$ is represented by the matrix $\begin{pmatrix} -1 & 0 \\ 0 & 1\end{pmatrix}$, i.e. $h$ is a $\bar\sigma_\CC$-sesquilinear indefinite form on $A_\CC^2$.

Note, $\Sp_2(A,\sigma)$ acts on $A^2_\CC$ preserving $\omega_\CC$ and $h$.

\begin{df}
The space
$$\mathfrak P:=\{vA_\CC\mid v\in\Is(\omega_\CC),\;h(v,v)\in (A_\CC^{\bar\sigma})_+\}=\Is(\omega_\CC)\cap \mathcal X_{\OO(h)}$$
is called the \defin{projective model}.
\end{df}

To justify this Definition, we prove the following Proposition:

\begin{prop} The group $\Sp_2(A,\sigma)$ acts transitively on $\mathfrak P$ with the stabilizer of $(i,1)^TA_\CC$ equal to $\KSp_2(A,\sigma)$.

In particular, $\mathfrak P$ is a model of the symmetric space of $\Sp_2(A,\sigma)$.
\end{prop}

\begin{proof}
Let $v\in\Is(\omega_\CC)$ such that $h(v,v)\in (A_\CC^{\bar\sigma})_+$, so $vA_\CC\in \mathfrak P$. We can renormalize $v$ so that $h(v,v)=2$. We write $v=u+wi$, then $(w,u)$ is a symplectic basis of $A^2$. Indeed,
$$2=h(v,v)=i\omega_\CC(u-iw,u+iw)=i(\omega(u,u)+\omega(w,w)+i(\omega(u,w)-\omega(w,u)))=$$
$$=\omega(w,u)-\omega(u,w)+i(\omega(u,u)+\omega(w,w));$$
$$0=\omega_\CC(v,v)=\omega(u,u)-\omega(v,v)+i(\omega(u,v)+\omega(w,u)).$$
Therefore, $\omega(u,u)=\omega(w,w)=0$ and $\omega(w,u)=1$, i.e. $(w,u)$ is a symplectic basis of $A^2$.

Since $\Sp_2(A,\sigma)$ acts transitively on symplectic bases of $A^2$, $\Sp_2(A,\sigma)$ acts transitively on $\mathfrak P$.

Now compute the stabilizer of $l_0:=(i,1)^TA_\CC$. Let $M=(M_{ij})\in\Sp_2(A,\sigma)$. $M(l_0)=l_0$ if and only if
$$M_{11}i+M_{12}=(M_{21}i+M_{22})i,$$
i.e. $M_{11}=M_{22}$, $M_{21}=-M_{12}$. This holds if and only if $M\in\UU_2(A,\sigma)\cap \Sp_2(A,\sigma)=\KSp_2(A,\sigma)$.
\end{proof}

\begin{df}
The space $\mathfrak P$ is called the 
the \defin{projective model} of the symmetric space of $\Sp_2(A,\sigma)$.
\end{df}

Finally, we describe the $\Sp_2(A,\sigma)$-equivariant homeomorphism between the complex structure model $\mathfrak C$ and the projective  model $\mathfrak P$. For this, we extend every complex structure $J$ on $A^2$ to a complex structure $J_\CC$ on $A^2_\CC$ in the $\CC$-linear way.

\begin{prop}
For every complex structure $J\in \mathfrak C$, there exist regular $x,y\in A^2_\CC$ such that $J_\CC(x)=xi$, $J_\CC(y)=-yi$. Elements $x,y$ are uniquely defined up to multiplication by elements of $A_\CC^\times$.
\end{prop}

\begin{proof}
Since $\Sp_2(A,\sigma)$ acts transitively on $\mathfrak C$, it is enough to prove the proposition for the standard complex structure $J_0$.

Since $J_0(a,b)^T=(b,-a)^T,$ we obtain $(b,-a)^T=(a,b)^Ti$ if and only if $b=ai$, i.e.
$$x=(a,ai)^T=(1,i)^Ta,$$
where $a\in A_\CC$ arbitrary element. For $a\in A^\times$, x is regular.
Analogously, $y=(i,1)^Ta$ where $a\in A_\CC^\times$ arbitrary element.
\end{proof}

For a complex structure $J\in\mathfrak C$, we denote by $l_J$ the $A_\CC$-line $yA_\CC$ such that $J_\CC(y)=-yi$.

\begin{cor}
The map
$$\begin{array}{cccl}
F\colon & \mathfrak C & \to & \mathfrak P\\
 & J & \mapsto & l_J
\end{array}$$
defines is an $\Sp_2(A,\sigma)$-equivariant homeomorphism.
\end{cor}

\subsection{Precompact model}

Lastly, we define the precompact model of the symmetric space of $\Sp_2(A,\sigma)$.

For this we consider the following $\Sp_2(A_\CC,\sigma_\CC)$-transformation that maps the indefinite form $h$ introduced in the previous section to the standard indefinite form $h_{st}$:
$$T:=\frac{1}{\sqrt{2}}\begin{pmatrix}
1 & i \\
i & 1
\end{pmatrix},$$
i.e. $\bar\sigma(T)^T[h]T=\diag(-1,1)=[h_{st}]$. Since $T\in\Sp_2(A_\CC,\sigma_\CC)$, it stabilizes the set $\Is(\omega_\CC)$.

\begin{df}
The space 
$$\mathfrak{B}:=\mathring{D}(A^{\sigma_\CC}_\CC,\bar\sigma_\CC):=\{c\in A^{\sigma_\CC}_\CC\mid 1-\bar cc\in (A^{\bar\sigma_\CC}_\CC)_+\}$$
is called the \defin{precompact model}.
\end{df}

To justify this Definition, we prove the following Proposition:

\begin{prop}\label{Proj-Precomp-R_A}
The map
$$\begin{array}{cccl}
\Phi\colon & T^{-1}\mathfrak P & \to & \mathring{D}(A^{\sigma_\CC}_\CC,\bar\sigma_\CC)\\
 & (a,b)^TA_\CC & \mapsto & ab^{-1}
\end{array}$$
is a homeomorphism. The set $\mathring{D}(A^{\sigma_\CC}_\CC,\bar\sigma_\CC)$ is precompact on $A^{\sigma_\CC}_\CC$.

In particular, $\mathring{D}(A^{\sigma_\CC}_\CC,\bar\sigma_\CC)$ is a model of the symmetric space of $\Sp_2(A,\sigma)$.
\end{prop}

\begin{proof} Let $v=(v_1,v_2)^T\in \Is(\omega_\CC)$ such that $vA_\CC\in\mathfrak P$ and $v=u+wi$ where $(w,u)$ is a symplectic basis of $(A^2,\omega)$. Then:
$$2=h(v,v)=h_{st}(T^{-1}v,T^{-1}v)=-\bar\sigma(x_1)x_1+\bar\sigma(x_2)x_2\in A^{\bar\sigma}_+$$
where $T^{-1}v=:(x_1,x_2)^T$. Therefore, $\bar\sigma(x_2)x_2=\bar\sigma(x_1)x_1+2\in A^{\bar\sigma}_+$ because $A^{\bar\sigma}_+$ is a proper convex cone. This means that $x_2$ is invertible, i.e. $x_2\in A^\times$, $(c,1)^T:=(x_1x_2^{-1},1)^T\in\Is(\omega_\CC)$ and $(c,1)^TA_\CC=(T^{-1}v)A_\CC$.

$(c,1)^T\in\Is(\omega_\CC)$ if and only if $c\in A^{\sigma}_\CC$, and $h((c,1)^T,(c,1)^T)=1-\bar c c\in A^{\bar\sigma}_+$. Therefore, $\Phi(T^{-1}v)\in \mathring{D}(A^\sigma_\CC,\bar\sigma)$.

The map $\Phi$ is injective because $T$ is injective and, if $x_1x_2^{-1}=\Phi(x_1,x_2)^T=\Phi(y_1,y_2)^T=y_1y_2^{-1}$, then $(y_1,y_2)^T=(x_1,x_2)^Tx_2^{-1}y_2$, i.e. $(x_1,x_2)^TA_\CC=(y_1,y_2)^TA_\CC$.

The map $\Phi$ is surjective because for every $c\in \mathring{D}(A^\sigma_\CC,\bar\sigma)$, $(c,1)^TA_\CC=(T^{-1}v)A_\CC$ for $v:=T(c,1)^T\sqrt{2}(1-\bar c c)^{-\frac{1}{2}}$. Then $v\in \Is(\omega_\CC)$ and $h(v,v)=2$. Therefore, $v=u+wi$ for $(w,u)$ a symplectic basis of $(A^2,\omega)$. Therefore, $vA_\CC\in T^{-1}\mathfrak P$.

The set $\mathring{D}(A^\sigma_\CC,\bar\sigma)$ is precompact in $A^{\sigma_\CC}_\CC$ because it is a subset of the following domain:
$$D(A^{\sigma_\CC}_\CC,\bar\sigma):=\{a\in A^{\sigma_\CC}_\CC\mid 1-\bar aa\in (A^{\bar\sigma}_\CC)_{\geq 0}\}\subseteq A^{\sigma_\CC}_\CC$$
that is compact by Proposition~\ref{comp_disc}.
\end{proof}

\begin{rem}
The group $T^{-1}\Sp_2(A,\sigma)T<\Sp_2(A_\CC,\sigma_\CC)$ acts on $\mathring{D}(A^{\sigma_\CC}_\CC,\bar\sigma_\CC)$ by M\"obius transformations.
\end{rem}

\subsection{Connection between models}\label{Connection_btw_models}

In this section, we consider  $\Sp_2(A,\sigma)$-equivariant homeomorphisms between the projective model, the upper half-space model and precompact model of the symmetric space for $\Sp_2(A,\sigma)$.

It is easy to check that the following map:
$$\begin{matrix}
F \colon & \mathfrak P & \to & \mathfrak U\\
 & (x_1,x_2)^TA_\CC & \mapsto & x_1x_2^{-1}
\end{matrix}$$
is an $\Sp_2(A,\sigma)$-equivariant homeomorphism. As we have seen in the Proposition~\ref{Proj-Precomp-R_A}, the map
$$\Phi\circ T^{-1}\colon\mathfrak P\to \mathring{D}(A^{\sigma_\CC}_\CC,\bar\sigma_\CC):=\{c\in A^{\sigma_\CC}_\CC\mid 1-\bar cc\in (A^{\bar\sigma_\CC}_\CC)_+\}.$$
defines another $\Sp_2(A,\sigma)$-equivariant homeomorphism.

The maps $F$ and $\Phi\circ T^{-1}$ can be seen as different coordinate charts for the projective model $\mathfrak P$ of the symmetric space for $\Sp_2(A,\sigma)$.

\subsection{Compactification and Shilov boundary}

In this section, we construct a natural compactification of the symmetric space of $\Sp_2(A,\sigma)$.

As we have seen, the precompact model $\mathring{D}(A^{\sigma_\CC}_\CC,\bar\sigma_\CC)$ is a precompact domain in $A^{\sigma_\CC}_\CC$, so 
taking the topological closure of $\mathring{D}(A^{\sigma_\CC}_\CC,\bar\sigma_\CC)$ in $A^{\sigma_\CC}_\CC$, we obtain the compactification
$$D(A^{\sigma_\CC}_\CC,\bar\sigma_\CC):=\{c\in A^{\sigma_\CC}_\CC\mid 1-\bar cc\in (A^{\bar\sigma_\CC}_\CC)_{\geq 0}\}$$
of $\mathring{D}(A^{\sigma_\CC}_\CC,\bar\sigma_\CC)$.

\begin{df}
We call $$\check{S}(A^{\sigma_\CC}_\CC,\bar\sigma_\CC):=\{c\in A^{\sigma_\CC}_\CC\mid 1-\bar cc=0\}$$
the \defin{Shilov boundary} of the precompact model $\mathring{D}(A^{\sigma_\CC}_\CC,\bar\sigma_\CC)$.
\end{df}

Note, that
$\check{S}(A^{\sigma_\CC}_\CC,\bar\sigma_\CC)=U_{(A_\CC,\bar\sigma_\CC)}\cap A^{\sigma_\CC}_\CC.$
So the Shilpov boundary is a compact subspace of $D(A^{\sigma_\CC}_\CC,\bar\sigma_\CC)$. 

\begin{rem}
The map $\Phi^{-1}$ extends to the boundary of $D(A^{\sigma_\CC}_\CC,\bar\sigma_\CC)$ and remains continuous and bijective. Since the boundary is compact, it is a homeomorphism. Therefore, we can see the boundary also in the projective model. In particular, we can see the Shilov boundary there.
\end{rem}

The next Proposition describes the Shilov boundary in the projective model.

\begin{prop}
The preimage of the Shilov boundary $\check{S}(A^{\sigma_\CC}_\CC,\bar\sigma_\CC)$ in $\PP(\Is(\omega_\CC))$ under the map $\Phi\circ T^{-1}$ gives a compact subset of the boundary of the projective model. It consists of all lines of the form $xA_\CC$ such that $x\in\Is(\omega)$ regular.
\end{prop}

\begin{proof}
Note that the line $l\in \Is(\omega)$ is of the form $xA_\CC$ for some $x\in\Is(\omega)$ if and only if $\bar l =l$.

Assume $c\in \check{S}(A^{\sigma_\CC}_\CC,\bar\sigma_\CC)$, i.e. $\bar c^{-1}=c$. Then
$$(\Phi\circ T^{-1})\overline{T\circ\Phi^{-1}(c)}=\Phi\left(\begin{pmatrix}0 & i\\ i & 0\end{pmatrix}\Clm{\bar c}{1}\right)=\Phi\left(\Clm{1}{\bar ci}\right)=\bar c^{-1}=c$$
i.e. for $l=(c,1)^TA_\CC$, $\bar l=l$.

If we take a line $xA_\CC$ for some $x=(x_1,x_2)^T\in\Is(\omega)$, then
$$c:=(\Phi\circ T^{-1})(xA_\CC)=(x_1-ix_2)(-ix_1+x_2)^{-1}.$$
Since $x\in\Is(\omega)$, $c\in A^\sigma_\CC$
$$\bar c c=(x_1+ix_2)(ix_1+x_2)^{-1}(x_1-ix_2)(-ix_1+x_2)^{-1}=$$
$$=i(x_1+ix_2)(x_1-ix_2)^{-1}(x_1-ix_2)(-ix_1+x_2)^{-1}=$$
$$=i(x_1+ix_2)(-ix_1+x_2)^{-1}=(x_1+ix_2)(x_1+ix_2)^{-1}=1.$$
Therefore, $(\Phi\circ T^{-1})(xA)\in \check{S}(A^{\sigma_\CC}_\CC,\bar\sigma_\CC)$.
\end{proof}

\begin{cor}
The space $\PP(\Is(\omega))$ of isotropic lines of $(A^2,\omega)$  embedded into $\PP(\Is(\omega_\CC))$ as:
$$xA \mapsto xA_\CC$$
is a Shilov boundary in the projective model. This is a closed (even compact) orbit of the action of $\Sp_2(A,\sigma)$ on the boundary of the projective model.
\end{cor}

\section{Models for the symmetric space of \texorpdfstring{$\Sp_2(A,\sigma)$}{Sp2(A,sigma)} over complexified algebras}\label{AC-models}

The goal of this Chapter is to construct different models of the symmetric space for $\Sp_2(A,\sigma)$ where $A$ is the complexification of a real Hermitian algebra.

To simplify the notation in this and the next chapters, we denote by $(A_\R,\sigma_\R)$ a real Hermitian algebra. By $A=A_\R\otimes_\R\CC$, we denote the complexification of $A_\R$. The complex linear extension of $\sigma_\R$ is denoted by $\sigma$, the complex antilinear extension of $\sigma_\R$ is denoted by $\bar\sigma$.  

Similar to the approach for real Hermitian algebras, where we considered their complexification, we consider here quaternionifications of $A$ (see Section~\ref{Quatern_ext}).

As we noticed in the previous chapter, this construction works for every real closed field. All the models we are going to describe are semi-algebraic sets, and they are semi-algebraically isomorphic to each other. This allows us to define a natural semi-algebraic structure on symmetric spaces of $\Sp_2(A,\sigma)$ for a complexification of every semisimple Hermitian algebra $(A,\sigma)$ over any real closed field $\K$.

\subsection{Quaternionic structures model}\label{Quat_Str_Mod}

Let $(A_\R,\sigma_\R)$ be a Hermitian algebra with anti-involution. We consider the complexification $A:=A_\R\otimes_\R\CC$. As we have seen in Corollary~\ref{AC-Herm_A}, $(A,\bar\sigma)$ is a Hermitian algebra.

\begin{df}
A \defin{quaternionic structure} on an right $A$-module $V$ is an additive map $J\colon V\to V$ such that $J^2=-\Id$ and $J(xa)=J(x)\bar a$ for all $x\in V$, $a\in A$.
\end{df}

Let $V=A^2$ and $\omega$ be the standard symplectic form on $A^2$. For every quaternionic structure $J$ on $A^2$, we can define the form:
$$\begin{matrix}
h_J\colon & A^2\times A^2 & \to & A \\
& (x,y) & \mapsto & \omega(J(x),y)
\end{matrix}$$
that is $\bar\sigma$-sesquilinear. Indeed, for $a_1,a_2\in A$
$$h_J(xa_1,ya_2)=\omega(J(xa_1),ya_2)=\omega(J(x)\bar a_1,ya_2)=\bar\sigma(a_1)h_J(x,y)a_2.$$

\begin{df}
The space:
$$\mathfrak C:=\{J\text{ quaternionic structure on $A^2$}\mid h_J\text{\text{ is a $\bar\sigma$-inner product}\}}$$
is called \defin{quaternionic structures model}.
\end{df}

We show that $\mathfrak C$ is a model of the symmetric space of $\Sp_2(A,\sigma)$.

\begin{df}
The \defin{standard quaternionic structure} on $A^2$ is the map
$$\begin{matrix}
J_0\colon & A^2 & \to & A^2 \\
& (x,y) & \mapsto & (\bar y,-\bar x)
\end{matrix}$$
\end{df}

\begin{rem}
$h_{J_0}$ is the standard $\bar\sigma$-inner product on $A^2$.
\end{rem}

\begin{prop}\label{QuatStr-SympBas}
Let $J$ be a quaternionic structure on $A^2$. $J\in\mathfrak C$ if and only if there exists a regular isotropic $w\in A^2$ such that $(J(w),w)$ is a symplectic basis.
\end{prop}

\begin{proof}
1. Let $J\in\mathfrak C$ and $w\in A^2$ some regular isotropic element. Since $h_J(w,w)\in A^{\bar\sigma}_+$, we can normalize $w$ so that $h_J(w,w)=1$. Then:
$$\omega(J(w),J(w))=h_J(w,J(w))=\bar\sigma(h_J(J(w),w))=\bar\sigma(\omega(w,w))=0,$$
$$\omega(J(w),w)=h_J(w,w)=1.$$
Therefore, $(J(w),w)$ is a $\sigma$-symplectic basis.

2. Let $w\in A^2$ and $(J(w),w)$ is a $\sigma$-symplectic basis. Then,
$$h_J(w,w)=\omega(J(w),w)=1$$
$$h_J(J(w),J(w))=\omega(J^2(w),J(w))=\omega(J(w),w)=1,$$
$$h_J(J(w),w)=\omega(J^2(w),w)=-\omega(w,w)=0.$$
Therefore, $(w,J(w))$ is an orthonormal basis for $h_J$, and in this basis, $h_J$ is the standard $\sigma$-inner product, so $h_J$ is an $\bar\sigma$-inner product.
\end{proof}

\begin{cor}
For every $J\in\mathfrak C$, for every $v\in\Is(\omega)$, $J(v)\in\Is(\omega)$.
\end{cor}

\begin{teo}
$\Sp_2(A,\sigma)$ acts on $\mathfrak C$ in the following way:
$$(g,J)\mapsto g^{-1}\circ J\circ g.$$
This action is transitive. The stabilizer of the standard quaternionic structure is $\KSp^c_2(A,\sigma)$.

In particular, $\mathfrak C$ is a model of the symmetric space of $\Sp_2(A,\sigma)$.
\end{teo}

\begin{proof}
1. First, we prove that $\Sp_2(A,\sigma)$ acts on $\mathfrak C$ by conjugation. Let $J\in \mathfrak C$, $g\in\Sp_2(A,\sigma)$. Consider $J':=g^{-1}\circ J\circ g$. Then
$$(J')^2=g^{-1}\circ J\circ g\circ g^{-1}\circ J\circ g=-\Id.$$
So $J'$ is a quaternionic structure on $A^2$. For a regular $x\in A^2$,
$$h_{J'}(x,x)=\omega(J'(x),x)=\omega(g^{-1}Jg(x),x)=\omega(Jg(x),g(x))=$$
$$=h_J(g(x),g(x))\in A^{\bar\sigma}_+.$$
Therefore, $h_{J'}$ is an inner product on $A^2$, i.e. $J'\in \mathfrak C$.

2. Second, we prove that the action is transitive. Let $J\in \mathfrak C$, take a symplectic basis $(J(w),w)$ from the Proposition~\ref{QuatStr-SympBas}. Since $\Sp_2(A,\sigma)$ acts transitively on symplectic bases, there exists $g\in\Sp_2(A,\sigma)$ which maps the standard symplectic basis to $(J(w),w)$. That means, $g$ maps the standard complex structure $J_0$ to $J$. So the action is transitive.

3. Finally, compute the stabilizer of $J_0$. $g\in\Stab_{\Sp_2(A,\sigma)}(J_0)$ if and only if
$g\in\Sp_2(A,\sigma)$ and $g\in \OO(h_{J_0})=\UU_2(A_\CC,\bar\sigma)$, i.e.
\begin{equation*}
g\in\Sp_2(A,\sigma)\cap\UU_2(A,\bar\sigma)= \KSp^c_2(A,\sigma).\qedhere
\end{equation*}
\end{proof}

\begin{rem}
Since any quaternionic structure is a $\CC$-antilinear map, if we write the action of $\Sp_2(A,\sigma)$ in the matrix form, we need to add the complex conjugation: i.e. let $[J]$ be the matrix for the quaternionic structure $J$, then
$$[g^{-1}\circ J\circ g]=g^{-1}[J]\bar g.$$
\end{rem}

\subsection{Upper half-space model for \texorpdfstring{$\Sp_2(A,\sigma)$}{Sp2(A,sigma)}}\label{Upperhalf_Comp}

Let $A_\R$ be an Hermitian $\R$-algebra with an anti-involution $\sigma_\R$. We assume $A:=A_\R\otimes_\R\CC\{I\}$ to be the complexification of $A_\R$. We denote here the imaginary unit by $I$ because the algebra $A$ sometimes is already a complex algebra where we just forget about its complex structure, so it may contain $i$ as an element. In order to be more precise, we do not use the letter $i$ in our construction.

We denote by $\sigma$ the complex linear extension of $\sigma_\R$. We denote by $\bar\sigma$ the complex antilinear extension of $\sigma_\R$.

We denote by $A_\HH$ the quaternionification of $A_\R$, i.e. $A_\HH:=A_\R\otimes_\R\HH\{I,J,K\}$. By our convention form the previous Section~\ref{Quatern_ext}, we have $A\subset A_\HH$.

We extend $\sigma$ to $A_\HH$ quaternionic linearly, i.e. $$\sigma_0:=\sigma(x)+J\sigma(y)=\sigma(x)+\sigma(\bar y)J=\sigma(x)+\bar\sigma(y)J.$$
So $A^{\sigma_0}_\HH=\Fix_{A_\HH}(\sigma_0)=A^{\sigma}\oplus A^{\bar\sigma}J$ is well defined.

Every element of $z\in A_\HH^{\sigma_0}$ can be uniquely written as $z=x+yJ$ where $x\in A^{\sigma}$, $y\in A^{\bar\sigma}$. We denote by $\Ree(z):=x$, $\Imm(z):=y$. We also have a quaternionic conjugation on $A_\HH$ given by $\bar z=\bar x-J\bar y=\bar x-yJ$.

\begin{df}
The \defin{upper half-space} is
$$\mathfrak U:=\{z\in A_\HH^{\sigma_0}\mid \Imm(z)\in A^{\bar\sigma}_+\}$$
\end{df}

\begin{prop}
The group $\Sp_2(A,\sigma)$ acts on $\mathfrak U$ via \defin{M\"obius transformation}
$$z\mapsto M.z=(az+b)(cz+d)^{-1}\text{, where } M=\begin{pmatrix}a & b \\ c & d \end{pmatrix}$$
transitively with the stabilizer of $1J$ equal to $\KSp^c_2(A,\sigma)$.

In particular, $\mathfrak U$ is a model for the symmetric space for $\Sp_2(A,\sigma)$.
\end{prop}

\begin{proof} First, we show that the action is well defined.
Since $\Sp_2(A,\sigma)$ is generated by matrices
$$\begin{pmatrix}
    a & 0 \\
    0 & \sigma(a)^{-1}
    \end{pmatrix},
\begin{pmatrix}
    0 & 1 \\
    -1 & 0
    \end{pmatrix},
\begin{pmatrix}
    1 & b \\
    0 & 1
    \end{pmatrix}$$
where $a\in A^\times$, $b\in A^\sigma$, we proof $M.z\in \mathfrak U$ on these generators.

If $M:=\begin{pmatrix}
    1 & b \\
    0 & 1
    \end{pmatrix}$ with $b\in A^\sigma$, then $M.z=z+b\in A^{\sigma_0}_\HH$ and $\Imm(M.z)=\Imm(z)\in A^{\bar\sigma}_+$.

If $M:=\begin{pmatrix}
    0 & 1 \\
    -1 & 0
    \end{pmatrix}$,
then $M.z=-z^{-1}\in A^{\sigma_0}_\HH$. Let $z=x+yJ$, then it is easy to check that $z^{-1}=\bar y^{-1}\bar x\bar b-bJ,$ where 
$b=\Imm(M.z)=(\bar x y^{-1}x+y)^{-1}$.

We need to check that $b\in A^{\bar\sigma}_+$. Since $y\in A^{\bar\sigma}_+$, we have $y^{-1}\in A^{\bar\sigma}_+$. Moreover, $\bar x y^{-1}x\in A^{\bar\sigma}_{\geq 0}$ for any $x\in A^{\sigma}$. Because $A^{\bar\sigma}_+$ is a convex cone, $\bar x y^{-1}x+y\in A^{\bar\sigma}_+$. Therefore, $b=(\bar x y^{-1}x+y)^{-1}\in A^{\bar\sigma}_+$

If $M:=\begin{pmatrix}
    a & 0 \\
    0 & \sigma(a)^{-1}
    \end{pmatrix}$
for $a\in A^\times$, then $M.z=az\sigma(a)\in A^{\sigma_0}_\HH$. Further,  $\Imm(M.z)=a\Imm(z)\sigma(a)\in A^{\bar\sigma}_+$ because $A^{\bar\sigma}_+$ is closed under action by congruence of $A^\times$.

Now, we check the transitivity. 
Let $z=x+yJ\in\mathcal S$ then $y=u^2$ for some $u\in (A^{\bar\sigma})^\times$. Then
$$\pi\left(
\begin{pmatrix}
1 & x \\
0 & 1
\end{pmatrix}
\begin{pmatrix}
u & 0 \\
0 & \sigma(u)^{-1}
\end{pmatrix}\right)=
\pi\left(
\begin{pmatrix}
u & x\sigma(u)^{-1} \\
0 & \sigma(u)^{-1}
\end{pmatrix}\right)=x+uJ\sigma(u)=$$
$$=x+u\bar\sigma(u)J=x+yJ=z$$

Finally, let us find the stabilizer of $1J$.
An element $M=\begin{pmatrix} a & b \\ c & d \end{pmatrix}$ stabilizes $1J$ if and only if
$$1J=M.1J=(aJ+b)(cJ+d)^{-1}=(aJ+b)(-\bar c+\bar dJ)^{-1}J.$$
So, $a=\bar d$ and $c=-\bar b$, i.e. $M\in\KSp_2^c(A,\sigma)$.
\end{proof}

\subsection{Projective model for \texorpdfstring{$\Sp_2(A,\sigma)$}{Sp2(A,sigma)}}

Now we define the projective model of the symmetric space of $\Sp_2(A,\sigma)$. For this, we consider the following quaternionic extension of $A$:
$$A_\HH:=\HH[A,A_\R,i,j]=A_\R\otimes_\R\HH\{i,j,k\}.$$
This space can be embedded into $\Mat_2(A)$ as a subalgebra in the following way:
$$\begin{matrix}
A_\HH & \hookrightarrow & \Mat_2(A) \\
a_1+a_2j & \mapsto  & \begin{pmatrix}
a_1 & a_2 \\
-\bar a_2 & \bar a_1
\end{pmatrix}.
\end{matrix}$$
The anti-involution $\bar\sigma^T$ on $\Mat_2(A)$ restricts to the following anti-involution on $A_\HH$:
$$\sigma_1(a_1+a_2j):=\bar\sigma(a_1)-\sigma(a_2)j,$$
where $a_1,a_2\in A$. Because $(A,\bar\sigma)$ is Hermitian, by the Proposition~\ref{Herm_2Matrix}, $(\Mat_2(A),\bar\sigma^T)$ is Hermitian and, therefore, $(A_\HH,\sigma_1)$ is Hermitian as well.

We denote:
$$A_\HH^{\sigma_1}:=\Fix_{A_\HH}(\sigma_1),\; (A_\HH^{\sigma_1})_+:=\theta_\HH(A_\HH^\times)$$
where
$$\begin{matrix}
\theta_\HH\colon & A_\HH & \to & A_\HH^{\sigma_1} \\
 & a & \mapsto & \sigma_1(a)a.
\end{matrix}$$

We also consider the following anti-involution on $A_\HH$:
$$\sigma_0(a_1+a_2j):=\sigma(a_1)+\bar\sigma(a_2)j,$$
where $a_1,a_2\in A$ and extend $\omega$ in the following way:
$$\omega_\HH(x,y):=\sigma_0(x)^T\Om y.$$
The following $\sigma_1$-sesquilinear form is an indefinite form on $A_\HH^2$:
$$h(x,y):=\sigma_1(x)^T
\begin{pmatrix}
0 & j \\
-j & 0
\end{pmatrix}y.$$
Indeed,
$$h(y,x)=\sigma_1(y)^T
\begin{pmatrix}
0 & j \\
-j & 0
\end{pmatrix}x=\sigma_1\left(\sigma_1(x)^T
\begin{pmatrix}
0 & j \\
-j & 0
\end{pmatrix}y\right)=\sigma_1(h(x,y)).$$
Then in the basis $e_1:=\left(\frac{1}{\sqrt{2}},\frac{j}{\sqrt{2}}\right)^T$, $e_2:=\left(\frac{1}{\sqrt{2}},-\frac{j}{\sqrt{2}}\right)^T$, the form $h$ is represented by the matrix $\begin{pmatrix} -1 & 0 \\ 0 & 1\end{pmatrix}$, i.e. $h$ is a $\sigma_1$-sesquilinear indefinite form on $A_\HH^2$.

\begin{prop}
$\Sp_2(A,\sigma)$ acts on $A_\HH^2$ preserving $h$. So we can see $\Sp_2(A,\sigma)$ as a subgroup of $\OO(h)$.
\end{prop}

\begin{proof}
Let $x,y\in A_\HH^2$, $M\in\Sp_2(A,\sigma)$, then
$$h(Mx,My)=\sigma_1(Mx)^T
\begin{pmatrix}
0 & j \\
-j & 0
\end{pmatrix}My
=\sigma_1(x)^T\bar\sigma(M)^Tj
\begin{pmatrix}
0 & 1 \\
-1 & 0
\end{pmatrix}My=$$
$$=\sigma_1(x)^Tj\sigma(M)^T
\begin{pmatrix}
0 & 1 \\
-1 & 0
\end{pmatrix}My
=\sigma_1(x)^T
\begin{pmatrix}
0 & j \\
-j & 0
\end{pmatrix}y=h(x,y).$$
So $M$ preserves $h$.
\end{proof}

Every quaternionic structure $J$ on $A^2$ can be extended additively to a quaternionic structure $J_\HH$ on $A^2_\HH$ in the following linear way:
$$J_\HH(x(a+bj)):=J(x)(\bar a+\bar bj).$$
where $x\in A^2$, $a,b\in A$.

\begin{prop}
For every quaternionic structure $J\in \mathfrak C$, there exist regular $x,y\in A^2_\HH$ such that $J_\HH(x)=xj$, $J_\HH(y)=-yj$. The elements $x,y$ are uniquely defined up to multiplication by elements of $A_\HH^\times$.
\end{prop}

\begin{proof}
Since $\Sp_2(A,\sigma)$ acts transitively on $\mathfrak C$, it is enough to prove the proposition for the standard quaternionic structure $J_0$.

Since
$$J_0(a_1+a_2j,b_1+b_2j)^T=(\bar b_1+\bar b_2j,-\bar a_1-\bar a_2j)^T,$$
we obtain
$$(\bar b_1+\bar b_2j,-\bar a_1-\bar a_2j)^T=(a_1+a_2j,b_1+b_2j)^Tj=(-a_2+a_1j,-b_2+b_1j)^T$$ if and only if $\bar a_1=b_2$, $\bar a_2=-b_1$, i.e.
$$x=(a_1+a_2j,-\bar a_2+\bar a_1j)^T=(a_1+a_2j,j(a_1+a_2j))^T=(1,j^T)a,$$
where $a=a_1+a_2j\in A_\HH$ arbitrary element. The element $x$ is regular if and only if $a\in A^\times_\HH$.
Analogously, $y=(j,1)^Ta$ where $a=a_1+a_2j\in A_\HH^\times$ arbitrary element.
\end{proof}

For a quaternionic structure $J\in\mathfrak C$, we denote by $l_J$ the $A_\HH$-line $yA_\HH$ such that $J_\CC(y)=-yj$.

We consider the spaces of isotropic elements and isotropic lines of $(A^2_\HH,\omega_\HH)$:
$$\Is(\omega_\HH):=\{x\mid x\in A^2_\HH\text{ regular, }\omega_\HH(x,x)=0\},$$
$$\PP(\Is(\omega_\HH)):=\{xA\mid x\in \Is(\omega_\HH)\}.$$
We also consider the symmetric space of $\OO(h)$:
$$\mathcal X_{\OO(h)}:=\{xA\mid h(x,x)\in (A^{\sigma_1}_\HH)_+\}.$$

\begin{df}
We call the space
$$\mathfrak P:=\mathcal X_{\OO(h)}\cap \PP(\Is(\omega_\HH))$$
the \defin{projective model}.
\end{df}

To justify this Definition, we prove the following Proposition:

\begin{prop}
The map
$$\begin{array}{cccl}
F\colon & \mathfrak C & \to & \mathfrak P\\
 & J & \mapsto & l_J
\end{array}$$
defines is a homeomorphism that is equivariant under the action of $\Sp_2(A,\sigma)$.

In particular, $\mathfrak{P}$ is a model of the symmetric space of $\Sp_2(A,\sigma)$. 
\end{prop}

\begin{proof}
1. Show that $l_J\in\mathcal X_{\OO(h)}$. Since $\Sp_2(A,\sigma)$ acts transitively on $\mathfrak C$, it is enough to check it the standard quaternionic structure $J_0$:
$$h((j,1)^T,(j,1)^T)=\sigma_1(j,1)
\begin{pmatrix}
0 & j \\
-j & 0
\end{pmatrix}\begin{pmatrix} j\\ 1\end{pmatrix}=(-j,1)\begin{pmatrix} j\\ 1\end{pmatrix}=2\in (A^{\sigma_1}_\HH)_+.$$

2. Show that $l_J\in\PP(\Is(\omega))$. It is enough to prove it for $J_0$:
$$\omega((j,1)^T,(j,1)^T)=\sigma_0(j,1)
\begin{pmatrix}
0 & 1 \\
-1 & 0
\end{pmatrix}\begin{pmatrix} j\\ 1\end{pmatrix}=(-1,j)\begin{pmatrix} j\\ 1\end{pmatrix}=0.$$

3. Show that $F$ is surjective. Let $v=u+wj\in A_\HH$ such that $vA_\HH\in \mathfrak P$. Since $h(v,v)\in (A_\HH^{\sigma_1})_+$, we can renormalize $v$ so that $h(v,v)=2$. Since $v\in\Is(\omega)$,
$$0=\omega_\HH(v,v)=\omega(u,u)+j\omega(w,w)j+\omega(u,w)j+j\omega(w,u)=$$
$$=\omega(u,u)-\overline{\omega(w,w)}+\left(\omega(u,w)+\overline{\omega(w,u)}\right)j$$
So we have:
$$\omega(u,u)=\overline{\omega(w,w)}$$
$$\omega(u,w)=-\overline{\omega(w,u)}.$$
Moreover,
$$2=h(v,v)=h(u+wj,u+wj)=h(u,u)-jh(w,w)j+h(u,w)j-jh(w,u).$$
Notice, for $u,w\in A^2$, $h(u,w)=\omega(\bar u,\bar w)j=j\omega(u,w)$. Therefore,
$$h(v,v)=\omega(\bar u,\bar u)j+\omega(w,w)j-\omega(\bar u,\bar w)+\omega(w, u).$$
$$=2\omega(w,u)+2\omega(w,w)j$$
So we have:
$$\omega(w,u)=1$$
$$\omega(u,u)=\omega(w,w)=0$$
It means that $(w,u)$ is a symplectic basis of $(A^2,\omega)$. We can define the following quaternionic structure: $J(u)=w$, $J(w)=-u$. By the Proposition~\ref{QuatStr-SympBas}, $J\in \mathfrak C$. Since
$$J_\HH(v)=J_\HH(u+wj)=w-uj=-(u+wj)j=-vj,$$
we obtain $F(J)=vA$, i.e. $F$ is surjective.

4. The map $F$ is injective because if $l_J=l_{J'}=yA$ for $J,J'\in\mathcal S'$ and some $y=y_1+y_2j\in A^2_\HH$. Then $J(y_1)=J'(y_1)=-y_2$, $J(y_2)=J'(y_2)=y_1$ and $(y_1,y_2)$ is a basis of $A^2$, i.e. $J=J'$.

5. Now, show the equivariance of $F$. Let $M\in\Sp_2(A,\sigma)$, $J\in\mathfrak C$ and $u,w\in A^2$ such that $w:=J(u)$, $J(w)=-u$. Then $MJM^{-1}(Mu)=Mw$, $MJM^{-1}(Mw)=-Mu$. That means that for $v=u+wj$,
$$F(MJM^{-1})=(Mv)A_\HH=M(vA_\HH)=MF(J),$$
i.e. $F$ is equivariant with respect to the $\Sp_2(A,\sigma)$-action.
\end{proof}

\subsection{Precompact model for \texorpdfstring{$\Sp_2(A,\sigma)$}{Sp2(A,sigma)}}

Now we define the precompact model of the symmetric space of $\Sp_2(A,\sigma)$. As we have seen in the Proposition~\ref{Precomp_Mod_O(h)}, the space $\mathcal X=\mathcal X_{\OO(h_{st})}$ for the standard $\sigma_1$-indefinite form on $A^2_\HH$ can be seen as a precompact domain
$$\mathring{D}(A_\HH,\sigma_1)=\{c\in A_\HH\mid 1-\sigma_1(c)c\in (A^{\sigma_1}_\HH)_+\}.$$
To see the symmetric space for $\Sp_2(A,\sigma)$ as a subset of this domain, we need an $A_\HH$-linear transformation that maps $h$ to the standard indefinite form. We can take the following matrix:
$$T:=\frac{1}{\sqrt{2}}\begin{pmatrix}
1 & j \\
j & 1
\end{pmatrix}.$$
Then $\sigma_1(T)^T[h]T=\diag(-1,1)=[h_{st}]$ and $T^{-1}\mathfrak P\subseteq \mathcal X$. Notice, $T\in\Sp_2(A_\HH,\sigma)$, therefore it stabilizes the set of isotropic elements of $(A_\HH^2,\omega$).

\begin{df}
The space
$$\mathfrak{B}:=\mathring{D}(A_\HH^{\sigma_0},\sigma_1):=\mathring{D}(A_\HH,\sigma_1)\cap A^{\sigma_0}_\HH=\{c\in A^{\sigma_0}_\HH\mid 1-\sigma_1(c)c\in (A^{\sigma_1}_\HH)_+\}$$
is called the \defin{precompact model}.
\end{df}

To justify this Definition, we prove the following Proposition.

\begin{prop}\label{Proj-Precomp-C}
The image of $T^{-1}\mathfrak P$ under the homeomorphism $\Phi\colon\mathcal X\to\mathring{D}(A_\HH,\sigma_1)$ is
$\mathring{D}(A_\HH^{\sigma_0},\sigma_1)$ which is precompact in $A_\HH^{\sigma_0}$.

In particular, $\mathring{D}(A_\HH^{\sigma_0},\sigma_1)$ is a model of the symmetric space of $\Sp_2(A,\sigma)$.
\end{prop}

\begin{proof}
To characterize the image of the symmetric space for $\Sp_2(A,\sigma)$ inside $\mathring{D}(A_\HH,\sigma_1)$, we remind that $(x_1,x_2)^T\in \Is(\omega)$ if and only if $\sigma_0(x_1)x_2\in A_\HH^{\sigma_0}$. Therefore, $(c,1)^T$ is isotropic if and only if $\sigma_0(c)\in A_\HH^{\sigma_0}$, i.e. $c\in A_\HH^{\sigma_0}$.
\begin{equation*}
\Phi(T^{-1}\mathcal S'')=\{c\in A^{\sigma_0}_\HH\mid 1-\sigma_1(c)c\in (A^{\sigma_1}_\HH)_+\}\subseteq A^{\sigma_0}_\HH.
\end{equation*}

The domain $\mathring{D}(A_\HH^{\sigma_0},\sigma_1)$ is precompact in $A_\HH^{\sigma_0}$ because it is a subset of the following domain:
$$D(A_\HH^{\sigma_0},\sigma_1)=\{c\in A_\HH^{\sigma_0}\mid 1-\sigma_1(c)c\in (A^{\sigma_1}_\HH)_{\geq 0}\}\subseteq A_\HH^{\sigma_0}.$$
that is compact by Proposition~\ref{comp_disc}.
\end{proof}

\begin{rem}
The group $T^{-1}\Sp_2(A,\sigma)T$ acts on $\mathring{D}(A_\HH^{\sigma_0},\sigma_1)$ by M\"obius transformations.
\end{rem}

\subsection{Connection between models}

Consider a Hermitian algebra $(A_\R,\sigma_\R)$ and its complexification $A=A_\R\otimes_\R\CC$. In this section, we consider  $\Sp_2(A,\sigma)$-equivariant homeomorphisms between the projective model, the upper half-space model and precompact model of the symmetric space for $\Sp_2(A,\sigma)$.

It is easy to check that the map:
$$\begin{matrix}
F \colon & \mathfrak P & \to & \mathfrak U\\
 & (x_1,x_2)A_\HH & \mapsto & x_1x_2^{-1}
\end{matrix}$$
is an $\Sp_2(A,\sigma)$-equivariant homeomorphism.

As we have seen in the Proposition~\ref{Proj-Precomp-C}, the map
$$\Phi\circ T^{-1}\colon\mathfrak P \to \mathring{D}(A^{\sigma_0}_\HH,\sigma_1) =\{c\in A^{\sigma_0}_\HH\mid 1-\sigma_1(c)c\in (A^{\sigma_1}_\HH)_+\}.$$
defines another $\Sp_2(A,\sigma)$-equivariant homeomorphism.

The maps $F$ and $\Phi\circ T^{-1}$ can be seen as different coordinate charts for the projective model $\mathfrak P$ of the symmetric space for $\Sp_2(A,\sigma)$.

\subsection{Compactification and Shilov boundary}

In this section, we construct a natural compactification of the symmetric space of $\Sp_2(A,\sigma)$.

Let $(A,\sigma)$ be the complexification of a Hermitian algebra as before. The space
$$\mathring{D}(A_\HH^{\sigma_0},\sigma_1)=\{c\in A^{\sigma_0}_\HH\mid 1-\sigma_1(c)c\in (A^{\sigma_1}_\HH)_+\}$$
is precompact. We take the topological closure of $\mathring{D}(A_\HH^{\sigma_0},\sigma_1)$ in  $A^{\sigma_0}_\HH$:
$$D(A_\HH^{\sigma_0},\sigma_1):=\{c\in A^{\sigma_0}_\HH\mid 1-\sigma_1(c)c\in (A^{\sigma_1}_\HH)_{\geq 0}\}.$$

\begin{df}
We call 
$$\check{S}(A_\HH^{\sigma_0},\sigma_1):=\{c\in A^{\sigma_0}_\HH\mid 1-\sigma_1(c)c=0\}$$
\defin{Shilov boundary} of the precompact model $\mathring{D}(A_\HH^{\sigma_0},\sigma_1)$.
\end{df}

Note, that $\check{S}(A_\HH^{\sigma_0},\sigma_1)$ is compact as a closed subspace of a compact space $D(A_\HH^{\sigma_0},\sigma_1)$.

\begin{rem}
The map $\Phi^{-1}$ extends to the boundary of $D(A_\HH^{\sigma_0},\sigma_1)$ and remains continuous and bijective. Since the boundary is compact, it is a homeomorphism. Therefore, we can see the boundary also in the projective model. In particular, we can see the Shilov boundary there.
\end{rem}

The next Proposition describes the Shilov boundary in the projective model.

\begin{prop}
The preimage of the Shilov boundary $\check{S}(A_\HH^{\sigma_0},\sigma_1)$ in $\Is(\omega_\HH)$ the map $\Phi\circ T^{-1}$ gives a compact subset of the boundary of the projective model. It consists of all lines of the form $xA_\HH$ such that $x\in\Is(\omega)$.
\end{prop}

\begin{proof}
Note that the line $l\in \Is(\omega_\HH)$ is of the form $xA_\HH$ for some $x\in\Is(\omega)$ if and only if $\eta(l) =l$ where $\eta\colon A_\HH\to A_\HH$ the following involution
$$\eta(c_1+c_2j):=c_1-c_2j$$
for $c_1,c_2\in A_{\CC\{i\}}$. Notice, $\eta$ is an involution on $A_\HH$ and
$$\sigma_1(\sigma_0(c_1+c_2j))=\bar c_1-\bar c_2j=-j\eta(c_1+c_2j)j.$$

Assume $c\in \check{S}(A_\HH^{\sigma_0},\sigma_1)$, i.e. $\sigma_1(c)^{-1}=c$, $\sigma_0(c)=c$. Then
$$(\Phi\circ T^{-1})\eta(T\circ\Phi^{-1}(c))=\Phi\left(\begin{pmatrix}0 & j\\ j & 0\end{pmatrix}\Clm{\eta(c)}{1}\right)=\Phi\left(\Clm{j}{j\eta(c)}\right)=$$
$$=-j\eta(c)^{-1}j=\sigma_1(\sigma_0(c))^{-1}=\sigma_1(c)^{-1}=c$$
i.e. for $l=(c,1)^TA_\HH$, $\eta(l)=l$.

If we take a line $xA_\HH$ for some $x=(x_1,x_2)^T\in\Is(\omega)$, then
$$c:=(\Phi\circ T^{-1})(xA)=(x_1-jx_2)(-jx_1+x_2)^{-1}.$$
Since $x\in\Is(\omega)\subset\Is(\omega_\HH)$, $c\in A_\HH^{\sigma_0}$. Further
$$\sigma_1(c)c=\sigma_1(\sigma_0(c))c=(\bar x_1+j\bar x_2)(j\bar x_1 +\bar x_2)^{-1}(x_1-jx_2)(-jx_1+x_2)^{-1}=$$
$$=(\bar x_1+j\bar x_2)(j\bar x_1 +\bar x_2)^{-1}(j\bar x_1+\bar x_2)(-j)(-jx_1+x_2)^{-1}=$$
$$=(\bar x_1+j\bar x_2)(-j)(-jx_1+x_2)^{-1}=(-jx_1+x_2)(-jx_1+x_2)^{-1}=1.$$
Therefore, $(\Phi\circ T^{-1})(xA)\in \check{S}(A_\HH^{\sigma_0},\sigma_1)$.
\end{proof}

\begin{cor}
The space $\PP(\Is(\omega))$ of isotropic lines of $(A^2,\omega)$  embedded into $\PP(\Is(\omega_\HH))$ as:
$$xA \mapsto xA_\HH$$
is a Shilov boundary in the projective model. This is a closed (even compact) orbit of the action of $\Sp_2(A,\sigma)$ on the boundary of the projective model.
\end{cor}

\section{Realizations of classical symmetric spaces}\label{A-class_ex}

In this section we apply the general construction of the different models of the symmetric space associated to $\Sp_2(A,\sigma)$ to give explicit models for the symmetric spaces of the classical Lie groups that can be realized as $\Sp_2(A,\sigma)$. This applies in particular to $\Sp(2n,\CC)$, $\GL(n,\CC)$ and $\OO(4n,\CC)$.

We construct explicit examples of models of symmetric space for classical Hermitian Lie groups of tube type. We will always denote by $A_\R$ a real Hermitian algebra, the complexified algebra will be denoted by $A:=A_\R\otimes_\R\CC$. The quaternionification of $A_\R$ will be denoted by $A_\HH$.

For the algebras $\Mat(n,\R)$ and $\Mat(n,\CC)$, we denote by $\sigma$ the transposition. For $\Mat(n,\CC)$, we denote by $\bar\sigma$ the composition of transposition and complex conjugation. For $\Mat(n,\HH\{i,j,k\})$, we denote by $\sigma_0$ the anti-involution acting in the following way:
$$\sigma_0(a+bj):=a^T+\bar b^Tj,$$
and by $\sigma_1$ the anti-involution acting in the following way:
$$\sigma_1(a+bj):=\bar a^T-b^Tj$$
for $a,b\in\Mat(n,\CC\{i\})$. In particular, we use the same notation in the case $n=1$, i.e. $\bar\sigma$ is the complex conjugation on $\CC$.

To denote different models of the symmetric space for a group $G$ that can be seen as $\Sp_2(A,\sigma)$ for some reel or complex $A$ and anti-involution $\sigma$, we use the following letters: $\mathfrak U(G)$ for the upper half-space model, $\mathfrak P(G)$ for the projective model, $\mathfrak B(G)$ for the precompact model and $\mathfrak C(G)$ for the complex/quaternionic structure model.

\subsection{Algebra \texorpdfstring{$(\Mat(n,\R),\sigma)$}{Mat(n,R)} and its complexification}

Let $A_\R:=\Mat(n,\R)$ be the real algebra with the anti-involution $\sigma$ given by transposition. Then
$$A=\Mat(n,\R)\otimes_\R\CC=\Mat(n,\CC),$$
$$\Sp_2(\Mat(n,\CC),\sigma)=\Sp(2n,\CC),\;\Sp_2(\Mat(n,\R),\sigma)=\Sp(2n,\R),$$
$$A_\HH=\Mat(n,\HH),\; A^{\sigma}=\Sym(n,\CC),\; A^{\bar\sigma}_+=\Herm^+(n,\CC),$$
$$\sigma_0(M_1+M_2j)=\sigma(M_1)+\bar\sigma(M_2)j=M_1^T+\bar M_2^Tj.$$
$$\sigma_1(M_1+M_2j)=\bar\sigma(M_1)-\sigma(M_2)j=\bar M_1^T-M_2^Tj.$$
where $M_1,M_2\in \Mat(n,\CC)$.
$$A_\HH^{\sigma_0}=\{M_1+M_2j\in\Mat(n,\HH)\mid M_1\in\Sym(n,\CC),\;M_2\in\Herm(n,\CC)\}.$$

\begin{ex} 
The {\em upper half-space model} of the symmetric space of
$\Sp(2n,\CC)$ is:
$$\mathfrak U(\Sp(2n,\CC))=\{M_1+M_2J\mid M_1\in\Sym(n,\CC),\;M_2\in\Herm^+(n)\}\subset$$
$$\subset\Mat(n,\HH).$$
The Siegel upper half space of $\Sp(2n,\R)$ is the real locus of this space:
$$\mathfrak U(\Sp(2n,\R))=\{M_1+M_2J\mid M_1\in\Sym(n,\R),\;M_2\in\Sym^+(n,\R)\}\subset$$
$$\subset\mathfrak U(\Sp(2n,\CC)).$$
\end{ex}

\begin{ex} 
The {\em precompact model} of the symmetric space of $\Sp(2n,\CC)$ is :
\begin{gather*}
\mathfrak B(\Sp(2n,\CC))=\\
=\{M_1+M_2j\in A_\HH^{\sigma_0}\mid \Id_n-(\bar M_1-\bar M_2j)(M_1+M_2j)\in \Herm^+(n,\HH)\}
\end{gather*}
The symmetric space for $\Sp(2n,\R)$ can be seen as the intersection of $\mathfrak B(\Sp(2n,\CC))$ with $\Mat(n,\CC\{j\})$:
\begin{gather*}
\mathfrak B(\Sp(2n,\R))=\\
=\{M_1+M_2j\in\Sym(n,\CC\{j\}) \mid \Id_n-(M_1-M_2j)(M_1+M_2j)\in \Herm^+(n,\CC\{j\})\}=\\
=\{M\in\Sym(n,\CC\{j\}) \mid \Id_n-\bar M M\in \Herm^+(n,\CC\{j\})\}\subset\mathfrak B(\Sp(2n,\CC)).
\end{gather*}
\end{ex}

\begin{ex}\label{ProjMod1} 
We consider $x,y\in A^2$
$$\omega(x,y)=\sigma_0(x)^T\Ome{\Id_n} y,$$
$$h(x,y)=\sigma_1(x)^T\Ome{\Id_n j} y.$$

Then the {\em projective model} of $\Sp(2n,\CC)$ is:
$$\mathfrak P(\Sp(2n,\CC))=\{xA_\HH\mid x\in A^2_\HH,\;\omega(x,x)=0,\;h(x,x)\in\Herm^+(n,\HH)\}.$$
The Shilov boundary corresponds in this model to the space:
$$\check{S}(\Sp(2n,\CC))\cong\{xA_\HH\mid x\in A^2_\HH,\;\omega(x,x)=h(x,x)=0\}\cong$$
$$\cong\{xA\mid x\in A^2,\;\omega(x,x)=0\}.$$

The projective model for $\Sp(2n,\R)$ can be seen as:
$$\mathfrak P(\Sp(2n,\R))=\{xA_{\CC\{j\}}\mid x\in A^2_{\CC\{j\}},\;\omega(x,x)=0,\;h(x,x)\in\Herm^+(n,\CC\{j\})\}$$
where $\CC\{j\}\subset\HH$ and $A_{\CC\{j\}}=A_\R\otimes_\R\CC\{j\}\cong A$. $\mathfrak P(\Sp(2n,\R))$ can be embedded into $\mathfrak P(\Sp(2n,\CC))$ using the following injective map: for $x\in\CC\{j\}$, $xA_{\CC\{j\}}\mapsto xA_{\HH}$.
The Shilov boundary corresponds in this model to the space:
$$\check{S}(\Sp(2n,\R))\cong\{xA_{\CC\{j\}}\mid x\in A^2_\R,\;\omega(x,x)=h(x,x)=0\}\cong$$
$$\cong\{xA_\R\mid x\in A_\R^2,\;\omega(x,x)=0\}.$$

We can also construct the projective model in terms of Lagrangians of $\HH^{2n}$. Consider $\HH^{2n}$ as a right module over $\HH$.  We can identify a line $xA_\HH$ for a regular $x\in A_\HH^2$ with a $n$-dimensional submodule of $\HH^{2n}$ in the following way:
$$L(xA):=\Span_{\HH}(xe_1,\dots,xe_n)\subset\HH^{2n}$$
where $e_i$ is the $i$-th basis vector (considered as a column) of the standard basis of $\HH^n$. In fact, the map $L$ is well-defined (does not depend on the choice of a regular $x\in xA$) and, moreover, it is a bijection.

We define two forms on $\HH^{2n}$: for $u,v\in\HH^{2n}$,
$$\tilde\omega(u,v):=\sigma_0(u)^T\Ome{\Id_n} v,$$
$$\tilde h(u,v):=\sigma_1(u)^T\Ome{j\Id_n} v.$$
If we take $x\in\Is(\omega)$, then $L(xA)\in \Lag(\HH^{2n},\tilde\omega)$. Using the map $L$, we obtain the following projective model for $\Sp(2n,\CC)\cong\Sp_2(A,\sigma)$:
$$\mathfrak P'(\Sp(2n,\CC))=\{l\in\Lag(\HH^{2n},\tilde\omega)\mid\forall v\in l\bs\{0\},\;\tilde h(v,v)>0\}.$$
The Shilov boundary corresponds in this model to the space:
$$\check{S}(\Sp(2n,\CC))\cong\{l\in\Lag(\HH^{2n},\tilde\omega)\mid\forall v\in l\bs\{0\},\;\tilde h(v,v)=0\}\cong\Lag(\CC^{2n},\tilde\omega).$$

The projective model for the symmetric space of $\Sp(2n,\R)\cong\Sp_2(A_\R,\sigma_\R)$ is:
$$\mathfrak P'(\Sp(2n,\R))=\{l\in\Lag(\CC\{j\}^{2n},\tilde\omega)\mid\forall v\in l\bs\{0\},\;\tilde h(v,v)>0\}.$$
It can be embedded to the projective model of  $\Sp(2n,\CC)$ by the map:
$$\begin{matrix}
\Lag(\CC\{j\}^{2n},\tilde\omega) & \to & \Lag(\HH^{2n},\tilde\omega)\\
l & \mapsto & \Span_\HH(l)
\end{matrix}.$$
The Shilov boundary corresponds in this model to the space:
$$\check{S}(\Sp(2n,\R))\cong\{l\in\Lag(\CC\{j\}^{2n},\tilde\omega)\mid\forall v\in l\bs\{0\},\;\tilde h(v,v)=0\}\cong\Lag(\R^{2n},\tilde\omega).$$
\end{ex}

\begin{ex}
Now we construct the quaternionic structures model and the complex structures model.
The quaternionic structure on $A$ can be seen as a $2n\times 2n$-matrix $J$ acting on $A^2$ as $J(x)=J\bar x$ for $x\in A^2$. Since $J(J(x))=J\overline{J\bar x}=-x$, $J\bar J=-\Id_n$.

The corresponding $\bar\sigma$-sesquilinear form is then
$$h_J(x,y)=\omega(J(x),y)=\bar x^T J^T\Ome{\Id_n} y.$$
So we obtain the quaternionic structure model for $\Sp(2n,\CC):$
$$\mathfrak C(\Sp(2n,\CC)):=\left\{J\in\Mat(2n,\CC)\mid J^T\Ome{\Id_n}\in\Herm^+(n,\CC),\,J\bar J=-\Id\right\}.$$
The space of complex structures on $A_\R^2$ can be seen as a subspace of $\mathfrak C(\Sp(2n,\CC))$ because every complex structure can be extended in the unique way to the quaternionic structure on $A^2$ in the following way:  for a complex structure $J$ we define 
$$J_\CC(x+yi):=J(x)-J(y)i$$
where $x,y\in A_\R$. So we obtain the inclusion of the complex structure model for $\Sp(2n,\R)$ into the quaternionic model for $\Sp(2n,\CC)$ as subspace of quaternionic structure fixing $A_\R^2\subset A^2$:
$$\mathfrak C(\Sp(2n,\R)):=\left\{J\in\Mat(2n,\R)\mid J^T\Ome{\Id_n}\in\Sym^+(n,\R),\,J^2=-\Id_{2n}\right\}=$$
$$=\{J\in \mathfrak C(\Sp(2n,\CC))\mid J\in\Mat(2n,\R)\}$$
\end{ex}

\subsection{Algebra \texorpdfstring{$(\Mat(n,\CC),\bar\sigma)$}{Mat(n,C)} and its complexification}

In this section, we consider the algebra $A_\R:=\Mat(n,\CC\{I\})$ with the anti-involution $\bar\sigma$ given by transposition and complex conjugation. Then
$$A=\Mat(n,\CC\{I\})\otimes_\R\CC\{i\},$$
$$\Sp_2(A_\R,\bar\sigma)=\UU(n,n),$$
$$\Sp_2(A,\bar\sigma\otimes\Id)=\GL(2n,\CC).$$
$$A_\HH=\Mat(n,\CC\{I\})\otimes_\R\HH\{i,j,k\}.$$

\begin{ex}\label{Upperhalf2} First, we construct the upper half-space model. In the Section~\ref{Isom_chi}, we studied the following $\CC\{I\}$-algebras isomorphism:
$$\begin{matrix}
\chi\colon & \Mat(n,\CC\{i\})\otimes_\R\CC\{I\} & \to & \Mat(n,\CC\{I\})\times\Mat(n,\CC\{I\})\\
& a+bi & \mapsto & (a+bI , a-bI)
\end{matrix}$$
where $a,b\in\Mat(n,\CC\{I\})$.

We have seen,
$$\chi(A_\R)=\chi(\Mat(n,\CC\{i\}))=\{(m,\bar m)\mid m\in\Mat(n,\CC\{I\})\},$$
$$\chi(A^{\bar\sigma\otimes\Id})=\{(m,m^T)\mid m\in\Mat(n,\CC\{i\})\}\cong\Mat(n,\CC).$$
$$\chi(A^{\bar\sigma\otimes\bar\sigma})=\Herm(n,\CC\{I\})\times\Herm(n,\CC\{I\}),$$ $$\chi(A^{\bar\sigma\otimes\bar\sigma}_+)=\Herm^+(n,\CC\{I\})\times\Herm^+(n,\CC\{I\}).$$
So we have the following model for the symmetric space for $\GL(2n,\CC)$:
$$\mathfrak U(\GL(2n,\CC))=
\left\{\begin{pmatrix}
m_{11}\\
m^T_{11}
\end{pmatrix}
+\begin{pmatrix}
m_{12}\\
m_{22}
\end{pmatrix}J
\left|\;\begin{matrix}
m_{11}\in\Mat(n,\CC\{I\}),\\
m_{12},m_{22}\in\Herm^+(n,\CC\{I\})
\end{matrix}\right.\right\}\subset$$
$$\subset\HH[\Mat(n,\CC\{I\})\times\Mat(n,\CC\{I\}),\chi(\Mat(n,\CC\{i\})),(I,I),J].$$
Since $A_\R^{\sigma}=A_\R\cap A^{\bar\sigma\otimes\Id}=A_\R\cap A^{\bar\sigma\otimes\bar\sigma}=\Herm(n)$, we obtain the symmetric space for $\UU(n,n)$ is:
$$\mathfrak U(\UU(n,n))\cong
\left\{\begin{pmatrix}
m_1\\
\bar m_1
\end{pmatrix}
+\begin{pmatrix}
m_2\\
\bar m_2
\end{pmatrix}j
\midwd\begin{matrix}
m_1\in\Herm(n,\CC\{I\}),\\
m_2\in\Herm^+(n,\CC\{I\})
\end{matrix}\right\}\subset\mathfrak U(\GL(2n,\CC)).$$
To see $\mathfrak U(\UU(n,n))$ as a subset of $\Mat(n,\CC\{I\})\times\Mat(n,\CC\{I\})$, we have to identify $J$ and $(I,I)=\chi(1\otimes I)$, so we get
$$\mathfrak U(\UU(n,n))=$$
$$=\left\{(m_1+m_2I,\bar m_1+\bar m_2I)
\mid m_1\in\Herm(n,\CC\{I\}), m_2\in\Herm^+(n,\CC\{I\})\right\}\subset$$
$$\subset\Mat(n,\CC\{I\})\times\Mat(n,\CC\{I\}).$$
In a pair $(m_1+m_2I,\bar m_1+\bar m_2I)$ for $m_1\in\Herm(n,\CC\{I\}), m_2\in\Herm^+(n,\CC\{I\})$, the second component is completely determined by the first one. It is easy to see, because $m_2I$ is skew-Hermitian and $m_1+m_2I$ corresponds to the decomposition of an element from $\Mat(n,\CC\{I\})$ in Hermitian and skew-Hermitian part. Therefore, $m_1$ and $m_2$ are well-defined by $m_1+m_2I$.
Therefore, we can identify
$$\mathfrak U(\UU(n,n))\cong\{m_1+m_2I\mid m_1\in\Herm(n,\CC\{I\}), m_2\in\Herm^+(n,\CC\{I\})\}.$$
\end{ex}

\begin{ex} Now we construct the precompact model.
We use the map $\psi$ from the Section~\ref{Isom_psi} to identify $A_\HH$ with $\Mat(2n,\CC)$.
$$\begin{matrix}
\psi\colon & \Mat(n,\HH\{i,j,k\})\otimes_\R\CC\{I\} & \to & \Mat(2n,\CC\{i\})\\
 & (q_1+q_2j) + (p_1+p_2j)I & \mapsto &
\begin{pmatrix}
q_1+p_1i & q_2+p_2i\\ -\bar q_2-\bar p_2i & \bar q_1+\bar p_1i
\end{pmatrix}.
\end{matrix}$$
where $q_1,q_2,p_1,p_2\in\Mat(n,\CC\{i\})$.

The anti-involution $\bar\sigma\otimes\sigma_0$ on $\Mat(n,\CC\{I\})\otimes_\R\HH\{i,j,k\}$ induces the following anti-involution
$$\psi\circ(\bar\sigma\otimes\sigma_0)\circ\psi^{-1}$$
on $\Mat(2n,\CC)$: $m\mapsto \begin{pmatrix}
0 & \Id\\
\Id & 0
\end{pmatrix}\bar m^T\begin{pmatrix}
0 & \Id\\
\Id & 0
\end{pmatrix}$. Therefore,
$$\psi(A_\HH^{\bar\sigma\otimes\sigma_0})=\left\{m\in\Mat(2n,\CC\{i\})\midwd m=\begin{pmatrix}
0 & \Id\\
\Id & 0
\end{pmatrix}\bar m^T\begin{pmatrix}
0 & \Id\\
\Id & 0
\end{pmatrix}\right\}.$$
Similarly, the anti-involution $\bar\sigma\otimes\sigma_1$ on $\Mat(n,\CC\{I\})\otimes_\R\HH\{i,j,k\}$ induces the following anti-involution
$$\psi\circ(\bar\sigma\otimes\sigma_1)\circ\psi^{-1}$$
on $\Mat(2n,\CC)$: $M\mapsto \bar M^T$
and so $\psi(A_\HH^{\bar\sigma\otimes\sigma_1})=\Herm(2n,\CC)$. So we obtain the following precompact model for the symmetric space of $\GL(2n,\CC)$:
$$\mathfrak B(\GL(2n,\CC))=\{M\in\psi(A_\HH^{\bar\sigma\otimes\sigma_0})\mid \Id_{2n}-\bar M^T M\in \Herm^+(2n,\CC)\}.$$
To see the precompact for $\UU(n,n)$ as a subspace of $\mathfrak B(\GL(2n,\CC))$, we have to intersect of with $\psi(\Mat(n,\CC\{I\})\otimes_\R\CC\{j\})$. We remind from the Section~\ref{Embedd2}
$$\psi(\Mat(n,\CC\{I\})\otimes_\R\CC\{j\})=$$
$$=\left\{m\in\Mat(2n,\CC\{i\})\midwd m=-\Ome{\Id}m\Ome{\Id}\right\}.$$
Since
$$\psi(\Mat(n,\CC\{I\})\otimes_\R\CC\{j\})\cap\psi(A_\HH^{\bar\sigma\otimes\sigma_0})=\left\{\mtrx{a}{b}{-b}{a}\midwd a,b\in\Herm(n,\CC)\right\},$$
we obtain:
$$\mathfrak B(\UU(n,n))\cong\left\{\mtrx{a}{b}{-b}{a}\midwd \mtrx{\Id_n-a^2-b^2}{ba-ab}{ab-ba}{\Id_n-a^2-b^2}\in\Herm^+(2n,\CC)\right\}\subset$$
$$\subset\mathfrak B(\GL(2n,\CC)).$$
Under the map $\chi$ from the Section~\ref{Isom_chi}, $A$ can be identified with $\Mat(n,\CC)\times \Mat(n,\CC),$ so we obtain the following precompact model for $\UU(n,n)$:
$$\mathfrak B(\UU(n,n))=$$
$$=\{(M,M^T)\mid M\in \Mat(n,\CC),\;\Id_n-\bar M^TM\in \Herm^+(n,\CC),\Id_n-\bar MM^T\in \Herm^+(n,\CC)\}.$$
The second component if the pair $(M,M^T)$ is determined by the first one. Moreover, if $\Id_n-\bar M^TM\in \Herm^+(n,\CC)$ then $\Id_n-\bar MM^T\in \Herm^+(n,\CC)$. Therefore, we can identify:
$$\mathfrak B(\UU(n,n))=\{M\in \Mat(n,\CC)\mid\Id_n-\bar M^TM\in \Herm^+(n,\CC)\}.$$
\end{ex}

\begin{rem}
The description for the precompact model of the symmetric space of $\UU(n,n)$ seen as $\Sp_2(\Mat(n,\CC),\bar\sigma)$ agree with the description for the projective model of the symmetric space of $\UU(n,n)$ seen as $\OO(h_{st})$ for $h_{st}$ the standard indefinite form (see Section~\ref{SymSp_O(h)}).
\end{rem}

\begin{ex} We now construct the projective model.  Under $\psi$, the anti-involution $\bar\sigma\otimes\sigma_0$ on $\Mat(n,\CC\{I\})\otimes_\R\HH\{i,j,k\}$ induces the following anti-involution
$$\sigma':=\psi\circ(\bar\sigma\otimes\sigma_0)\circ\psi^{-1}$$
on $\Mat(2n,\CC)$: $M\mapsto \mtrx{0}{\Id_n}{\Id_n}{0}\bar M^T\mtrx{0}{\Id_n}{\Id_n}{0}$. Therefore,
$$\psi(A_\HH^{\bar\sigma\otimes\sigma_0})=\left\{M\in \Mat(2n,\CC)\midwd \mtrx{0}{\Id_n}{\Id_n}{0}\bar M^T\mtrx{0}{\Id_n}{\Id_n}{0}\right\}.$$
Similarly, the anti-involution $\bar\sigma\otimes\sigma_1$ on $\Mat(n,\CC\{I\})\otimes_\R\HH\{i,j,k\}$ induces the following anti-involution
$$\sigma'':=\psi\circ(\bar\sigma\otimes\sigma_1)\circ\psi^{-1}$$
on $\Mat(2n,\CC)$: $M\mapsto \bar M^T$
and so $\psi(A_\HH^{\bar\sigma\otimes\sigma_1})=\Herm(2n,\CC)$.

Further, for $x,y\in (A')^2$
$$\omega(x,y)=\sigma'(x)^T\Ome{\Id_{2n}} y=$$
$$=\mtrx{0}{\Id_n}{\Id_n}{0}\bar x^T\Di{\begin{matrix}0 & \Id_n\\ \Id_n & 0 \end{matrix}}
\begin{pmatrix}
0 & 0  & \Id_n & 0\\
0 & 0  & 0 & \Id_n\\
-\Id_n & 0  & 0 & 0\\
0 & -\Id_n  & 0 & 0
\end{pmatrix} y=$$
$$=\mtrx{0}{\Id_n}{\Id_n}{0}\bar x^T \begin{pmatrix}
0 & 0  & 0 & \Id_n\\
0 & 0  & \Id_n & 0\\
0 & -\Id_n  & 0 & 0\\
-\Id_n & 0  & 0 & 0
\end{pmatrix}y$$
$$h(x,y)=\sigma''(x)^T\Ome{\Id_{2n}i} y=\bar x^T\Ome{\Id_{2n}i}y$$
Note, $x\in\Is(\omega)$ if and only if $x\in\Is(\omega')$ where
$$\omega'(x,y)=\bar x^T \begin{pmatrix}
0 & 0  & 0 & \Id_n\\
0 & 0  & \Id_n & 0\\
0 & -\Id_n  & 0 & 0\\
-\Id_n & 0  & 0 & 0
\end{pmatrix}y$$

We obtain the projective model for $\GL(2n,\CC)$:
$$\mathfrak P(\GL(2n,\CC))=\{xA'\mid x\in (A')^2,\;\omega'(x,x)=0,\;h(x,x)\in\Herm^+(2n,\CC)\}.$$
The Shilov boundary corresponds in this model to the space:
$$\check S(\GL(2n,\CC))\cong\{xA'\mid x\in (A')^2,\;\omega'(x,x)=h(x,x)=0\}.$$

The projective model for $\UU(n,n)$ can be seen as a subspace of $\mathfrak P(\GL(2n,\CC))$ in the following way. As we have seen in the Section~\ref{Embedd1},
$$\psi(A_\R\otimes_\R\CC\{j\})=\left\{\begin{pmatrix}
q & p\\
-p & q
\end{pmatrix}\midwd p,q\in \Mat(n,\CC\{i\})\right\}=$$
$$=\left\{m\in\Mat(2n,\CC\{i\})\midwd m=-\Ome{\Id_n}m\Ome{\Id_n}\right\}.$$
Therefore, if we define
$$\delta(x):=-\Di{\COme{\Id_n}}x\Ome{\Id_n}$$
for $x\in(A')^2_\HH.$ We obtain
$$\mathfrak P(\UU(n,n))\cong\{xA'\in\mathfrak P(\GL(2n,\CC))\mid x\in (A')^2_\HH,\;\delta(x)=x\}$$
We can also see the projective model for $\UU(n,n)$ in another way. We consider the isomorphism $\chi$ from the Section~\ref{Isom_chi} identifying $\Mat(n,\CC\{I\})\otimes_\R\CC\{i\}$ with $\Mat(n,\CC\{i\})\times\Mat(n,\CC\{i\})=:A''$. Then the induced by $\bar\sigma\otimes\Id$ anti-involution $$\chi\circ(\bar\sigma\otimes\Id)\circ\chi^{-1}$$
on $\Mat(n,\CC\{i\})\times\Mat(n,\CC\{i\})$ acts in the following way:
$$(m_1,m_2)\mapsto(m^T_2,m^T_1).$$
The induced by $\bar\sigma\otimes\bar\sigma$ involution $$\chi\circ(\bar\sigma\otimes\bar\sigma)\circ\chi^{-1}$$
on $\Mat(n,\CC\{i\})\times\Mat(n,\CC\{i\})$ acts in the following way:
$$(m_1,m_2)\mapsto(\bar m_1^T,\bar m_2^T).$$
Note,
$$(A'')^2= \Mat(n,\CC\{i\})^2\times\Mat(n,\CC\{i\})^2.$$
We take $x_1,x_2,y_1,y_2\in \Mat(n,\CC\{i\})^2$, then we can define
$$\omega((x_1,x_2),(y_1,y_2)):=$$
$$=\chi\circ(\bar\sigma\otimes\Id)\circ\chi^{-1}(x_1,x_2) \Ome{(\Id_n,\Id_n)}(y_1,y_2)=$$
$$=\left(x_2^T\Ome{\Id_n} y_1,x_1^T\Ome{\Id_n} y_2\right),$$
$$h((x_1,x_2),(y_1,y_2)):=\left(\bar x_1^T\Ome{\Id_ni}y_1,\bar x_2^T\Ome{\Id_ni}y_2\right).$$

We obtain the projective model for $\UU(n,n)$:
$$\mathfrak P(\UU(n,n))=\left\{(x_1,x_2)A''\left|
\begin{array}{l}
x_1,x_2\in\Mat(n,\CC\{i\})^2,\;\hat\omega(x_1,x_2)=0,\\
\hat h(x_1,x_1),\hat h(x_2,x_2)\in\Herm^+(n,\CC)
\end{array}\right.
\right\}$$
where $\hat\omega(x_1,x_2):=x_1^T\Ome{\Id_n} y_2$, $\hat h(x,y):=\bar x_1^T\Ome{\Id_ni}y_1$.
Since $\hat\omega$ is non-degenerate, the line $x_2\Mat(n,\CC\{i\}$ is uniquely defined by $x_1$.

Let us check that for the pair $(x_1,x_2)$ such that $\hat\omega(x_1,x_2)=0$, $\hat h(x_1,x_1)\in\Herm^+(n,\CC)$, we have always $\hat h(x_2,x_2)\in\Herm^+(n,\CC)$. As we have seen in the Section~\ref{Connection_btw_models}, we can always choose $x_1=(m_1,1)^T$, $x_2=(m_2,1)^T$. Then
$$\hat\omega(x_1,x_2)=m_1^T-m_2=0,$$
$$\hat h(x_1,x_1)=i(\bar m_1^T-m_1)\in\Herm^+(n,\CC).$$
These two conditions imply
$$\hat h(x_2,x_2)=i(\bar m_2^T-m_2)=i(\bar m_1-m_1^T)=i(\bar m_1^T-m_1)^T\in\Herm^+(n,\CC).$$
Therefore, we can write the following identification:
$$\mathfrak P(\UU(n,n))\cong\left\{x\Mat(n,\CC\{i\})\mid \hat h(x,x)\in\Herm^+(n,\CC)
\right\}.$$
The Shilov boundary corresponds in this model to the space:
$$\check S(\UU(n,n))\cong\left\{x\Mat(n,\CC\{i\})\mid \hat h(x,x)=0
\right\}.$$

To construct the projective model in terms of Lagrangians, similarly to the Example~\ref{ProjMod1}, we can identify the space of $A'$-lines of $(A')^2$ with the space $\Gr(2n,\CC^{4n})$ of $2n$-dimensional subspaces of $\CC^{4n}$ by the rule:
$$L(xA'):=\Span_{\HH}(xe_1,\dots,xe_{2n})$$
where $e_i$ is the $i$-th basis vector (considered as a column) of the standard basis of $\CC^{2n}$.

We define two forms on $\CC^{4n}$: for $u,v\in\CC^{4n}$,
$$\tilde\omega(u,v):=\bar u^T
\begin{pmatrix}
0 & 0  & 0 & \Id_n\\
0 & 0  & \Id_n & 0\\
0 & -\Id_n  & 0 & 0\\
-\Id_n & 0  & 0 & 0
\end{pmatrix}v,$$
$$\tilde h(u,v):=\bar u^T\Ome{i\Id_{2n}} v.$$
The projective model for the symmetric space of $\GL(4n,\CC)\cong\Sp_2(A,\sigma)$ can be seen as the following space:
$$\mathfrak P'(\GL(4n,\CC))=\{l\in\Lag(\CC^{4n},\tilde\omega)\mid\forall v\in l\bs\{0\},\;\tilde h(v,v)>0\}.$$
where $\Lag(\CC^{4n},\tilde\omega)$ is the space of all maximal isotropic subspaces of $\CC^{4n}$ with respect to $\tilde\omega$.
The Shilov boundary corresponds in this model to the space:
$$\check S(\GL(4n,\CC))=\{l\in\Lag(\CC^{4n},\tilde\omega)\mid\forall v\in l\bs\{0\},\;\tilde h(v,v)=0\}.$$

We can see the the projective model for the symmetric space of $\UU(n,n)\cong \Sp_2(A_\R,\sigma_\R)$ as a subspace of $\mathfrak P(\GL(4n,\CC))$:
$$\mathfrak P'(\UU(n,n))\cong\{l\in\mathfrak P(\GL(4n,\CC))\mid\delta'(l)=l\}$$
where
$$\begin{matrix}
\delta'\colon & \CC^{4n} & \to & \CC^{4n}\\
& v & \mapsto & \mtrx{\COme{\Id_n}}{0}{0}{\COme{\Id_n}}v.
\end{matrix}$$

We can also see another projective model for the symmetric space of $\UU(n,n)\cong \Sp_2(A_\R,\bar\sigma)$ if we identify again $A=\Mat(n,\CC)\otimes_\R\CC$ with $\Mat(n,\CC)\times\Mat(n,\CC)=:A'$ by the map $\chi$ form the Section~\ref{Isom_chi}.
$$\begin{matrix}
\chi\colon & \Mat(n,\CC\{I\})\otimes_\R\CC\{i\} & \to & \Mat(n,\CC\{i\})\times\Mat(n,\CC\{i\})\\ & a+bI &
\mapsto & (a+bi , a-bi)
\end{matrix}$$
As before, we can identify every line $xA'\subset (A')^2$ with pair of $n$-dimensional subspaces of $\CC^{2n}$. We define two forms on $\CC^{2n}$: for $u,v\in \CC^{2n}$
$$\tilde\omega(u,v):=u^T\Ome{\Id_n}v,$$
$$\tilde h(u,v):=\bar u^T\Ome{\Id_ni}v.$$
The pair $(l_1,l_2)$ of $n$-dimensional subspaces of $\CC^{2n}$ is called $\omega$-orthogonal if for all $v\in l_1$, $u\in l_2$, $\tilde\omega(v,u)=0$. So we can see the projective model of the symmetric space for $\UU(n,n)$:
$$\mathfrak P'(\UU(n,n))=\{(l_1,l_2)\text{ $\tilde\omega$-orthogonal pair}\mid\forall u\in l_1\cup l_2\bs\{0\},\;\tilde h(u,u)>0\}.$$
Since $\omega$ is non-degenerate, the space $l_2$ is completely determined by $l_1$. And as we have seen for $\mathfrak P(\UU(n,n))$, if for all $u\in l_1\bs\{0\}$, $\tilde h(u,u)>0$, then for all $u\in l_2\bs\{0\}$, $\tilde h(u,u)>0$. Therefore, we can identify
$$\mathfrak P'(\UU(n,n))\cong\{l\in\Gr(n,\CC^{2n})\mid\forall u\in l\bs\{0\},\;\tilde h(u,u)>0\}.$$
The Shilov boundary corresponds in this model to the space:
$$\check S(\UU(n,n))\cong\{l\in\Gr(n,\CC^{2n})\mid\forall u\in l\bs\{0\},\;\tilde h(u,u)=0\}.$$

\begin{rem}
The description for the projective model of the symmetric space of $\UU(n,n)$ seen as $\Sp_2(\Mat(n,\CC),\bar\sigma)$ agree with the description for the projective model of the symmetric space of $\UU(n,n)$ seen as $\OO(h_{st})$ for $h_{st}$ the standard indefinite form (see Section~\ref{SymSp_O(h)}).
\end{rem}
\end{ex}

\begin{ex} Now we define the quaternionic structures model and the complex structures model.
We use the map $\chi$ from the Section~\ref{Isom_chi},
$$\begin{matrix}
\chi\colon & \Mat(n,\CC\{I\})\otimes_\R\CC\{i\} & \to & \Mat(n,\CC\{i\})\times\Mat(n,\CC\{i\})\\ & a+bI &
\mapsto & (a+bi , a-bi)
\end{matrix}$$
to identify $A$ with $A':=\Mat(n,\CC\{i\})\times\Mat(n,\CC\{i\})$. The involution $\Id\otimes\bar\sigma$ is mapped under $\chi$ to the involution
$$(m_1,m_2)\mapsto(\bar m_2,\bar m_1).$$
on $\Mat(n,\CC\{i\})\times\Mat(n,\CC\{i\})$.

If we take a quaternionic structure $J$ on $A^2$ then we define
$$J':=\chi\circ J\circ\chi^{-1}.$$
If we see $J'$ as a pair $(J_1,J_2)$ of $2n\times 2n$ complex matrices then $J_1 \bar J_2=-\Id_{2n}$ because for $(m_1,m_2)\in (A')^2\cong \Mat(n,\CC\{i\})^2\times\Mat(n,\CC\{i\})^2$,
$$J'(m_1,m_2)=(J_1,J_2)(\bar m_2,\bar m_1)=(J_1\bar m_2,J_2\bar m_1),$$
$$-(m_1,m_2)=(J')^2(m_1,m_2)=(J_1\overline{J_2\bar m_1},J_2\overline{J_1\bar m_2}).$$

The induced by $\bar\sigma\otimes\Id$ anti-involution $$\chi\circ(\bar\sigma\otimes\Id)\circ\chi^{-1}$$
on $\Mat(n,\CC\{i\})\times\Mat(n,\CC\{i\})$ acts in the following way:
$$(m_1,m_2)\mapsto(m^T_2,m^T_1).$$
We take the standard symplectic structure on $(A')^2$: for $x_1,x_2,y_1,y_2\in\Mat(n,\CC\{i\})$
$$\omega((x_1,x_2),(y_1,y_2))=\chi\circ(\bar\sigma\otimes\Id)\circ\chi^{-1}(x_1,x_2) \Ome{(\Id_n,\Id_n)}(y_1,y_2)=$$
$$=\left(x_2^T\Ome{\Id_n} y_1,x_2^T\Ome{\Id_n} y_1\right).$$
For a quaternionic structure on $(A')^2$ seen as pair of matrices $(J_1,J_2)$, we define
$$h_{(J_1,J_2)}((x_1,x_2),(y_1,y_2)):=$$
$$=\left((J_1\bar x_1)^T\Ome{\Id_n}y_1, (J_2\bar x_2)^T\Ome{\Id_n}y_2\right)=$$
$$=\left(\bar x_1^TJ_1^T\Ome{\Id_n}y_1, \bar x_2^TJ_2^T\Ome{\Id_n}y_2\right).$$
The quaternionic structure model for $\GL(2n,\CC)$ is then:
$$\mathfrak C(\GL(2n,\CC)):=\left\{(J_1,J_2)\left|
\begin{array}{l}
J_1,J_2\in \Mat(2n,\CC),\,J_1\bar J_2=-\Id_{2n},\\
J_1^T\Ome{\Id_n},\,J_2^T\Ome{\Id_n}\in\Herm^+(2n,\CC)
\end{array}\right.\right\}.$$
Since $J_1\bar J_2=-\Id_{2n}$, by given $J_1$ such that $J_1^T\Ome{\Id_n}\in\Herm^+(2n,\CC)$, we can calculate $J_2=-\bar J_1^{-1}$. Then
$$J_2^T\Ome{\Id_n}=-\bar J_1^{-T}\Ome{\Id_n}=\left(\Ome{\Id_n}\bar J_1^{T}\right)^{-1}\in \Herm^+(2n,\CC)$$
if and only if
$$\Ome{\Id_n}\bar J_1^{T}\in\Herm^+(2n,\CC)$$
if and only if
$$\bar J_1^{T}\Ome{\Id_n}\in\Herm^+(2n,\CC).$$
Therefore, we can identify
$$\mathfrak C(\GL(2n,\CC))\cong\left\{J\in\Mat(2n,\CC)\midwd
J^T\Ome{\Id_n}\in\Herm^+(2n,\CC)
\right\}.$$
In this presentation of the symmetric space, $\GL(2n,\CC)$ acts on it in the following way: $$(g,J)\mapsto -g^{-1}J\Ome{\Id_n}\bar g^{-T}\Ome{\Id_n}$$
for $g\in\GL(2n,\CC)$.

Since
$$\chi(A_\R)=\{(m,\bar m)\mid m\in \Mat(n,\CC\{i\})\},$$
the quaternionic structure model for $\UU(n,n)\cong\Sp_2(A_\R,\bar\sigma)$ can be seen as a subset of $\mathfrak C(\GL(2n,\CC))$ stabilizing $\chi(A_\R)$. $(J_1,J_2)\in \mathfrak C(\GL(2n,\CC))$ stabilizes $\chi(A_\R)^2$ if and only if for all $m\in\Mat(n,\CC\{i\})^2$,
$$(J_1,J_2)(m,\bar m)=(J_1(m),J_2(\bar m))=(m',\bar m'),$$
for some $m'\in\Mat(n,\CC)^2$, i.e. $J_1=J_2$. Therefore,
$$\mathfrak C(\UU(n,n))\cong\{(J,J)\in\mathfrak C(\GL(2n,\CC))\}.$$
We can also see $\mathfrak C(\UU(n,n))$ directly as the space complex structures on $A_\R^2$:
$$\mathfrak C(\UU(n,n))=\left\{J\in\Mat(2n,\CC)\midwd \bar J^T\Ome{\Id_n}\in\Herm^+(2n,\CC), J\bar J=-\Id_{2n}\right\}.$$
\end{ex}

\subsection{Algebra \texorpdfstring{$\Mat(n,\HH)$}{Mat(n,H)} and its complexification}

In this section, we consider the algebra $A_\R:=\Mat(n,\HH\{I,J,K\})$ with the anti-involution $\sigma_1$ given by transposition and quaternionic conjugation. Then
$$A=\Mat(n,\HH\{I,J,K\})\otimes_\R\CC\{i\},$$
$$\Sp_2(A_\R,\sigma_1)=\SO^*(4n),$$
$$\Sp_2(A,\sigma_1\otimes\Id)=\OO(4n,\CC),$$
$$A_\HH=\Mat(n,\HH\{I,J,K\})\otimes_\R\HH\{i,j,k\}.$$

\begin{ex}\label{Upperhalf3} First, we construct the upper half-space model. In the Section~\ref{Isom_psi}, we studied the following $\CC\{i\}$-$\CC\{I\}$-algebras isomorphism:
$$\begin{matrix}
\psi\colon & \Mat(n,\HH\{I,J,K\})\otimes_\R\CC\{i\} & \to & \Mat(2n,\CC\{I\})\\
 & (q_1+q_2J) + (p_1+p_2J)i & \mapsto &
\begin{pmatrix}
q_1+p_1I & q_2+p_2I\\
-\bar q_2-\bar p_2I & \bar q_1+\bar p_1I
\end{pmatrix}.
\end{matrix}$$
where $q_1,q_2,p_1,p_2\in\Mat(n,\CC\{I\})$.

We remind,
$\psi(\Id_n\otimes i)=\Id_{2n}I$ and
$$\psi(A_\R)=\psi(\Mat(n,\HH\{I,J,K\}))=\left\{
\begin{pmatrix}
q_1 & q_2\\
-\bar q_2 & \bar q_1
\end{pmatrix}\midwd q_1,q_2\in\Mat(n,\CC)
\right\}.$$

Under $\psi$, the anti-involution $\sigma_1\otimes\Id$ on $\Mat(n,\HH\{I,J,K\})\otimes_\R\CC\{i\}$ indices the following anti-involution
$$\sigma':=\psi\circ(\sigma_1\otimes\Id)\circ\psi^{-1}$$
on $\Mat(2n,\CC\{I\})=\Mat_2(\Mat(n,\CC\{I\}))$:
$$\sigma'(m)=-\Ome{\Id_n} m^T\Ome{\Id_n}$$
for $m\in\Mat(2n,\CC\{I\})$. So we have:
$$\psi(A^{\sigma_1\otimes\Id})=\left\{m\in\Mat(2n,\CC\{I\})\midwd m=-\Ome{\Id_n} m^T\Ome{\Id_n}\right\}=$$
$$=\spp(2n,\CC\{I\}).$$
The anti-involution $\sigma_1\otimes\bar\sigma$ on $\Mat(n,\HH\{I,J,K\})\otimes_\R\CC\{i\}$ indices the following anti-involution
$$\tilde\sigma=\psi\circ(\sigma_1\otimes\bar\sigma)\circ\psi^{-1}$$
on $\Mat(2n,\CC)$:
$$\tilde\sigma(m)=\bar m^T.$$
So, as expected, $(\Mat(2n,\CC\{I\}),\tilde\sigma)$ is a Hermitian algebra and
$$\psi(A_+^{\sigma_1\otimes\bar\sigma})=\Herm^+(2n).$$

Since $\psi(1\otimes i)=\Id_{2n} I$, we have to do quaternionification with respect to $\Id I$. So the symmetric space is:
$$\mathfrak U(\OO(4n,\CC))=\{M_1+M_2j\mid M_1\in\spp(2n,\CC),\;M_2\in\Herm^+(2n)\}\subset$$
$$\subset\HH[\Mat(2n,\CC\{I\}),\psi(\Mat(n,\HH\{I,J,K\})),\Id_{2n} I,j].$$
Since $A_\R^\sigma=A^{\bar\sigma\otimes\Id}\cap A^{\bar\sigma\otimes\bar\sigma}$, the real locus of this space is the symmetric space of $\SO^*(4n)$:
$$\mathfrak U(\SO^*(4n))\cong$$
$$=\{M_1+M_2j\mid M_1\in \spp(2n,\CC)\cap\Herm(2n),\;M_2\in \spp(2n,\CC)\cap\Herm^+(2n)\}\subset$$
$$\subset\mathfrak U(\OO(4n,\CC)).$$
After identification $j$ and $\Id_{2n} I$, we obtain it as a subset of $\Mat(2n,\CC\{I\})$:
$$\mathfrak U(\SO^*(4n))=$$
$$=\{M_1+M_2I\mid M_1\in \spp(2n,\CC)\cap\Herm(2n),\;M_2\in \spp(2n,\CC)\cap\Herm^+(2n)\}\subset$$
$$\subset\Mat(2n,\CC\{I\}).$$
\end{ex}

\begin{ex}\label{Precomp3} Now we construct the precompact model. We use the map
$$\phi\colon \Mat(n,\HH\{I,J,K\})\otimes_\R\HH\{i,j,k\}\to \Mat(4n,\R)$$
from the Section~\ref{Isom_phi} to identify $A_\HH$ with $\Mat(4n,\R)$
defined on generators of $A_\HH$ as follows:
$$\phi(a\otimes i)=\begin{pmatrix}
0 & a & 0 & 0 \\ -a & 0 & 0 & 0 \\ 0 & 0 & 0 & -a \\ 0 & 0 & a & 0
\end{pmatrix},\;
\phi(a\otimes j)=\begin{pmatrix} 0 & 0 & a & 0 \\ 0 & 0 & 0 & a \\ -a & 0 & 0 & 0 \\ 0 & -a & 0 & 0
\end{pmatrix},$$
$$\phi(a\otimes k)=\begin{pmatrix}
0 & 0 & 0 & a \\ 0 & 0 & -a & 0 \\ 0 & a & 0 & 0 \\ -a & 0 & 0 & 0
\end{pmatrix},\;
\phi(aI\otimes 1)=\begin{pmatrix} 0 & -a & 0 & 0 \\ a & 0 & 0 & 0 \\ 0 & 0 & 0 & -a \\ 0 & 0 & a &
0
\end{pmatrix},$$
$$\phi(aJ\otimes 1)=\begin{pmatrix}
0 & 0 & -a & 0 \\ 0 & 0 & 0 & a \\ a & 0 & 0 & 0 \\ 0 & -a & 0 & 0
\end{pmatrix},\;
\phi(aK\otimes 1)=\begin{pmatrix} 0 & 0 & 0 & -a \\ 0 & 0 & -a & 0 \\ 0 & a & 0 & 0 \\ a & 0 & 0 &
0
\end{pmatrix}$$
where $a\in\Mat(n,\R)$.

As we have seen, the anti-involution $\sigma_1\otimes\sigma_0$ corresponds under $\phi$ to the following anti-involution on $\Mat(4n,\R)$: $M\mapsto -\Xi M^T \Xi$ where
$$\Xi:=\begin{pmatrix}
0 & 0 & 0 & \Id_n \\
0 & 0 & -\Id_n & 0 \\
0 & \Id_n & 0 & 0 \\
-\Id_n & 0 & 0 & 0
\end{pmatrix}.$$
The anti-involution $\sigma_1\otimes\sigma_1$ corresponds under $\phi$ to the transposition on $\Mat(4n,\R)$. So we obtain the following precompact model of the symmetric space of $\OO(4n,\CC)$:
$$\mathfrak B(\OO(4n,\CC))=\{M\in\phi(A_\HH^{\sigma_1\otimes\sigma_0})\mid 1-M^T M\in \Sym^+(4n,\R)\}$$
where
$$\phi(A_\HH^{\sigma_1\otimes\sigma_0})=\{M\in\Mat(4n,\R)\mid M=-\Xi M^T \Xi\}\cong\spp(4n,\R).$$

To see the precompact model $\mathfrak B(\SO^*(4n))$ for the symmetric space of $\SO^*(4n)$ as a subspace of $\mathfrak B(\OO(4n,\CC))$, we have to intersect $\mathfrak B(\OO(4n,\CC))$ with $\phi(\Mat(n,\HH\{I,J,K\})\otimes_\R\CC\{j\})$. We remind from the Section~\ref{Embedd2}:
$$\phi(\Mat(n,\HH\{I,J,K\})\otimes_\R\CC\{j\})=\left\{m\in\Mat(4n,\R)
\mid M=-\phi(\Id_n\otimes j)M\phi(\Id_n\otimes j)
\right\}.$$
Therefore, we obtain:
$$\mathfrak B(\SO^*(4n))\cong
\left\{M\in \mathfrak B(\OO(4n,\CC))\mid M=-\phi(\Id_n\otimes j)M\phi(\Id_n\otimes j)
\right\}.$$

Under the map $\psi$ from the Section~\ref{Isom_psi}, we can identify $A$ with $\Mat(2n,\CC)$.
The anti-involution $\sigma_1\otimes\Id$ corresponds to the following anti-involution on $\Mat(2n,\CC)$:
$$m\mapsto -\Ome{\Id} m^T\Ome{\Id}.$$
Therefore, $\psi(A^{\sigma_1\otimes\Id})=\spp(2n,\CC)$.

The anti-involution $\sigma_1\otimes\bar\sigma$ corresponds to the following anti-involution on $\Mat(2n,\CC)$: $M\mapsto \bar M^T$. Therefore, we obtain the precompact model for $\SO^*(4n)$:
$$\mathfrak B(\SO^*(4n))=\{M\in\spp(2n,\CC)\mid 1-\bar M^T M\in \Herm^+(2n,\CC)\}.$$
\end{ex}

\begin{ex} Now we construct the projective model. As we have seen, the map $\phi$ defines an $\R$-algebra isomorphism:
$$\phi\colon A_\HH\to \Mat(4n,\R)=:A'.$$
Moreover, the anti-involution $\sigma_1\otimes\sigma_0$ corresponds under $\psi$ to the following anti-involution $\sigma'_0$ on $\Mat(4n,\R)$:
$\sigma'_0(M)=-\Xi M^T \Xi$ where
$$\Xi:=\begin{pmatrix}
0 & 0 & 0 & \Id_n \\
0 & 0 & -\Id_n & 0 \\
0 & \Id_n & 0 & 0 \\
-\Id_n & 0 & 0 & 0
\end{pmatrix}.$$
The anti-involution $\sigma_1\otimes\sigma_1$ corresponds under $\phi$ to the transposition on $\Mat(4n,\R)$.

Further, for $x,y\in (A')^2$
$$\omega(x,y)=\sigma'_0(x)^T\Ome{\Id_{4n}} y=$$
$$=-\Xi x^T \Di{\Xi}
\begin{pmatrix}
0 & 0  & \Id_{2n} & 0\\
0 & 0  & 0 & \Id_{2n}\\
-\Id_{2n} & 0  & 0 & 0\\
0 & -\Id_{2n}  & 0 & 0
\end{pmatrix} y=$$
$$=-\Xi x^T \begin{pmatrix}
0 & 0 & 0 & 0 & 0 & 0 & 0 & \Id_n \\
0 & 0 & 0 & 0 & 0 & 0 & -\Id_n & 0 \\
0 & 0 & 0 & 0 & 0 & \Id_n & 0 & 0 \\
0 & 0 & 0 & 0 & -\Id_n & 0 & 0 & 0 \\
0 & 0 & 0 & -\Id_n & 0 & 0 & 0 & 0\\
0 & 0 & \Id_n & 0 & 0 & 0 & 0 & 0\\
0 & -\Id_n & 0 & 0 & 0 & 0 & 0 & 0\\
\Id_n & 0 & 0 & 0 & 0 & 0 & 0 & 0
\end{pmatrix}y$$
$$h(x,y)=x^T
\begin{pmatrix}
0 & 0 & 0 & 0 & 0 & 0 & \Id_n & 0 \\
0 & 0 & 0 & 0 & 0 & 0 & 0 & \Id_n \\
0 & 0 & 0 & 0 & -\Id_n & 0 & 0 & 0 \\
0 & 0 & 0 & 0 & 0 & -\Id_n & 0 & 0 \\
0 & 0 & -\Id_n & 0 & 0 & 0 & 0 & 0\\
0 & 0 & 0 & -\Id_n & 0 & 0 & 0 & 0\\
\Id_n & 0 & 0 & 0 & 0 & 0 & 0 & 0\\
0 & \Id_n & 0 & 0 & 0 & 0 & 0 & 0
\end{pmatrix} y.$$
By definition of $h$, we use that $$\phi(\Id_n\otimes j)=
\begin{pmatrix}
0 & 0 & \Id_n & 0 \\
0 & 0 & 0 & \Id_n \\
-\Id_n & 0 & 0 & 0 \\
0 & -\Id_n & 0 & 0
\end{pmatrix}.$$ Note, $x\in\Is(\omega)$ if and only if $x\in\Is(\omega')$ where
$$\omega'(x,y)=x^T \begin{pmatrix}
0 & 0 & 0 & 0 & 0 & 0 & 0 & \Id_n \\
0 & 0 & 0 & 0 & 0 & 0 & -\Id_n & 0 \\
0 & 0 & 0 & 0 & 0 & \Id_n & 0 & 0 \\
0 & 0 & 0 & 0 & -\Id_n & 0 & 0 & 0 \\
0 & 0 & 0 & -\Id_n & 0 & 0 & 0 & 0\\
0 & 0 & \Id_n & 0 & 0 & 0 & 0 & 0\\
0 & -\Id_n & 0 & 0 & 0 & 0 & 0 & 0\\
\Id_n & 0 & 0 & 0 & 0 & 0 & 0 & 0
\end{pmatrix}y.$$
So we obtain the projective model for the symmetric space of $\OO(4n,\CC)$:
$$\mathfrak P(\OO(4n,\CC))=\{xA'\mid x\in (A')^2,\;\omega'(x,x)=0,\;h(x,x)\in\Sym^+(4n,\R)\}.$$
We can see the Shilov boundary in this model as the space:
$$\check S(\OO(4n,\CC))\cong\{xA'\mid x\in (A')^2,\;\omega'(x,x)=h(x,x)=0\}.$$

The projective model for $\SO^*(4n)$ can be seen as a subspace of $\mathfrak P(\OO(4n,\CC))$ in the following way. As we have seen in the Section~\ref{Embedd2},
$$\psi(A_\R\otimes_\R\CC\{j\})=\left\{m\in\Mat(4n,\R)\mid m=-\phi(1\otimes j)m\phi(1\otimes j)\right\}.$$
Therefore,
$$\mathfrak P(\SO^*(4n))\cong\{xA'\in\mathfrak P(\GL(2n,\CC))\mid x=-\phi(1\otimes j)x\phi(1\otimes j)\}.$$

To see another projective model for the symmetric space of $\SO^*(4n)$, we remind that $A=\Mat(n,\HH)\otimes_\R\CC$ is to $\Mat(2n,\CC)=:A''$ isomorphic under the map $\psi$ from the Section~\ref{Isom_psi}. The anti-involution $\sigma_1\otimes\Id$ corresponds under this map to the anti-involution $\sigma'$ on $\Mat(2n,\CC)$ given by:
$$\sigma'(m)=-\Ome{\Id_n} m^T\Ome{\Id_n}.$$
The anti-involution $\sigma_1\otimes\bar\sigma$ corresponds under $\psi$ to the complex conjugation composed with transposition on $\Mat(2n,\CC)$.

We define for $x,y\in (A'')^2$,
$$\omega(x,y):=x^T
\begin{pmatrix}
\COme{\Id_n} & 0\\
0 & \COme{\Id_n}
\end{pmatrix}
\Ome{\Id_{2n}}y=$$
$$=x^T\begin{pmatrix}
0 & 0 & 0 & \Id_n \\
0 & 0 & -\Id_n & 0 \\
0 & -\Id_n & 0 & 0 \\
\Id_n & 0 & 0 & 0
\end{pmatrix}y,$$
$$\tilde h(x,y):=\bar x^T\Ome{\Id_{2n} i}y.$$

Then the projective model of the symmetric space for $\SO^*(4n)$ can be seen as:
$$\mathfrak P(\SO^*(4n))=\{xA''\mid x\in (A'')^2,\tilde h(x,x)\in \Herm^+(2n,\CC)\}.$$
We can see the Shilov boundary in this model as the space:
$$\check S(\SO^*(4n))\cong\{xA''\mid x\in (A'')^2,\tilde h(x,x)=0\}.$$

Now we construct the projective model in terms of isotropic subspaces. As before, we identify using the map $L$ the space of $A'$-lines and the space $\Gr(4n,\R^{8n})$ of $4n$-dimensional subspaces of $\R^{8n}$: $$L(xA'):=\Span_\R(xe_1,\dots,xe_{4n})$$ where $e_i$ is the $i$-th standard basis vector of $\R^{4n}$. We define two forms on $\R^{8n}$: for $u,v\in\R^{8n}$,
$$\tilde\omega(u,v):=u^T
\begin{pmatrix}
0 & 0 & 0 & 0 & 0 & 0 & 0 & \Id_n \\
0 & 0 & 0 & 0 & 0 & 0 & -\Id_n & 0 \\
0 & 0 & 0 & 0 & 0 & \Id_n & 0 & 0 \\
0 & 0 & 0 & 0 & -\Id_n & 0 & 0 & 0 \\
0 & 0 & 0 & -\Id_n & 0 & 0 & 0 & 0\\
0 & 0 & \Id_n & 0 & 0 & 0 & 0 & 0\\
0 & -\Id_n & 0 & 0 & 0 & 0 & 0 & 0\\
\Id_n & 0 & 0 & 0 & 0 & 0 & 0 & 0
\end{pmatrix}
v,$$
$$\tilde h(u,v):=u^T\begin{pmatrix}
0 & 0 & 0 & 0 & 0 & 0 & \Id_n & 0 \\
0 & 0 & 0 & 0 & 0 & 0 & 0 & \Id_n \\
0 & 0 & 0 & 0 & -\Id_n & 0 & 0 & 0 \\
0 & 0 & 0 & 0 & 0 & -\Id_n & 0 & 0 \\
0 & 0 & -\Id_n & 0 & 0 & 0 & 0 & 0\\
0 & 0 & 0 & -\Id_n & 0 & 0 & 0 & 0\\
\Id_n & 0 & 0 & 0 & 0 & 0 & 0 & 0\\
0 & \Id_n & 0 & 0 & 0 & 0 & 0 & 0
\end{pmatrix}v.$$

The space of $\tilde\omega$-isotropic vectors of $\R^{8n}$ is denoted by $\Is(\tilde\omega)$. Then the projective model of the symmetric space for $\OO(4n,\CC)$ can be seen as:
$$\mathfrak P'(\OO(4n,\CC))=\{l\in\Lag(\R^{8n},\tilde\omega)\mid\forall x\in l\bs\{0\},\tilde h(x,x)>0\}$$
where $\Lag(\R^{8n},\tilde\omega)$ is the space of all maximal $\tilde\omega$-isotropic subspaces of $\R^{8n}$.

We can see the Shilov boundary in this model as the space:
$$\check S(\OO(4n,\CC))\cong\{l\in\Lag(\R^{8n},\tilde\omega)\mid\forall x\in l\bs\{0\},\tilde h(x,x)=0\}.$$

The projective model for the symmetric space of $\SO^*(4n)$ can be seen as a subspace of $\mathfrak P(\OO(4n,\CC))$:
$$\mathfrak P'(\SO^*(4n))=\{l\in\mathfrak P(\OO(4n,\CC))\mid \delta(l)=l\},$$
where
$$\begin{matrix}
\delta\colon & \R^{8n} & \to & \R^{8n}\\
& v & \mapsto & \begin{pmatrix}
\phi(\Id_n\otimes j) & 0 \\
0 & \phi(\Id_n\otimes j)
\end{pmatrix} v.
\end{matrix}$$

To see another projective model for the symmetric space of $\SO^*(4n)$, we remind that $A=\Mat(n,\HH)\otimes_\R\CC$ is to $\Mat(2n,\CC)=:A''$ isomorphic under the map $\psi$ from the Section~\ref{Isom_psi}. The anti-involution $\sigma_1\otimes\Id$ corresponds under this map to the anti-involution $\sigma'$ on $\Mat(2n,\CC)$ given by:
$$\sigma'(m)=-\Ome{\Id_n} m^T\Ome{\Id_n}.$$
The anti-involution $\sigma_1\otimes\bar\sigma$ corresponds under $\psi$ to the complex conjugation composed with transposition on $\Mat(2n,\CC)$.

To construct the projective model in terms of Lagrangians, as before, we identify using the map $L$ the space of $A''$-lines and the space $\Gr(2n,\CC^{4n})$ of $2n$-dimensional subspaces of $\CC^{4n}$:
$$L(xA''):=\Span_\R(xe_1,\dots,xe_{2n})$$ where $e_i$ is the $i$-th standard basis vector of $\CC^{2n}$. We define two forms on $\R^{8n}$: for $u,v\in\R^{8n}$,
$$\tilde\omega(u,v):=u^T
\begin{pmatrix}
\COme{\Id_n} & 0\\
0 & \COme{\Id_n}
\end{pmatrix}
\Ome{\Id_{2n}}v=$$
$$=u^T\begin{pmatrix}
0 & 0 & 0 & \Id_n \\
0 & 0 & -\Id_n & 0 \\
0 & -\Id_n & 0 & 0 \\
\Id_n & 0 & 0 & 0
\end{pmatrix}v,$$
$$\tilde h(u,v):=\bar u^T\Ome{\Id_n i}v.$$
\end{ex}
Then the projective model of the symmetric space for $\SO^*(4n)$ can be seen as:
$$\mathfrak P'(\SO^*(4n))=\{l\in\Lag(\CC^{4n},\tilde\omega)\mid\forall x\in l\bs\{0\},\tilde h(x,x)>0\}$$
where $\Lag(\CC^{4n},\tilde\omega)$ is the space of all maximal $\tilde\omega$-isotropic subspaces of $\CC^{4n}$.

We can see the Shilov boundary in this model as the space:
$$\check S(\SO^*(4n))\cong\{l\in\Lag(\CC^{4n},\tilde\omega)\mid\forall x\in l\bs\{0\},\tilde h(x,x)=0\}.$$

\begin{ex} Now we construct the quaternionic structures model and the complex structures model.

We use the map $\psi$ from the Section~\ref{Isom_psi}, to identify $A$ with $A':=\Mat(2n,\CC\{I\})$.
$$\begin{matrix}
\psi\colon & \Mat(n,\HH\{I,J,K\})\otimes_\R\CC\{i\} & \to & \Mat(2n,\CC\{I\})\\
 & (q_1+q_2J) + (p_1+p_2J)i & \mapsto &
\begin{pmatrix}
q_1+p_1I & q_2+p_2I\\ -\bar q_2-\bar p_2I & \bar q_1+\bar p_1I
\end{pmatrix}.
\end{matrix}$$
where $q_1,q_2,p_1,p_2\in\Mat(n,\CC\{I\})$.

The induced by $\Id\otimes\bar\sigma$ involution
$$\sigma':=\psi\circ(\Id\otimes\bar\sigma)\circ\psi^{-1}$$
on $\Mat(2n,\CC)$ acts in the following way:
$$m\mapsto -\Omega\bar m\Omega$$
where $\Omega=\Ome{\Id_{n}}\in \Mat(2n,\CC)$. We also denote by $\Omega_0:=\diag(\Omega,\Omega)\in\Mat(4n,\CC)$.

If we take a quaternionic structure $Q$ on $A^2$ then we define
$$Q':=\psi\circ Q\circ\psi^{-1}.$$
We can see $Q'$ as a complex $4n\times 4n$-matrix acting on $(A')^2$ in the following way: for $x\in (A')^2$,
$$Q'(x):=Q'\sigma'(x)=-Q'\Omega_0\bar x\Omega_0.$$
$Q'$ is a quaternionic structure, therefore,
$$-x=(Q')^2(x)=Q'\Omega_0\overline{Q'\Omega_0\bar x\Omega_0}\Omega_0=-Q'\Omega_0\bar Q'\Omega_0 x.$$
So we obtain, $Q'$ is a quaternionic structure on $A$ if and only if
$$Q'\Omega_0\bar Q'\Omega_0=\Id_{4n}.$$

The induced by $\sigma_1\otimes\Id$ anti-involution $$\psi\circ(\sigma_1\otimes\Id)\circ\psi^{-1}$$
on $\Mat(2n,\CC)$ acts in the following way:
$$m\mapsto -\Omega m^T\Omega.$$
So we define the standard symplectic form $\omega$ on $(A')^2$ with respect to this anti-involution: for $x,y\in (A')^2$,
$$\omega(x,y):=-\Omega_0 x^T\Omega_0\Ome{\Id_{2n}}y$$
and for a quaternionic structure $Q'$, we define
$$h_{Q'}(x,y):=\omega(Q'(x),y)=\bar x^T\Omega_0(Q')^T\Omega_0\Omega_0\Ome{\Id_{2n}}y=$$
$$=-\bar x^T\Omega_0(Q')^T\Ome{\Id_{2n}} y.$$

The quaternionic structure model for $\GL(2n,\CC)$ is then:
$$\mathfrak C(\OO(4n,\CC)):=\left\{Q'\in\Mat(4n,\CC)\left|\begin{array}{l}
Q'\Omega_0\bar Q'\Omega_0=\Id_{4n},\\-\Omega_0(Q')^T\Ome{\Id_{2n}}\in\Herm^+(4n,\CC)
\end{array}
\right.\right\}.$$
The model for the symmetric space of $\mathfrak C(\SO^*(4n))$ can be seen as a subset of $\mathfrak C(\OO(4n,\CC))$ whose elements commute with $\sigma'$ i.e. $\sigma'(Q'(x))=Q'(\sigma'(x))$. Therefore:
$$\sigma'(Q'(x))=-\Omega\overline{Q'(x)}\Omega=\Omega\overline{Q'\Omega\bar x\Omega}\Omega=-\Omega\bar Q'\Omega x,$$
$$Q'(\sigma'(x))=-Q'\Omega \overline{\sigma'(x)} \Omega=Q'\Omega \overline{\Omega \bar x\Omega} \Omega=-Q'x$$
and we obtain:
$$\mathfrak C(\SO^*(4n))\cong\{Q'\in \mathfrak C(\OO(4n,\CC))\mid Q'=\Omega\bar Q'\Omega\}.$$

The space $\mathfrak C(\SO^*(4n))$ can be also seen directly as complex structures $Q$ on $A_\R^2$ such that the form:
$$h_Q(x,y)=\sigma_1(x)^T\sigma_1(Q)^T\Omega_0y$$
is positive definite. So we obtain:
$$\mathfrak C(\SO^*(4n))=\left\{Q'\in\Mat(2n,\HH\{i,j,k\}\left|
\begin{array}{l}
\sigma_1(Q)^T\Omega_0\in\Herm^+(2n,\HH\{I,J,K\}),\\
Q^2=-\Id_{2n}
\end{array}\right.\right\}.$$
\end{ex}

\appendix

\input{Herm_appendix}

\bibliographystyle{plain}
\bibliography{bibl}

\end{document}

%% file: Herm_appendix.tex
\section{Classification of Hermitian algebras}\label{app:classification}

The goal of this section is to classify all Hermitian algebras. To do this, we consider a more general class of algebras that we call pre-Hermitian, and classify them.

Let $(A,\sigma)$ be a ring  with an anti-involution $\sigma$. As usual, we say that $a\in A$ is symmetric if $\sigma(a)=a$, denote the set of all symmetric elements in $A$ by $A^\sigma$. Clearly, if $2\in A^\times$, then  $A^\sigma$ is a (unital) Jordan ring under the operation $a\circ b=2^{-1}(ab+ba)$.

If  $A$ is an algebra over a commutative ring $F$ and $\sigma$ is $F$-linear, then we will refer to $(A,\sigma)$ as an $F$-algebra.

\begin{df}\label{Jacobson}
The \defin{Jacobson radical} $J(A)$ of a unital ring $A$ is the set of all $x\in A$ such that $1+AxA\subseteq A^\times$. In particular, $1 + J(A)$ is a subgroup of $A^\times$.
\end{df}

It is  well-known (see e.g., \cite{Lam01}) that $J(A)$ is a nilpotent ideal for any (left or right) Artinian ring $A$. Moreover, such a ring  
is semisimple if and only if $J(A)=\{0\}$. In particular, this holds for finite dimensional algebras over any field.

\begin{prop}\label{pr:sigma radical}
$J(A)$ is invariant under any anti-involution of any ring $A$.
\end{prop}

\begin{proof}
Clearly, if $\sigma$ is any anti-involution of $A$ then  $1+A\sigma(x)A\subset A^\times$ for all $x\in J(A)$ hence $\sigma(x)\in J(A)$.
\end{proof}



\begin{definition}
We say that a ring $(A,\sigma)$ is {\it pre-Hermitian} if $A^\sigma\cap J(A)=\{0\}$.
\end{definition}

If $J(A)=\{0\}$, $A$ is sometimes called {\it Jacobson semisimple} (in particular, any $C^*$-algebra is Jacobson semisimple as a consequence of Gelfand-Naimark theorem). Also note that any $\RR$-subalgebra of $\Mat (n,\CC)$ invariant under the Hermitian transposition is semisimple and, therefore, Jacobson semisimple (see \cite{R}, Exercise 18, p. 168).
By definition, any Jacobson semisimple $(A,\sigma)$ is pre-Hermitian.


\begin{definition}
We say that a ring $(A,\sigma)$ is {\it Hermitian} if $a^2+b^2=0$ for $a,b\in A^\sigma$ implies that  $a=b=0$.
\end{definition}

In particular, nonzero symmetric elements of Hermitian rings are not nilpotent.

\begin{rem}
In contrast to the main part of this paper, we do not assume in this appendix that a Hermitian ring is an algebra over a real closed field.
\end{rem}

\begin{rem}
If $(A,\sigma)$ is a Hermitian ring such that $J(A)=\{0\}$, then, similarly to the Proposition~\ref{minus_antisym}, we can show that $-a^2\in A^\sigma_{\geq 0}$ (See definition \ref{df:cone}) for all $a\in A^{-\sigma}$.
\end{rem}

\begin{rem} Similarly to  $A^\sigma_{\geq 0}$, for any ring $(A,\sigma)$ denote by $A^\sigma_{>0}$ the set of all sums $a_1^2+\cdots +a_n^2$, $n\ge 1$, where all $a_i$ are nonzero elements of $A^\sigma$. By definition, $A^\sigma_{>0}$ is an additive sub-semigroup of $A$, which may or may not contain $0$. Clearly, $A^\sigma_{\ge 0}=A^\sigma_{>0}\cup \{0\}$. Also, it is immediate that if $0\notin A^\sigma_{>0}$, then $(A,\sigma)$ is Hermitian. It would be interesting to classify those rings in which the opposite implication holds.

\end{rem}

\begin{prop}\label{Herm_preHerm}
Any Hermitian ring $A$ with nilpotent $J(A)$ is pre-Hermitian.
\end{prop}

\begin{proof}
Assume, there exists $0\neq x\in J(A)\cap A^\sigma$. Since $J(A)$ is a nilpotent ideal, there exist $x\in J(a)$ such that $a:=x^n\neq 0$ and $a^2=0$ for some $n>0$. 
Therefore, $a=0$, which is a  contradiction.
\end{proof}

\begin{prop}\label{preHerm_prop}
Let $(A,\sigma)$ be a pre-Hermitian ring.
Then  $\sigma(x)=-x$ and $xy=-yx$ for any $x,y\in J(A)$. In particular, $2x^2=0$.
\end{prop}

\begin{proof}
Let $x\in J(A)$. Since $x+\sigma(x)\in J(A)\cap A^\sigma=\{0\}$, $\sigma(x)=-x$. Furthermore, $-xy=\sigma(xy)=\sigma(y)\sigma(x)=yx$ for $x,y\in J(A)$.
\end{proof}

\begin{proposition}\label{pr:skewsym radical}
If $(A,\sigma)$ is pre-Hermitian then:

(a) $\sigma(x)=-x$ for all $x\in J(A)$ and
$xa=\sigma(a)x$
for all $x\in J(A)$, $a\in A$.

(b)
$yx=-xy$ for all $x,y\in J(A)$ and $xyz=0$ for all $x,y,z\in J(A)$.

(c) $(\sigma(a)-a)xy=xy(\sigma(a)-a)=0$ all $a\in A$, $x,y\in J(A)$
\end{proposition}

\begin{proof} Prove (a). Since $\sigma(x)\in J(A)$ for all $x\in J(A)$ by Proposition \ref{pr:sigma radical},  $x+\sigma(x)\in J(A)\cap A^\sigma=\{0\}$. This proves the fist assertion. To prove the second assertion, using the fact that $xa\in J(A)$ for all $x\in J(A)$, $a\in A$, we obtain
$$xa=-\sigma(xa)=-\sigma(a)\sigma(x)=\sigma(a)x\ .$$
This proves (a).

To prove (b) note that $xy=-\sigma(xy)=-yx$
for all $x,y\in J(A)$. To prove the second assertion, note that
$$\sigma(xyz)=-zyx=-yxz=xyz$$
for all $x,y,z\in J(A)$ hence $xyz=0$.
 This proves (b).

 To prove (c) note that on the one hand, $xya=x\sigma(a)y=axy$ and on the other hand, $(xy)a=\sigma(a)(xy)$ for all $a\in A$ and $x,y\in J(A)$. This proves (c).

The proposition is proved.
\end{proof}

\begin{corollary} If $(A,\sigma)$ and $(B,\sigma')$ are pre-Hermitian algebras over $F$, $char ~F \ne 2$ and $(A,\sigma)\otimes (B,\sigma')$ is also pre-Hermitian, then either $J(A)=\{0\}$ or $J(B)=\{0\}$.

\end{corollary}

\begin{proof} Indeed, If $x\in J(A)$, $y\in J(B)$, then $x\otimes y\in (A\otimes B)^{\sigma\otimes \sigma'}$ by Proposition \ref{pr:skewsym radical}.
\end{proof}




\begin{example} 
\label{ex:exterior thin} Let  $V$ be a finite-dimensional vector space over a field $F$, $char ~F\ne 2$,   $\hat A=\Lambda(V)$ be the exterior algebra of $V$, and $\hat \sigma$ be the unique anti-involution of $\hat A$ such that $\hat \sigma(v)=-v$ for all $v\in V$. Denote by $A$ the quotient of $\hat A$ by the ideal generated by $\Lambda^3 V$, so that $A= F\oplus V\oplus \Lambda^2 V$ as a vector space. Since $\hat \sigma$ preserves the ideal generated by $\Lambda^3 V$, it induces a well-defined  anti-involution  $\sigma$ on $A$. Clearly, $J(A)=V\oplus \Lambda^2 V$ and $\sigma(x)=-x$ for all $x\in J(A)$. Thus, $(A,\sigma)$ is pre-Hermitian with $A^\sigma=F$ and $A^\sigma_{> 0}=F_{> 0}$.

\end{example}





\begin{proposition}
\label{pr:orthogonal radical}
Let $A$ be a pre-Hermitian ring. Then
$$A^\circ \cdot J(A)=J(A)\cdot A^\circ =0$$
where $A^\circ=A\cdot [A,A]=[A,A]\cdot A$ is the ideal of $A$  generated by all commutators $[a,b]=ab-ba$, $a,b\in A$.

\end{proposition}

\begin{proof} It follows from Proposition \ref{pr:skewsym radical} that
$$abx=x\sigma(ab)=x\sigma(b)\sigma(a)=bx\sigma(a)=bax$$
for all $a,b\in A$, $x\in J(A)$ hence $[a,b]x=0$. Also, $x[a,b]=[\sigma(b),\sigma(a)]x=0$.

The proposition is proved.
\end{proof}

\begin{example} $A^\circ=A$ if $A$ is simple noncommutative and
$(\Mat_n (A))^{\circ}=\Mat_n (A)$ if $A^{\circ}=A$.
\end{example}

The following definition is motivated by Proposition \ref{pr:skewsym radical}.

\begin{definition} We say that a (unitless) ring $J$ is {\it nilpotent pre-Hermitian} if:

$\bullet$ $2x=0$ implies that $x=0$ (this is relevant only in characteristic $2$).

$\bullet$ $yx=-xy$  and  $xyz=0$ for all $x,y,z\in J$.

\end{definition}

Clearly, for any nilpotent pre-Hermitian ring $J$ with $0\notin 2(J\setminus \{0\})$, the assignments $j\mapsto -j$ define an anti-involution on $J$ with no fixed points in $J\setminus \{0\}$.

Let $J$ be a ring and $K$ be a commutative ring that acts on $J$ from the left. We denote the action by:
$$\begin{matrix}
\act \colon & K\times J & \to & J\\
& (k,j) & \mapsto & k\act j.
\end{matrix}$$
Similarly, if $K$ acts in $J$ from the right, we denote the action by:
$$\begin{matrix}
\ract \colon & J\times K & \to & J\\
& (j,k) & \mapsto & j\ract k.
\end{matrix}$$
We say that a ring $J$  is a \defin{left $K$-algebra} if $J$ is a left $K$-module with respect to $\act$ and
$$k\act (jj')=(k\act j) j'$$
for all $j,j'\in J$, $k\in K$ (so we will sometimes denote it simply by $k\act jj'$).\\
We say that a ring $J$  is a \defin{right $K$-algebra} if $J$ is a right $K$-module with respect to $\ract$ and
$$(jj')\ract k=j(j'\ract k)$$
for all $j,j'\in J$, $k\in K$ (so we will sometimes denote it simply by $jj'\ract k$).

\begin{definition} Let $K$ be a commutative unital ring with an involution $\overline{\cdot}$, $J$ be a unitless $K$-algebra and $\gamma:K\to J$ be a homomorphism of abelian groups. We say that $J$ is a \defin{$(K,\overline{\cdot},\gamma)$-algebra} if:
\begin{equation}
\label{eq:Kgamma algebra}
\begin{split}
j(\gamma(kk')-k\act\gamma(k')-k'\act\gamma(k)-\gamma(k)\gamma(k'))=0,\\
j(k\act j')=(\overline k\act j-j\gamma(k))j',(k-\overline k)\act jj'=j\gamma(k)j'
\end{split}
\end{equation}
for all $j,j'\in J$, $k,k'\in K$.

\end{definition}

When $J$ satisfies $J^3=0$, e.g., when $J$ is  pre-Hermitian, the conditions 
\eqref{eq:Kgamma algebra} simplify to
\begin{equation}
\label{eq:Kgamma algebra simplified}
\begin{split}
j(\gamma(kk')-k\act\gamma(k')-k'\act\gamma(k))=0,\\
j(k\act j')=\overline k\act jj'=k\act jj'
\end{split}
\end{equation}
for all $j,j'\in J$, $k\in K$.


\begin{proposition}
\label{pr:semidirect sum}
Let $K$ be a commutative unital ring and $J$ be a $(K,\overline{\cdot},\gamma)$-algebra. Then:

(a) $J\oplus K$ has a structure of an associative unital ring with the multiplication given by
\begin{equation}
\label{eq:Kgamma product}
(j+k)(j'+k')=jj'+k\act j'+\overline {k'}\act j-j\gamma(k')+kk'
\end{equation}
for all $j,j'\in J$, $k,k'\in K$ (we denote this ring by $J\rtimes_\gamma K$ and refer to  as semidirect sum of $J$ and $K$ over $\gamma$).

(b) $J$ is a two-sided ideal in $J\rtimes_\gamma K$, moreover, the projection to the second factor is a surjective homomorphism $J\rtimes_\gamma K \twoheadrightarrow K$ of rings whose kernel is $J$.

(c) Suppose that $2\in K^\times$ and $\gamma(kk')=\overline k\act \gamma(k')+k'\act \gamma(k)-\gamma(k)\gamma(\overline {k'})+\frac{\gamma(k)\gamma(k')}{2}$ for all $k,k'\in K$. Then the assignments $k\mapsto \iota(k):=\overline k+\frac{\gamma( k)}{2}$  define an injective ring homomorphism $\iota:K\hookrightarrow J\rtimes_\gamma K$. Suppose additionally that $j\gamma(k)=2j\gamma(\overline k)+\gamma(\overline k)j$  for all $j\in J$, $k\in K$. Then  $J\rtimes_\gamma K=J\rtimes_{\bf 0} \iota(K)$.

\end{proposition}

\begin{proof} Define a map $\ract:J\times K\to J$ by
$$j\ract k:=\overline k\act j-j\gamma(k)\ .$$

\begin{lemma}
\label{le:act vs ract}
$\ract$ is an action of $K$ on $J$ commuting with $\act$ and
$(jj')\ract k=j(j'\ract k)$ for all  $j,j'\in J$, $k\in K$, (i.e., $J$ is a right $K$-algebra).

\end{lemma}

\begin{proof} First, show that $\ract$ commute with $\act$. Indeed,
$$(k\act j)\ract k'=\overline {k'}\act(k\act j) - (k\act j)\gamma(k')=k \act(\overline {k'}\act j) -k\act j\gamma(k') =k\act (j\ract k')$$
for all $j\in J$, $k,k'\in K$ because $J$ is a left $K$-algebra.

Furthermore,
$(j\ract k)\ract k'=\overline {k'}\act (j\ract k)-(j\ract k)\gamma(k')=\overline {kk'} \act j-\overline {k'}\act j\gamma(k)-(j\ract k)\gamma(k')$
$$\overline {kk'} \act j-(j\ract k')\gamma(k)-j\gamma(k)\gamma(k')-(j\ract k)\gamma(k')
=\overline {kk'}\act j-j\gamma(k'k)=j\ract (kk')$$
for all $k,k'\in K$, $j\in J$  by the commutation of $\ract$ with $\act$ and  the first condition of \eqref{eq:Kgamma algebra}.

Since $J\gamma(1)=\{0\}$ by the first condition \eqref{eq:Kgamma algebra}, i.e., $\ract 1=Id_J$, this proves that $J$ is a  $K$-module under $\ract$.

Finally,
$$j(j'\ract k)=j(\overline k\act j'-j'\gamma(k))=k\act j j'-j\gamma(\overline k)j'-jj'\gamma(k)=\overline k\act jj'-jj'\gamma(k)=(jj')\ract k$$
for all $k,k'\in K$, $j\in J$ by the second and third conditions of \eqref{eq:Kgamma algebra}.
This proves that $J$ is a right $K$-algebra under $\ract$.

The lemma is proved.
\end{proof}

Note that the identity
\begin{equation}
\label{eq:middle action}
(j\ract k)j'=j(k\act j')
\end{equation}
for all $k\in K$, $j,j'\in J$ is equivalent to the second condition \eqref{eq:Kgamma algebra}.

The following is immediate and well-known.

\begin{lemma} 
\label{le:associative biaction}
Let $R$ and $S$ be associative ring, $R$ is unitless, $S$ is unital  and $R$ is an $S$-bimodule such that
\begin{equation}
\label{eq:associative biaction}
s\act (rr')=(s\act r)r',~(rr')\ract s=r(r'\ract s),~(r\ract s)r'=r(s\act r') 
\end{equation}
for all $r,r'\in R$, $s,s'\in S$.
Then  $A:=R\oplus S$ is a unital associative ring with the product given by 
$$(r+s)(r'+s')=rr'+s\act r'+r\ract s'+ss'$$
for all $r,r'\in R$, $S,s'\in S$.

\end{lemma} 

Thus, \eqref{eq:middle action} and Lemma \ref{le:act vs ract} guarantee that all assumptions of Lemma \ref{le:associative biaction} hold for $R=J$, $S=K$,
therefore, $J\oplus K$ is a unital associative ring.  This proves (a).

Part (b) is obvious.

Prove (c). Indeed,
$\iota(k)\iota(k')=(\overline k+\frac{\gamma(k)}{2})(\overline{k'}+\frac{\gamma(k')}{2})=\overline{kk'}+\frac{\overline k\act \gamma(k')+\gamma(k)\ract \overline{k'}}{2}+\frac{\gamma(k)\gamma(k')}{4}$
$$
=\overline{kk'}+\frac{\overline k\act \gamma(k')+k'\act \gamma(k)-\gamma(k)\gamma(\overline {k'})}{2}+\frac{\gamma(k)\gamma(k')}{4}
=\overline{kk'}+\frac{\gamma(k'k)}{2}=\iota(kk')
$$
for all $k,k'\in K$ by the first relation \eqref{eq:Kgamma algebra}. Therefore, $\iota$ is a homomorphism of rings. Its injectivity follows because $\iota$ splits the canonical homomorphism from (b). Finally,
$$j\iota(k)=j(\overline k+\frac{\gamma(k)}{2})=kj-j\gamma(\overline k)+\frac{j\gamma(k)}{2}=k j+\frac{\gamma(\overline k)}{2}j=\iota(\overline k)j$$
for all $j\in J$, $k\in K$. This proves that $J$ is a $(\iota(K),\widetilde {\cdot},{\bf 0})$-algebra with the trivial $\gamma={\bf 0}$ and:

$\bullet$ The involution $\widetilde {\cdot}$ defined by $\widetilde{\iota(k)}=\iota(\overline k)$ for all $k\in K$.

$\bullet$  The left action of $\iota(k)$ on $J$  by the left multiplication in $J\rtimes_\gamma K$.

In particular, $J\rtimes_\gamma K=J\rtimes_{\bf 0} \iota(K)$. Part (c) is proved.

The proposition is proved.
\end{proof}

Given a  commutative ring $(K,\overline{\cdot})$ with anti-involution, we say that  a left  $K$-algebra  $J$ is a \defin{$(K,\overline{\cdot})$-algebra} if
$$\overline k\act jj'=k\act jj'=j(k\act j')$$
for all $j,j'\in J$,  $k\in K$. Clearly, $(K,\overline{\cdot})$-algebras are same as $(K,\overline{\cdot},{\bf 0})$-algebras. Also, in view of \eqref{eq:Kgamma algebra simplified}, any nilpotent pre-Hermitian $(K,\overline{\cdot},\gamma)$-algebra  is automatically a $(K,\overline{\cdot})$-algebra.

\begin{proposition}
\label{pr:semidirect sum simplified}
 Let $(K,\overline{\cdot})$ be any commutative ring with anti-involution and let $J$ be any nilpotent pre-Hermitian ring. Suppose that  $J$ is a $(K,\overline{\cdot})$-algebra and let $\gamma:K\to J$ be any homomorphism of abelian groups such that:
$\gamma(kk')=\overline k\act \gamma(k')+k'\act \gamma(k)$ for all $k,k'\in k$.
Then:

(a) $J$ is a $(K,\overline{\cdot},\gamma)$-algebra.

(b) Suppose additionally that  $\gamma(\overline k)=\gamma(k)$ for all $k\in K$.  Then the assignments $j+k\mapsto \gamma(k)-j+\overline k$ define an  anti-involution $\sigma$ on $J\rtimes_\gamma K$.  Moreover, if $K$ is semisimple (i.e., is a direct sum of fields), then $(J\rtimes_\gamma K,\sigma)$ is pre-Hermitian.

\end{proposition}

\begin{proof}
Prove (a). Indeed, the conditions \eqref{eq:Kgamma algebra simplified} hold automatically for this choice of $\gamma$ and because $J$ is $(K,\overline{\cdot})$-algebra. This proves (a).

Prove (b). First,  let us verify that $\sigma$ is an anti-involution.
Indeed,
$$\sigma(\sigma(j+k))=\sigma(\gamma(k)-j+\overline k)=j-\gamma(k)+\sigma(\overline k)=j-\gamma(k)+k+\gamma(\overline k)=j+k$$
for all  $j\in J$, $k\in K$. That is, $\sigma^2=1$.

Furthermore, by definition \eqref{eq:Kgamma product},
$j  k=\overline kj-j\gamma(k)=\overline kj+\gamma(k)j=k\sigma(k)j$
hence $kj=j\sigma(k)$ for all $j\in J$, $k\in K$.
Then
$$\sigma(kk')=\overline {kk'}+\gamma(kk')=\overline {kk'}+\overline k\gamma(k')+k'\gamma(k)=\overline {kk'}+\overline k(\sigma(k')-\overline k')+k'\gamma(k)=$$
$$=\overline k\sigma(k')+\gamma(k)\sigma(k')=\sigma(k)\sigma(k')$$
for all $k,k'\in K$. Clearly,
$$\sigma(jj')=-jj'=j'j=\sigma(j')\sigma(j)$$
for all $j,j'\in J$. Also,
$$\sigma(kj)=\sigma(k\act j)=-k\act j=-kj=-j\sigma(k)=\sigma(j)\sigma(k)$$
$$\sigma(jk)=\sigma(\sigma(k)j)=\sigma(j)\sigma(\sigma(k)=-jk=-\sigma(k)j=\sigma(k)\sigma(j)$$
for all $j\in J$, $k\in K$.

This proves the first assertion.  To prove the second assertion, note that semisimplicity of $K$ and Proposition \ref{pr:semidirect sum}(b) imply that Jacobson radical of $J\rtimes_\gamma K$ is $J$. This finishes proof of (b).

The proposition is proved.
\end{proof}

The following is an immediate corollary of Proposition \ref{pr:semidirect sum simplified}.

\begin{corollary}
\label{cor:pre-Hermitian semidirect sum}
Let $(A,\sigma)$ be any pre-Hermitian ring and $K$ be any commutative unital subring of $A$ such that $K\cap J(A)=\{0\}$ and $\sigma(K)\subset K+J(A)$. Then

(a) $J(A)$ is both a $(K,\overline{\cdot})$-algebra and a $(K,\overline{\cdot},\gamma)$-algebra, where:

$\bullet$ $J(A)$ is a $K$-algebra via  left multiplication.

$\bullet$ $\overline{\cdot}:K\to K$  and  $\gamma:K\to J(A)$ are determined by $\sigma(k)=\gamma(k)+\overline k$ for all $k\in K$.

(b) The subring of $A$ generated by $K$ and $J(A)$ is naturally isomorphic to $J(A)\rtimes_\gamma K$.

\end{corollary}

Recall that $F$ is a \defin{perfect field} if every irreducible polynomial over $F$ has distinct roots. In particular, all fields of characteristic zero and all finite fields are perfect.

\begin{theorem} [Wedderburn-Mal'cev  theorem (see \cite{R}, Exercise 18, p. 191)]
\label{th:semidirect}
Let $R$ be a finite dimensional algebra over a perfect field $F$. Then

(a) There is a splitting $\iota$ of the short exact sequence $J(R)\to R\to S:=R/J(R)$, e.g., $R=\iota(S)\oplus J(R)$.

(b) The images of all splittings $\iota:S\hookrightarrow R$ are conjugate in $R$ by $1+J(R)$.
\end{theorem} 

\begin{remark} In fact, the multiplication in $R$ in Theorem \ref{th:semidirect} is as in Lemma \ref{le:associative biaction} since $J(R)$ is naturally a bimodule over $S\iota(S)$.

\end{remark}

The following is immediate.

\begin{lemma} In the assumptions of Theorem \ref{th:semidirect}  suppose that $\sigma(\iota(S))=\iota(S)$ for an anti-involution $\sigma$ of $R$ and a splitting $\iota:S\hookrightarrow R$. Then  $R^\sigma=\iota(S)^\sigma\oplus J(R)^\sigma$.

\end{lemma}

The following is well-known (cf. \cite {R}, Theorem 25C.17).

\begin{theorem} In the assumptions of Theorem \ref{th:semidirect}, there exists a faithful $n$-dimensional representation $\rho$ of $R$ into the algebra of ${\bf n}$-block upper triangular matrices (for some partition ${\bf n}$ of $n$) such that  $\rho(J(R))$ is in the block-strictly upper triangular part of $Mat ({\bf n}, F)$ and the image under $\rho$ of at least one splitting $\iota:S\hookrightarrow R$ is in the block-diagonal part of $Mat ({\bf n}, F)$.

\end{theorem}

We will use  Theorem \ref{th:semidirect} to finish the classification of finite-dimensional pre-Hermitian algebras over perfect fields as follows.

\begin{theorem}
\label{th:prehermitian classification}
Let $(A,\sigma)$ be a finite-dimensional pre-Hermitian algebra over a perfect field $F$, denote by $K$ the maximal abelian ideal  of the semisimple quotient $S=A/J(A)$ and by $B$ its complement so that $S=B\oplus K$.
Then there is a unique copy of $B$ in $A$ splitting the  canonical homomorphism $\pi:A\twoheadrightarrow S$, that is, $A= B\oplus  \tilde K$, $\sigma(B)=B$, $\sigma(\tilde K)=\tilde K$, where $\tilde K=\pi^{-1}(K)$. More precisely, $\tilde K=J(A)$ if $K= \{0\}$ and $\tilde K\cong J(A)\rtimes_\gamma K$ otherwise (in the notation of Corollary \ref{cor:pre-Hermitian semidirect sum}).

\end{theorem}

\begin{proof} Indeed, $B^\circ=B$ because each simple component $C$ of $B$ satisfies $C^\circ=C$ since $[C,C]\ne 0$. Furthermore, in the notation of Theorem \ref{th:semidirect}(a), fix a splitting $\iota:S\hookrightarrow A$ of homomorphism $\pi:A\twoheadrightarrow S$.
 Then $\iota(B)=\iota(B^\circ)=\iota(B)^\circ\subset A^\circ$ hence $\iota(B)J(A)=J(A)\iota(B)=\{0\}$ by Proposition \ref{pr:orthogonal radical}. To prove the first assertion note that, by Theorem \ref{th:semidirect}(b), the images of $B$ under any splitting $S\hookrightarrow A$ are conjugate to $\iota(B)$ by $1+ J(A)$. Since
 $(1+J(A))\iota(b)=\iota(b)(1+J(A))=\iota(b)$
 for all $b\in B$, we see that $\iota(B)$ is a unique copy of $B$ in $A$. In particular, $\sigma(\iota(B))=\iota(B)$.

Furthermore, by definition, $\sigma(\tilde K)=\tilde K$ and $\tilde K=\iota(K)+J(A)$ as a vector space over $F$. Since $\iota(B) \iota(K)= \iota(K)\iota(B)=\{0\}$, we see that
 $\iota(B)\tilde K=\tilde K \iota(B)=\{0\}$.
Therefore, $A=B\oplus \tilde K$, as an algebra.

If $K=\{0\}$, then, clearly, $\tilde K=J(A)$. Suppose that $K\ne \{0\}$.  Since $\sigma(\iota(K))\subset \tilde K=\iota(K)+J(A)$, we see that $\iota(K)$ satisfies the hypotheses of Corollary \ref{cor:pre-Hermitian semidirect sum}. Therefore,
$\tilde K=J(A)\rtimes_\gamma \iota(K)$.

The theorem is proved.
\end{proof}

The following is immediate.

\begin{corollary} In the assumptions of Theorem \ref{th:prehermitian classification},  one has

(a) If $K=\{0\}$ then $A$ is Hermitian if and only if $B$ is Hermitian. Otherwise, $A$ is Hermitian if and only if both $B$ and $\tilde K$ are Hermitian.

(b) If $A$ is Hermitian, then

$A^\sigma_{\geq 0}=\begin{cases}
B^\sigma_{\geq 0} & \text{if $\tilde
K=\{0\}$}\\
B^\sigma_{\geq 0}\oplus \tilde K^\sigma_{\geq 0} &\text{otherwise}
\end{cases}$ and  $A^\sigma_{+}=\begin{cases}
B^\sigma_{+} & \text{if $\tilde
K=\{0\}$}\\
B^\sigma_{+}\oplus K^\sigma_+ &\text{otherwise}
\end{cases}$.
\end{corollary}

We discuss now anti-involutions on simple rings.

Clearly, if  $(A,\sigma)$ is a Hermitian algebra over a field $F$, then so is $F$, i.e., $F$ is a formally real field. In particular, if $A^\sigma=F$, then $(A,\sigma)$ is Hermitian if and only if $F$ is a formally real.

The following is an immediate consequence of the  Skolem-Noether theorem.

For any ring $(A,\tau)$ with an anti-involution $\tau$ denote by $A^{[\tau]}$ the set of all $v\in A^\times$ such that $\tau(v)=v z$, for some $z\in Z(A)^\times$ such that $\tau(z)=z^{-1}$ (In particular, if each element of the center $Z(A)$ is fixed under $\tau$, then $\tau(v)\in \{v,-v\}$, i.e., ). Clearly, the assignments $(w,v)\mapsto w\act v:=w v\tau(w)$ define an action of the multiplicative group $A^\times$ on $A^{[\tau]}$.

\begin{lemma} Let $(A,\tau)$ be a simple Artinian ring. Then

(a) For any anti-involution $\sigma:A\to A$
there is an element $v\in A^{[\tau]}$ (unique up to multiplication by elements of the field $Z(A)$)  such that
$\sigma(a)=v\tau(a)v^{-1}$
for all $a\in A$.

(b) For any $v\in A^{[\tau]}$ the assignments $a\mapsto v\tau(a)v^{-1}$ define an involution $\sigma_v$ on $A$.

(c) For any $w\in A^\times$ the assignments $a\mapsto waw^{-1}$ define an isomorphism  of rings with anti-involution $(A,\sigma_v)\widetilde \to (A,\sigma_{w\act v})$.

\end{lemma}

\begin{definition} We say that an algebra $(A,\sigma)$ over $F$ is {\it thin} if $A^\sigma=F$.

\end{definition}

The first example of a thin  algebra is given by Example \ref{ex:exterior thin}. Another one is a (generalized) quaternion algebra (see Remark \ref{rk:quat_ext} below).

Clearly every thin algebra $(A,\sigma)$ over $F$ is pre-Hermitian and the direct sum  of two $F$-algebras is never thin because $(A\oplus A')^{\sigma\oplus \sigma'}=A^\sigma\oplus {A'}^{\sigma'}$. Moreover, the following is an immediate consequence of Theorem \ref{th:prehermitian classification}.

\begin{lemma}
\label{le:thin semisimple} If a finite-dimensional algebra  $(A,\sigma)$ over a perfect field $F$ with $char(F)\ne 2$ is thin then it is either simple noncommutative, or direct sum of two copes of a simple algebra interchanged by $\sigma$, or pre-Hermitian whose semisimple quotient $A/J(A)$ is a field extension of $F$.

\end{lemma}

\begin{proposition}
\label{pr:thin}
Let $F$ be a perfect field with $char(F)\ne 2$ and $(A,\sigma)$ be a thin semisimple finite-dimensional  algebra over  $F$. Then:

(a) $A=F\oplus A^{-\sigma}$ where $A^{-\sigma}=\{a\in A:\sigma(a)=-a\}$ is the space of skew-symmetric elements.

(b) $A^{-\sigma}$ admits a unique nonsingular symmetric bilinear form $\beta$  such that $aa'+a'a=-\beta(a,a')$ for all $a\in A^{-\sigma}$.

(c) If $a$ is a nilpotent element in $A$ then $a\in A^{-\sigma}$ and $a^2=0$.

(d) $A^{-\sigma}$ is a Lie algebra with respect to the commutator bracket $[a,a']=aa'-a'a=2aa'+\beta(a,a')$

(e) $2a''a'a-2aa'a''=\beta([a,a'],a'')=\beta(a,[a',a''])$  
and
$\beta([a,a'],[a,a'])=\beta(a,a)\beta(a',a')-\beta(a,a')^2$
for $a,a',a''\in A^{-\sigma}$.

\end{proposition}

\begin{proof} Part (a) is obvious because $A^\sigma=F$. To prove (b) note that $\hat \beta(a,a'):=-aa'-a'a$ is fixed by $\sigma$ and thus belongs to $F$ for all $a,a'\in A$. This is obviously a symmetric bilinear form on $A$, denote by $\beta$ its restriction to $A^{-\sigma}$. If $z$ is in the radical of $\beta$, then $za+az=0$ for all $a\in A^{-\sigma}$, which implies that $z=0$ by Lemma  \ref{le:thin semisimple}.  This proves (b).

Prove (c). Let $a\in A$ be a nilpotent and suppose that $a\notin A^{-\sigma}$. Then there exists $c\in F^\times$ such that the nilpotent element $b:=ca$ satisfies $\sigma(b)=1-b$. Therefore, $(1-b)^n=0$ for some $n\ge 2$. This contradicts to the non-invertibility of $b$. Thus, $\sigma(a)=-a$ and $a^2$ is a nilpotent in $A^\sigma=F$, that is, $a^2=0$.

Part (d) is obvious because $\sigma([a,a'])=-[a,a']$ for all $a,a'\in A^\sigma$.

Prove (e). Indeed, 
$$\beta([a,a'],a'')=-[a,a']a''-a''[a,a']=-(2aa'+\beta(a,a'))a''-a''(-2a'a-\beta(a',a))$$
$$=2a''a'a-2aa'a''=-\beta([a'',a'],a)$$
for all $a,a',a''\in A^\sigma$. This proves the first assertion. To prove the second one, compute
$$\beta([a,a'],[a,a'])=2(aa'-a'a)a'a+2aa'(a'a-aa')$$
$$=-2(aa'+a'a)a'a-2aa'(a'a+aa')+4a{a'}^2a+4a'a^2a'= $$
$$=2\beta(a,a')(aa'+a'a)-2a^2\beta(a',a')-2{a'}^2\beta(a,a)=\beta(a,a)\beta(a',a') -\beta(a,a')^2$$
for all $a,a'\in A^{-\sigma}$. This proves (d).

The proposition is proved.
\end{proof}

\begin{teo}\label{th:thin}
Let $F$ be a perfect field with $char(F)\ne 2$ and $(A,\sigma)$ be a thin semisimple finite-dimensional  algebra over  $F$. Then $A$ is either  a division algebra over $F$ with  $\dim_F A\in \{1,2,4\}$ or 
   $A=F\oplus F$  (so that $\sigma$ is the permutation of summands).
\end{teo}

\begin{proof} Indeed, any simple component of $A$ is $\Mat_n(D)$, where $D$ is a division algebra over $F$. Note that if $n\ge 3$, then $A$ contains an element $e=e_{12}+e_{23}$ which contradicts to  Proposition \ref{pr:thin}(c). Thus, $n\le 2$ and
$A\in\{D,D\oplus D,\Mat_2(D),\Mat_2(D)\oplus \Mat_2(D)\}$ by Lemma \ref{le:thin semisimple}.
Note, however, that the last two algebras are not thin. If $A=D\oplus D$, then $(x,\sigma(x))\in A^\sigma$ for all $x\in D$, i.e. we always have a copy of $D$ in $A^\sigma=F$, that is, $D=F$.

The theorem is proved.
\end{proof}

For any ring $F$ and $\alpha,\beta\in F$ denote by $\HH_{\alpha,\beta}$ the $F$-algebra with a presentation:
$$\HH_{\alpha, \beta}=\langle i,j\ |\ i^2=\alpha,~j^2=\beta,~ji=-ij\rangle\ .$$
Clearly, it admits a unique anti-involution $\overline{\cdot}$ such that $\overline i=-i$, $\overline j=-j$. By  construction,
$a\overline a=\overline a a\in F$
for all $a\in \HH_{\alpha, \beta}$. In particular, $\HH_{\alpha, \beta}$ is a division algebra if and only if $F$ is a field and  $a \overline  a \in F^\times$ for all nonzero $a\in \HH_{\alpha, \beta}$.

\begin{remark}\label{rk:quat_ext}
Any 4-dimensional division algebra $D$ over a field $F$ with $\ch F\ne 2$ is isomorphic to $\HH_{\alpha,\beta}$ for some non-squares $\alpha,\beta\in F^\times$.
Clearly, it is thin with the above anti-involution $\overline {\cdot}$. Note that if  both $-\alpha$ and $-\beta$ are complete squares in $F$, then $\HH_{\alpha,\beta} \cong \HH_{-1,-1}$ is the ordinary algebra of quaternions.

\end{remark}

\begin{remark} Suppose that $\ch F\ne 2$ and the algebraic closure $\overline F$ of $F$ is  a quadratic extension of $F$. Then any division algebra over $F$ is isomorphic to  either $\overline F$ or to some $\HH_{\alpha,\beta}$.
\end{remark}

Generalizing the above observations, we construct some twisted group algebras of abelian groups with anti-involutions. Recall that, given a group $G$ and its linear action $\triangleright$ on a commutative ring $K$, a map $\chi:G\times G\to K$ is called a $2$-cocycle on $G$ (in $K$) if
$$(g\triangleright \chi(g',g''))\chi(g,g'g'')=\chi(g,g')\chi(gg',g'')$$
for all $g,g',g''\in G$.

Furthermore, denote by $K_\chi G$ the $\chi$-twisted group algebra of $G$, i.e., a $K$-algebra with a free $K$-basis $\{[g], g\in G\}$ and the multiplication table:
$$[g][g']=\chi(g,g')\cdot [gg'], [g]k=g(k)\cdot g$$
for all $g,g'\in G$, $k\in K$. It is well-known that any central simple (e.g., a division) algebra $A$ over a field $F=K^G$  is isomorphic to some $K_\chi G$ so that $K$ is a Galois field extension of $F$.

The following is immediate.

\begin{lemma} 
\label{le:twisted  group algebra}
Let $\overline{\cdot}$ be an involution on $K$. Then the following statements are equivalent for a $2$-cocycle $\chi:G\times G\to K$:

(a) The assignments $k\cdot [g]\mapsto [\sigma(g)]\cdot \overline k$ for $g\in G$, $k\in K$ define an anti-involution on $K_\chi G$.

(b)  $\sigma(g)\triangleright \overline k=\overline{g^{-1}\act  k}$ and $\overline {\chi(g,g')}=\overline{gg'}\triangleright \chi(\sigma(g'),\sigma(g))$ for all $g,g'\in G$, $k\in K$.

\end{lemma}

In particular, if the $G$-action   $\triangleright$ is trivial, the cocycle conditions simplifies:
$$\chi(gg',g'')\chi(g,g')=\chi(g,g'g'')\chi(g',g'')$$
for all $g,g',g''\in G$ (e.g., $\chi$ is a bicharacter of $G$).

The following is immediate.

\begin{lemma} 
\label{le:twisted abelian group algebra} 
Let $G$ be an abelian group, trivially acting on $K$, $\overline{\cdot}$ be an anti-involution on $K$ and $\varepsilon:G\to K^\times$, $g\mapsto \varepsilon_g$ be a map. Then the following are equivalent:

(a) The assignments $k\cdot [g]\mapsto \overline k\varepsilon_g\cdot [g]$ for $g\in G$, $k\in K$ define an anti-involution $\sigma_\varepsilon$ on $K_\chi G$.

(b)  $\overline{\varepsilon_g}=\varepsilon_g^{-1}$ and $\chi(g',g)=\overline{\chi(g,g')}\varepsilon_{gg'}\varepsilon_g^{-1}\varepsilon_{g'}^{-1}$ for all $g,g'\in G$.

\end{lemma}

In particular, $(K_\chi G)^{\sigma_\varepsilon}=\bigoplus\limits_{g\in G}K_g\cdot [g]$,
where $K_g=\{k\in K:\overline k=k\cdot \varepsilon_g\}$ for all $g\in G$.

Now we classify  anti-involutions on $4$-dimensional division algebras.

\begin{proposition} 
\label{pr:anti-involutions on H}
In the notation as above, suppose that   $\HH_{\alpha,\beta}$ is a noncommutative division algebra (over a field $F$). Then any anti-involution on $\HH_{\alpha,\beta}$ is either $\overline{\cdot}$ or is given by $$\sigma(a)= v\overline xv^{-1}$$
for some nonzero  imaginary $v$, i.e., such that $\overline v=-v$ (in what follows, we denote the latter anti-involution  by $\sigma_v$).
\end{proposition}



\begin{proof} By Skolem-Noether theorem, any anti-involution $\sigma$ on $\HH_{\alpha,\beta}$ can be written as
$\sigma (x)=v\overline xv^{-1}$ where $v\in \HH_{\alpha,\beta}^{\times}$. Then $\sigma^2(a)=uau^{-1}$ where
$u=v\overline v^{-1}$. Since $\sigma^2=1$, we have $a=uau^{-1}$ for all $a\in $, i.e. $u$ is central in $\HH_{\alpha,\beta}$. That is $u=c\in F^\times$.  This implies that $\overline v=c^{-1} v$ hence either $c=1$, $v\in F^\times$ or $c=-1$, $\overline v=-v$.
\end{proof}




\begin{proposition} 
\label{pr:imaginary symmetric}
In the assumptions of Proposition \ref{pr:anti-involutions on H}, for any nonzero imaginary $v\in \HH_{\alpha,\beta}$ there exist  imaginary  $w,w'$ such that  $[w,w']=v$  and $\{1,w,w',v\}$ is a basis of $\HH_{\alpha,\beta}$. 
In particular, $vw=-wv$, $vw'=-w'v$, $ww'+w'w\in F$, $(\HH_{\alpha,\beta})^{\sigma_v}=F+Fw+Fw'$ and $(\HH_{\alpha,\beta})^{-\sigma_v}=F\cdot v$.
\end{proposition}

\begin{proof} Denote by $V$ the set of all imaginary elements of $\HH_{\alpha,\beta}$.  
For an imaginary $v$ define an $F$-linear map  $f_v:V\to \HH_{\alpha,\beta}$   by $f_v(a):= va+av$.
Clearly, the range of $f_v$ is $F$, in particular, $Ker~f_v$ is a $2$-dimensional  subspace of the (3-dimensional) space $V$. 
  Then 
$$\sigma_v(a)=v\overline av^{-1}=-vav^{-1}=avv^{-1}=a$$
for any  $a\in Ker~f_v$.
That is, $F+Ker~f_v\subset (\HH_{\alpha,\beta})^{\sigma_v}$. Since $v\in (\HH_{\alpha,\beta})^{-\sigma_v}$, we see that $F+Ker~f_v= (\HH_{\alpha,\beta})^{\sigma_v}$ and $(\HH_{\alpha,\beta})^{-\sigma_v}=Fv$.

Therefore, $[w,w']\in F^\times v$ for any   basis $\{w,w'\}$ of $Ker~f_v$. 

The proposition is proved.
\end{proof}

Note that if $\HH=\HH_{-1,-1}$ is the ordinary quaternion algebra over $\RR$, then $w^2=-w\overline w<0$ for any imaginary $w\in \HH$, hence $(\sqrt{-w^2})^2+w^2=0$ and we obtain the following immediate corollary of Proposition \ref{pr:imaginary symmetric}.

\begin{corollary} The algebra $(\HH, \sigma)$ is  Hermitian if and only if $\sigma=\overline{\cdot}$.
\end{corollary}

We conclude the section with recalling the Cayley-Dickson construction of (generalized) octonions ${\mathbb O}_{\alpha,\beta}:=\HH_{\alpha,\beta}\oplus \HH_{\alpha,\beta}$, as a free module over a ring $F$ with the following multiplication table
$$(a,b)(a',b')=(aa'-\overline {b'}b,b'a+b\,\overline{a'})\ .$$

This is a non-associative $F$-algebra  with the anti-involution given by
$$\overline {(a,b)}=(\overline  a,-b)$$
so that
$\overline {(a,b)}(a,b)=(a,b)\overline {(a,b)}=(a\overline a+\overline  bb,0)$ for all $a,b\in {\mathbb O}_{\alpha,\beta}$.
In particular,
if $\HH_{\alpha,\beta}$ is a division algebra, then ${\mathbb O}_{\alpha,\beta}$ is a non-associative division algebra as well.

\section{Isomorphisms and embeddings of matrix algebras}

Some tensor products and direct products of matrix algebras are related in a way we want to discuss in this section. The described in here isomorphisms and embeddings are used in the main part of the paper as a tool to construct symmetric spaces and study their properties.

As usual, for every algebra $A$ and (anti-)involution $\sigma$, we
denote by $A^\sigma$ the set of fixed points of $\sigma$ in $A$.

We also remind our standard notation of anti-involutions on matrix algebras. 
\begin{enumerate}
    \item If $A$ is $\Mat(n,\R)$ or $\Mat(n,\CC)$, then the anti-involution given by the transposition of the matrix is denoted by $\sigma$.
    \item If $A$ is $\Mat(n,\CC)$, then by $\bar\sigma$ is denoted the anti-involution given by $\sigma$ composed with the complex conjugation.
    \item If $A$ is $\Mat(n,\HH\{i,j,k\})$, then the anti-involution $\sigma_0\colon A\to A$ is given by the rule $\sigma_0(r_1+r_2j)=\sigma(r_1)+\bar\sigma(r_2)j$ where $r_1,r_2\in\Mat(n,\CC\{i\})$. Another anti-involution $\sigma_1\colon A\to A$ is given by the rule $\sigma_1(r_1+r_2j)=\bar\sigma(r_1)-\sigma(r_2)j$ where $r_1,r_2\in\Mat(n,\CC\{i\})$.
\end{enumerate}
We also identify the algebra of $1\times 1$ matrices with the corresponding (skew-)field and use the same notation for the (anti-)involution in a (skew-)field as in the matrix algebra over this (skew-)field.

\subsection{Three isomorphisms of matrix algebras}\label{3_isom}

In this section, we describe three well-known matrix algebras isomorphisms. 

\subsubsection{\texorpdfstring{$\Mat(n,\CC)\otimes_\R\CC$}{} and \texorpdfstring{$\Mat(n,\CC)\times\Mat(n,\CC)$}{}}\label{Isom_chi}

\begin{fact} The following map is an isomorphism of $\CC\{i\}$-algebras:
$$\begin{matrix}
\chi\colon & \Mat(n,\CC\{I\})\otimes_\R\CC\{i\} & \to & \Mat(n,\CC\{i\})\times\Mat(n,\CC\{i\})\\ & a+bI &
\mapsto & (a+bi , a-bi)
\end{matrix}$$
where $a,b\in\Mat(n,\CC\{i\})$. In particular,
$$\chi(\Id I\otimes 1)=(i,-i),\;\chi(\Id\otimes i)=(i,i).$$
\end{fact}

The induced by $\sigma\otimes\Id$ anti-involution
$$\chi\circ(\sigma\otimes\Id)\circ\chi^{-1}$$
on $\Mat(n,\CC\{i\})\times\Mat(n,\CC\{i\})$ acts in the following way:
$$(m_1,m_2)\mapsto(m^T_1,m^T_2).$$

The induced by $\bar\sigma\otimes\Id$ anti-involution
$$\chi\circ(\bar\sigma\otimes\Id)\circ\chi^{-1}$$ on $\Mat(n,\CC\{i\})\times\Mat(n,\CC\{i\})$ acts in
the following way:
$$(m_1,m_2)\mapsto(m^T_2,m^T_1).$$

The induced by $\Id\otimes\bar\sigma$ involution
$$\chi\circ(\Id\otimes\bar\sigma)\circ\chi^{-1}$$
on $\Mat(n,\CC\{i\})\times\Mat(n,\CC\{i\})$ acts in the following way:
$$(m_1,m_2)\mapsto(\bar m_2,\bar m_1).$$

Therefore:
$$\chi((\Mat(n,\CC\{I\})\otimes_\R\CC\{i\})^{\bar\sigma\otimes\Id})=\{(m,m^T)\mid m\in \Mat(n,\CC\{i\})\},$$

$$\chi((\Mat(n,\CC\{I\})\otimes_\R\CC\{i\})^{\bar\sigma\otimes\bar\sigma})=\Herm(n,\CC\{i\})\times\Herm(n,\CC\{i\}),$$

$$\chi(\Mat(n,\CC\{I\}))=\chi((\Mat(n,\CC\{I\})\otimes_\R\CC\{i\})^{\Id\otimes\bar\sigma})=\{(m,\bar m)\mid m\in \Mat(n,\CC\{i\})\}.$$

\subsubsection{\texorpdfstring{$\Mat(n,\HH)\otimes_\R\CC$}{} and \texorpdfstring{$\Mat(2n,\CC)$}{}}\label{Isom_psi}

\begin{fact} The following map is an isomorphism of $\CC\{I\}$-$\CC\{i\}$-algebras:
$$\begin{matrix}
\psi\colon & \Mat(n,\HH\{i,j,k\})\otimes_\R\CC\{I\} & \to & \Mat(2n,\CC\{i\})\\
 & (q_1+q_2j) + (p_1+p_2j)I & \mapsto &
\begin{pmatrix}
q_1+p_1i & q_2+p_2i\\ -\bar q_2-\bar p_2i & \bar q_1+\bar p_1i
\end{pmatrix}.
\end{matrix}$$
where $q_1,q_2,p_1,p_2\in\Mat(n,\CC\{i\})$. In particular,
$$\chi(\Id i\otimes 1)=\begin{pmatrix}
\Id i & 0\\ 0 & -\Id i
\end{pmatrix},\;\chi(\Id\otimes j)=\begin{pmatrix}
0 & \Id\\ -\Id & 0
\end{pmatrix},$$
$$\chi(\Id k\otimes 1)=\begin{pmatrix}
0& \Id i\\ \Id i & 0
\end{pmatrix},\;\chi(\Id\otimes I)=\Id i.$$
\end{fact}

The induced by $\sigma_0\otimes\Id$ anti-involution
$$\psi\circ(\sigma_0\otimes\Id)\circ\psi^{-1}$$
on $\Mat(2n,\CC)$ acts in the following way:
$$m\mapsto \begin{pmatrix}
\Id & 0\\ 0 & -\Id
\end{pmatrix}m^T\begin{pmatrix}
\Id & 0\\ 0 & -\Id
\end{pmatrix}.$$

The induced by $\sigma_1\otimes\Id$ anti-involution $$\psi\circ(\sigma_1\otimes\Id)\circ\psi^{-1}$$
on $\Mat(2n,\CC)$ acts in the following way:
$$m\mapsto -\Ome{\Id}m^T\Ome{\Id}=\Ome{i}m^T\Ome{i}.$$

The induced by $\Id\otimes\bar\sigma$ involution $$\psi\circ(\Id\otimes\bar\sigma)\circ\psi^{-1}$$
on $\Mat(2n,\CC)$ acts in the following way:
$$m\mapsto -\Ome{\Id}\bar m\Ome{\Id}=\Ome{i}\bar m\Ome{i}.$$

The induced by $\sigma_0\otimes\bar\sigma$ anti-involution
$$\psi\circ(\sigma_0\otimes\bar\sigma)\circ\psi^{-1}$$ on $\Mat(2n,\CC)$ acts in the following way:
$$m\mapsto \begin{pmatrix}
0 & \Id\\ \Id & 0
\end{pmatrix}\bar m^T\begin{pmatrix}
0 & \Id\\ \Id & 0
\end{pmatrix}.$$

Therefore:
$$\psi((\Mat(n,\HH\{i,j,k\})\otimes_\R\CC\{I\})^{\sigma_1\otimes\Id})=$$
$$=\left\{m\in\Mat(2n,\CC\{i\})\midwd m=-\Ome{\Id}m^T\Ome{\Id}\right\}=$$
$$=\oo\Ome{\Id}=\spp(2n,\CC),$$

$$\psi((\Mat(n,\HH\{i,j,k\})\otimes_\R\CC\{I\})^{\sigma_0\otimes\bar\sigma})=$$
$$=\left\{m\in\Mat(2n,\CC\{i\})\midwd m=\begin{pmatrix}
0 & \Id\\ \Id & 0
\end{pmatrix}\bar m^T\begin{pmatrix}
0 & \Id\\ \Id & 0\end{pmatrix}\right\},$$

$$\psi((\Mat(n,\HH\{i,j,k\})\otimes_\R\CC\{I\})^{\sigma_1\otimes\bar\sigma})=\Herm(2n,\CC),$$

$$\psi(\Mat(n,\HH\{i,j,k\}))=\psi((\Mat(n,\HH\{i,j,k\})\otimes_\R\CC\{I\})^{\Id\otimes\bar\sigma})=$$
$$=\left\{m\in\Mat(2n,\CC\{i\})\midwd m=-\Ome{\Id}\bar m\Ome{\Id}\right\}=$$
$$=\left\{\begin{pmatrix}
q_1 & q_2\\ -\bar q_2 & \bar q_1
\end{pmatrix}\midwd q_1,q_2\in\Mat(n,\CC\{i\})\right\}.$$

\subsubsection{\texorpdfstring{$\Mat(n,\HH)\otimes_\R\HH$}{} and \texorpdfstring{$\Mat(4n,\R)$}{}}\label{Isom_phi}

\begin{fact}
The following map:
$$\phi\colon \Mat(n,\HH\{I,J,K\})\otimes_\R\HH\{i,j,k\}\to \Mat(4n,\R)$$
defined on generators of $A_\HH$ as follows:
$$\phi(a\otimes i)=\begin{pmatrix}
0 & a & 0 & 0 \\ -a & 0 & 0 & 0 \\ 0 & 0 & 0 & -a \\ 0 & 0 & a & 0
\end{pmatrix},\;
\phi(a\otimes j)=\begin{pmatrix} 0 & 0 & a & 0 \\ 0 & 0 & 0 & a \\ -a & 0 & 0 & 0 \\ 0 & -a & 0 & 0
\end{pmatrix},$$
$$\phi(a\otimes k)=\begin{pmatrix}
0 & 0 & 0 & a \\ 0 & 0 & -a & 0 \\ 0 & a & 0 & 0 \\ -a & 0 & 0 & 0
\end{pmatrix},\;
\phi(aI\otimes 1)=\begin{pmatrix} 0 & -a & 0 & 0 \\ a & 0 & 0 & 0 \\ 0 & 0 & 0 & -a \\ 0 & 0 & a &
0
\end{pmatrix},$$
$$\phi(aJ\otimes 1)=\begin{pmatrix}
0 & 0 & -a & 0 \\ 0 & 0 & 0 & a \\ a & 0 & 0 & 0 \\ 0 & -a & 0 & 0
\end{pmatrix},\;
\phi(aK\otimes 1)=\begin{pmatrix} 0 & 0 & 0 & -a \\ 0 & 0 & -a & 0 \\ 0 & a & 0 & 0 \\ a & 0 & 0 &
0
\end{pmatrix}$$
where $a\in\Mat(n,\R)$ is an $\R$-algebra isomorphism.
\end{fact}

The anti-involution $\sigma_1\otimes\sigma_0$ corresponds under $\phi$ to the following
anti-involution
$$\phi\circ(\sigma_1\otimes\sigma_0)\circ\phi$$
on $\Mat(4n,\R)$: $m\mapsto -\Xi m^T \Xi$ where
$$\Xi:=\begin{pmatrix}
0 & 0 & 0 & \Id_n \\ 0 & 0 & -\Id_n & 0 \\ 0 & \Id_n & 0 & 0 \\ -\Id_n & 0 & 0 & 0
\end{pmatrix}.$$

The anti-involution $\sigma_1\otimes\sigma_1$ corresponds under $\phi$ to the transposition on
$\Mat(4n,\R)$.

Therefore:
$$\phi((\Mat(n,\HH\{I,J,K\})\otimes_\R\HH\{i,j,k\})^{\sigma_1\otimes\sigma_0})=\left\{m\in\Mat(4n,\R)\mid m=-\Xi m^T \Xi\right\}=$$
$$=\oo(\Xi)\cong\spp(4n,\R),$$
$$\phi((\Mat(n,\HH\{i,j,k\})\otimes_\R\HH\{i,j,k\})^{\sigma_1\otimes\sigma_1})=\Sym(4n,\R),$$

The real locus $\Mat(n,\HH\{I,J,K\})$ of $\Mat(n,\HH\{I,J,K\})\otimes_\R\HH\{i,j,k\}$ is mapped by $\phi$
to:
$$\phi(\Mat(n,\HH\{I,J,K\}))=$$
$$=\left\{\begin{pmatrix}
a & -b & -c & -d \\ b & a & -d & c \\ c & d & a & -b \\ d & -c & b & a
\end{pmatrix}\midwd a,b,c,d\in\Mat(n,\R)\right\}.$$

\subsection{Embeddings between matrix algebras}

In this section, we consider the following two embeddings:
$$\Mat(n,\CC\{I\})\otimes\CC\{j\}\hookrightarrow \Mat(n,\CC\{I\})\otimes\HH\{i,j,k\},$$
$$\Mat(n,\HH\{I,J,K\})\otimes\CC\{j\}\hookrightarrow \Mat(n,\HH\{I,J,K\})\otimes\HH\{i,j,k\}.$$
We use these embedding to see the symmetric space for a real group inside the symmetric space for a complexified group.

\subsubsection{Embedding \texorpdfstring{$\Mat(n,\CC\{I\})\otimes\CC\{j\}\hookrightarrow \Mat(n,\CC\{I\})\otimes\HH\{i,j,k\}$}{}}\label{Embedd1}

In the previous sections, we have seen the isomorphisms:
$$\begin{matrix}
\chi\colon & \Mat(n,\CC\{I\})\otimes_\R\CC\{j\} & \to & \Mat(n,\CC\{j\})\times\Mat(n,\CC\{j\})\\ & a+bI &
\mapsto & (a+bj , a-bj)
\end{matrix}$$
where $a,b\in\Mat(n,\CC\{j\})$ and
$$\begin{matrix}
\psi\colon & \Mat(n,\CC\{I\})\otimes_\R\HH\{i,j,k\} & \to & \Mat(2n,\CC\{i\})\\
 & (q_1+q_2j) + (p_1+p_2j)I & \mapsto &
\begin{pmatrix}
q_1+p_1i & q_2+p_2i\\ -\bar q_2-\bar p_2i & \bar q_1+\bar p_1i
\end{pmatrix}.
\end{matrix}$$
where $q_1,q_2,p_1,p_2\in\Mat(n,\CC\{i\})$.

We want to describe the map
$$\psi\circ\iota\circ\chi^{-1}\colon \Mat(n,\CC\{j\})\times\Mat(n,\CC\{j\})\to \Mat(2n,\CC\{i\})$$
where
$$\iota\colon\Mat(n,\CC\{I\})\otimes\CC\{j\}\hookrightarrow \Mat(n,\CC\{I\})\otimes\HH\{i,j,k\}$$
is the natural embedding. Let $(a,b):=(a_1+a_2j,b_1+b_2j)\in\Mat(n,\CC\{j\})\times\Mat(n,\CC\{j\})$ for
$a_1,a_2,b_1,b_2\in\Mat(n,\R)$, then
$$\chi^{-1}(a,b)=\frac{a+b}{2}+\frac{a-b}{2j}I=\frac{a_1+b_1+(a_2+b_2)j}{2}+\frac{a_2-b_2-(a_1-b_1)j}{2}I.$$
Therefore,
$$\psi(\chi^{-1}(a,b))=\frac{1}{2}\begin{pmatrix}
a_1+b_1+(a_2-b_2)i & a_2+b_2-(a_1-b_1)i\\ -(a_2+b_2)+(a_1-b_1)i & a_1+b_1+(a_2-b_2)i
\end{pmatrix}.$$
We also describe the image of the map $\psi\circ\iota\circ\chi^{-1}$:
$$\Imm(\psi\circ\iota\circ\chi^{-1})=\left\{\begin{pmatrix}
q & p\\ -p & q
\end{pmatrix}\midwd p,q\in \Mat(n,\CC\{i\})\right\}=$$
$$=\left\{m\in\Mat(2n,\CC\{i\})\midwd m=-\Ome{\Id}m\Ome{\Id}\right\}.$$

\subsubsection{Embedding \texorpdfstring{$\Mat(n,\HH\{I,J,K\})\otimes\CC\{j\}\hookrightarrow \Mat(n,\HH\{I,J,K\})\otimes\HH\{i,j,k\}$}{}}\label{Embedd2}

In the previous sections, we have seen the isomorphisms:
$$\begin{matrix}
\psi\colon & \Mat(n,\HH\{I,J,K\})\otimes_\R\CC\{j\} & \to & \Mat(2n,\CC\{I\})\\
 & (q_1+q_2J) + (p_1+p_2J)j & \mapsto &
\begin{pmatrix}
q_1+p_1I & q_2+p_2I\\ -\bar q_2-\bar p_2I & \bar q_1+\bar p_1I
\end{pmatrix}.
\end{matrix}$$
where $q_1,q_2,p_1,p_2\in\Mat(n,\CC\{I\})$ and
$$\phi\colon \Mat(n,\HH\{I,J,K\})\otimes_\R\HH\{i,j,k\}\to \Mat(4n,\R)$$
defined as in the Section~\ref{Isom_phi}.

We want to describe the image of the map
$$\phi\circ\iota\circ\psi^{-1}\colon \Mat(2n,\CC\{I\})\to\Mat(4n,\R)$$ where
$$\iota\colon\Mat(n,\HH\{I,J,K\})\otimes\CC\{j\}\hookrightarrow \Mat(n,\HH\{I,J,K\})\otimes\HH\{i,j,k\},$$
is the natural embedding.

Note that an element $x\in
\Mat(n,\HH\{I,J,K\})\otimes\HH\{i,j,k\}$ is contained in the subalgebra $\Mat(n,\HH\{I,J,K\})\otimes\CC\{j\}$ if and only if $x$
commutes with $1\otimes j$. So we obtain:
$$\Imm(\psi\circ\iota\circ\chi^{-1})=\left\{m\in\Mat(4n,\R)\mid m=
-\phi(\Id_n\otimes j)m\phi(\Id_n\otimes j) \right\}.$$